\documentclass{amsart}
\usepackage{amssymb}
\usepackage{amsthm}
\usepackage{tikz}
\usepackage{color}
\usepackage{hyperref}
\newtheorem{theorem}{Theorem}[section]
\newtheorem{remark}[theorem]{Remark}
\newtheorem{question}[theorem]{Question}
\newtheorem{lemma}[theorem]{Lemma}

\newtheorem{proposition}[theorem]{Proposition}
\newtheorem{corollary}[theorem]{Corollary}
\usepackage{tikz-cd}

\begin{document}
\title{Arithmetic of function field  units}

\author{Bruno Angl\`es   \and Floric Tavares Ribeiro}

\address{
Universit\'e de Caen Basse-Normandie, 
Laboratoire de Math\'ematiques Nicolas Oresme, 
CNRS UMR 6139, 
Campus II, Boulevard Mar\'echal Juin, 
B.P. 5186, 
14032 Caen Cedex, France.
}
\email{bruno.angles@unicaen.fr, floric.tavares-ribeiro@unicaen.fr}

\begin{abstract} We prove a "discrete analogue" for Taelman's class modules of certain Conjectures formulated by R. Greenberg for cyclotomic fields. \end{abstract}
\date{ \today}
\maketitle
\tableofcontents
\section{Introduction}${}$\par
 Let us  recall some facts from classical cyclotomic theory (we refer the reader to \cite{WAS}, chapters 13 and 15). Let $p\geq 3$ be an odd prime number and let $\Delta= {\rm Gal}(\mathbb Q(\mu_p)/\mathbb Q).$ Let $M$ be the projective limit for the norm map  of the $p$-Sylow subgroups of the ideal class groups along the cyclotomic $\mathbb Z_p$-extension of $\mathbb Q(\mu_p): \mathbb Q(\mu_{p^{\infty}}).$ Let $\Gamma={\rm Gal}( \mathbb Q(\mu_{p^{\infty}})/\mathbb Q(\mu_p)),$ and let $\gamma \in \Gamma$ such that:
 $$\forall \zeta \in \mu_{p^{\infty}},\, \gamma (\zeta)=\zeta^{1+p}.$$
 We can identify $\mathbb Z_p[[\Gamma]]$ with $\mathbb Z_p[[T]]$ by sending $\gamma$ to $1+T.$ Therefore $M$ is a $\mathbb Z_p[[T]][\Delta]$-module. Let $\chi \in \widehat {\Delta}:= {\rm Hom}(\Delta, \mathbb Z_p^\times),$ and let $e_{\chi}=\frac{1}{p-1}\sum_{\delta \in \Delta} \chi(\delta) \delta^{-1}\in \mathbb Z_p[\Delta].$ For $\chi \in \widehat{\Delta}, $ we set:
 $$M_{\chi}=e_{\chi}M.$$
 Then $M_{\chi}$ is a finitely generated and torsion $\mathbb Z_p[[T]]$-module. If $\chi$ is odd distinct from the Teichm\"uller character : $\omega_p,$ B. Mazur and A. Wiles proved that the characteristic ideal of the $\mathbb Z_p[[T]]$-module $M_{\chi}$ is generated by some polynomial $P_{\chi}(T)\in \mathbb Z_p[T]$ such that there exists $U_{\chi}(T)\in \mathbb Z_p[[T]]^\times$ with the property (\cite{MAZ}):
 $$\forall y \in \mathbb Z_p, \, P_{\chi}((1+p)^y-1) U_{\chi}((1+p)^y-1)=L_p(y, \omega_p\chi^{-1}),$$
 where $L_p(.,\omega_p\chi^{-1})$ is the $p$-adic $L$-function attached to the even character $\omega_p\chi^{-1}.$ Observe that if Vandiver's Conjecture is true for the prime $p,$ then for $\chi$ odd, $M_{\chi}$ is a cyclic $\mathbb Z_p[[T]]$-module and for $\chi$ even: $M_{\chi}=\{0\}.$ R. Greenberg has conjectured  weaker statements  (see \cite{GRE}):\par
 \noindent {\sl Pseudo-cyclicity Conjecture:} If $\chi$ is odd then there exists an injective  morphism of $\mathbb Z_p[[T]]$-modules between a cyclic $\mathbb Z_p[[T]]$-module and $M_{\chi}$  such that the cokernel of this morphism is finite;\par
 \noindent {\sl Pseudo-nullity Conjecture:} If $\chi$ is even then $M_{\chi}$ is finite.\par
 \noindent These two conjectures are  open.\par
 ${}$\par
 L. Taelman has recently introduced  new arithmetic objects associated to Drinfeld modules (\cite{TAE1}): class modules and modules of units; in the case of the Carlitz module, these latter objects have similar properties to that of the ideal class groups and units of number fields (see for example \cite{ANG&TAE1}). While  Greenberg's Conjectures are "vertical" problems (one fixes a cyclotomic field and study the structures of the $p$-class groups along the cyclotomic $\mathbb Z_p$-extension), in this article we consider analogues of these problems for Taelman's class modules  but in the "horizontal" case.\par
 Let $\mathbb F_q$ be a finite field having $q$ elements and let $\theta$ be an indeterminate over $\mathbb F_q.$ Let  $\chi:(\frac{A}{fA})^\times \rightarrow (K^{ac})^\times$ be a Dirichlet character of conductor $f$ (see section \ref{Dirichletcharacters}), where $A=\mathbb F_q[\theta],$ and $K^{ac}$ is a fixed algebraic closure of $K=\mathbb F_q(\theta).$  Let $L$ be the $f$th cyclotomic function field (see \cite{ROS}, chapter 12). Let $C$ be the Carlitz module (see \cite{GOS}, chapter 3, and also  section \ref{Anexample}). Let's consider the isotypic component of Taelman's class module (\cite{TAE1}, \cite{TAE2}) attached to the $f$th cyclotomic function field  :
 $$H_{\chi}:=e_\chi (\frac{L_{\infty}}{O_L+\exp_C(L_{\infty})}\otimes _{\mathbb F_q} \mathbb F_q(\chi)),$$
 where $\exp_C$ is the Carlitz exponential map (\cite{GOS}, chapter 3), $O_L$ is the integral closure of $A$ in $L,$  $\mathbb  F_q(\chi) \subset K^{ac}$ is the finite extension of $\mathbb F_q$ obtained by adjoining to $\mathbb F_q$ the values of $\chi,$ and $e_\chi \in \mathbb F_q(\chi)[{\rm Gal}(L/K)]$ is the usual idempotent associated to $\chi$ (note that $f$ is square free). Then L. Taelman proved that $H_\chi$ is a finite $A[\chi]:= \mathbb F_q(\chi)[\theta]$-module (see for example  \cite{TAE1}). In this article,  we study the $A[\chi]$-module structure of $H_\chi$ when $\chi$ runs through the infinite family of Dirichlet characters of a given type (see  \ref{Dirichletcharacters}).\par
 Let $n\geq 1$ be an integer and let $t_1, \ldots,t_n$ be $n$ variables over $ K_{\infty}:=\mathbb F_q((\frac{1}{\theta})).$ Let $\mathbb T_n(K_{\infty})$ be the Tate algebra in the variables $t_1, \ldots, t_n$ with coefficients in $K_{\infty}$ (see \cite{PUT}, chapter 3). Let $\tau: \mathbb T_n(K_{\infty})\rightarrow \mathbb T_n(K_{\infty})$ be the continuous morphism of $\mathbb F_q[t_1, \ldots, t_n]$-algebras such that $\forall x\in K_{\infty}, \tau (x)=x^q.$ Let $\phi: A\rightarrow A[t_1, \ldots, t_n]\{\tau\}$ be the morphism of $\mathbb F_q$-algebras such that:
 $$\phi_{\theta}= (t_1-\theta)\cdots (t_n-\theta)\tau +\theta.$$
 \noindent Let $\exp_{\phi}$ be the exponential function attached to $\phi,$ i.e. $\exp_{\phi}$ is the unique element in $\mathbb T_n(K_{\infty})\{ \{\tau\}\}$ such that $\exp_{\phi} \equiv 1\pmod{\tau}$ and:
 $$\exp_{\phi} \theta= \phi_{\theta} \exp_{\phi}.$$
 Then inspired by Taelman's work (\cite{TAE1}, \cite{TAE2}, \cite{TAE3}), one can introduce a "generic class module" of level $n$ (see \cite{APTR}):
 $$H_n:= \frac{\mathbb T_n(K_{\infty})                       }{ A[t_1, \ldots, t_n]+\exp_{\phi}(\mathbb T_n(K_{\infty}))                          }.$$
 Observe that $H_n$ is an $A[t_1, \ldots, t_n]$-module via $\phi,$ furthermore it is a finitely generated $\mathbb F_q[t_1, \ldots, t_n]$-module and  a torsion $A[t_1, \ldots,t_n]$-module (we refer the reader to section \ref{section2} for more details). The terminology "generic class module" comes from the fact that the evaluations of elements of $\mathbb T_n(K_{\infty})$ on elements of $\overline{\mathbb F_q}^n$ ($\overline{\mathbb F_q}$ being the algebraic closure of $\mathbb F_q$ in $K^{ac}$) induce surjective morphisms from $H_n$ to  isotypic components of Taelman's class modules associated to Dirichlet characters of type $n.$ These generic class modules can be viewed as discrete analogues of the Iwasawa modules discussed above, having in mind that the role of $\mathbb Z_p$ is played by $\mathbb F_q[t_1, \ldots, t_n],$ the role of $T$ is played by $\phi_{\theta}.$  For example Mazur-Wiles Theorem has an analogue in this situation, let $n\geq q, $ $n\equiv 1\pmod{q-1}$ (this is the analogue of the condition $\chi$ odd, $\chi \not =\omega_p,$ for Iwasawa modules), then there exists $\mathbb B(t_1, \ldots, t_n)\in A[t_1, \ldots, t_n]$ which is a monic polynomial in $\theta,$ such that the Fitting ideal of the $A[t_1, \ldots, t_n]$-module $H_n$ is generated by $\mathbb B(t_1, \ldots, t_n)$ and (\cite{APTR}, Theorem 7.7):
 $$\mathbb (-1)^{\frac{n-1}{q-1}}\mathbb B(t_1, \ldots, t_n)\frac{\widetilde{\pi}}{\omega(t_1)\cdots \omega(t_n)}= L(t_1, \ldots,t_n),$$
 where $L(t_1, \ldots, t_n)\in \mathbb T_n(K_{\infty})^\times$ is the $L$-series attached to $\phi/A[t_1, \ldots, t_n]$ (see section \ref{section2}),  and $\omega(t)$ is the Anderson-Thakur special function (see \cite{AND&THA}, and also \cite{APTR}). The reader can now easily guess what are the discrete analogues of Greenberg's Conjectures in our situation. We prove these discrete Greenberg's Conjectures in section \ref{Thecaseofthecarlitzmodule} which was left as open questions in \cite{APTR}, more precisely:\par
 \noindent {\sl Pseudo-cyclicity:} let $n\geq 2, $ $n\equiv 1\pmod{q-1},$ there exists an injective  morphism of $A[t_1, \ldots, t_n]$-modules between a cyclic $A[t_1, \ldots, t_n]$-module and $H_n$  such that the cokernel of this morphism is a finitely generated and torsion $\mathbb F_q[t_1, \ldots, t_n]$-module (Theorem \ref{TheoremS2-2});\par
 \noindent {\sl Pseudo-nullity:}   let $n\geq 2, $ $n\not \equiv 1\pmod{q-1},$ then $H_n$ is a finitely generated and torsion $\mathbb F_q[t_1, \ldots, t_n]$-module (Theorem \ref{TheoremS2-1}).\par
By \cite{APTR},  Theorem \ref{TheoremS2-1} and Theorem \ref{TheoremS2-2} imply  that there exists $F(t_1, \ldots,t_n)\in \mathbb F_q[t_1, \ldots, t_n]\setminus \{0\}$  (depending on the module structure of $H_n$) such that for every Dirichlet character $\chi$ of type $n$ verifying $F(\eta_1, \ldots , \eta_n)\not =0,$ then, if $n\equiv 1\pmod{q-1},$ $H_{\chi}$ is a cyclic $A[\chi]$-module  and if $n\not \equiv 1\pmod{q-1}$ we have $H_{\chi}=\{ 0\}$. 

${}$\par
The paper is organized as follows:  in section \ref{section1}, we introduce a natural sub-module of finite index in the module of Taelman's units associated to a Drinfeld module : the module of Stark units;  we prove several basic properties of this module and study its connection to $L$-series. In section \ref{section2}, we show how the constructions of section \ref{section1} can also be made in the context of deformation of Drinfeld modules over Tate algebras;  we then study in deep the arithmetic properties  of Stark units attached to deformations of the Carlitz module and their connections to $L$-series  leading to the proof of the discrete Greenberg Conjectures. In the last section, combining the ideas developed in section \ref{section1} and section \ref{section2}, we prove a cyclicity result involving the "derivatives" of  Dirichlet-Goss $L$-series.\par
${}$\par
The authors warmly thank David Goss for fruitful comments on an earlier version of this article. \par

\section{Stark Units}\label{section1}${}$\par
In this section, we  construct a natural submodule of the module of Taelman's units associated to a Drinfeld module defined over the ring of integers of a finite extension of $K.$ The ideas developed in this section will be used in section \ref{section2} to prove the discrete Greenberg Conjectures.\par

\subsection{Preliminaries}\label{Prel}${}$\par
Let $k$ be a field and let $\theta$ be an indeterminate over $k.$ We set $R=k[\theta]$ and $\mathbb K=k((\frac{1}{\theta})).$ We endow $\mathbb K$ with the $\frac 1 \theta$-adic topology. An element $x\in \mathbb K^\times$ is said to be monic if $x=\frac{1}{\theta^m}+\sum_{i>m}x_i\frac{1}{\theta^i}, $ $m\in \mathbb Z,$ $x_i\in k.$\par
Let $M$ be a finite dimensional $k$-vector space which is an $R$-module, then there exist $r_1, \ldots, r_n $ which are monic elements in $R$ such that we have an isomorphism of $R$-modules:
$$M\simeq \prod_{j=1}^n \frac{R}{r_jR}.$$
We set:
$$[M]_R= r_1\cdots r_n,$$
the monic generator of the Fitting ideal of $M$.
Observe that:
$$[M]_R=\det_{k[X]}((1\otimes X){\rm Id}-(\theta\otimes 1)\mid_{M\otimes_kk[X]})\mid_{X=\theta}.$$
Therefore, if $M_1, M, M_2$ are $R$-modules such that $M$ is a finite dimensional $k$-vector space, and  suppose that we have a short exact sequence of $R$-modules:
$$0\rightarrow M_1\rightarrow M\rightarrow M_2\rightarrow 0,$$
then:
$$[M]_R=[M_1]_R [M_2]_R.$$
${}$\par\par
Let $V$ be a finite dimensional $\mathbb K$-vector space. We recall the following basic fact:
\begin{lemma}\label{lemmaPrel1} Let $M$ be a sub $R$-module of $V$ such that $M$ is discrete in $V.$ Let $W$ be the  sub $\mathbb K$-vector space of $V$ generated by $M.$ Then $M$ is a finitely generated and free $R$-module of rank $\dim_{\mathbb K} W.$
\end{lemma}
\begin{proof} Set $n= \dim_{\mathbb K} W.$ Choose $\|\cdot\|$ a norm of $\mathbb K$-vector space on $W$. Since $W$ is closed in $V,$ we deduce that $M$ is discrete in $W.$  Thus there exists $m\in M\setminus\{0\}$ such that $\|m\|$ is minimal, so that $Rm=M\cap \mathbb K m$. In particular, when $\dim_{\mathbb K}W=1$, then $M=Rm$ is free of rank $1$. Let $M'$ be the image of $M$ in $W/\mathbb K m$. By induction, $M'$ is free of rank $n-1$. Since we have a short exact sequence of torsion-free $R$-modules
$$ 0 \to Rm \to M \to M' \to 0 $$
we deduce that  $M$ is a free $R$-module of rank $n$.
\end{proof}
An $R$-lattice in $V,$ $M,$ is a  sub $R$-module of $V$ such that $M$ is discrete in $V$ and $M$ contains a $\mathbb K$-basis of $V.$ In particular, by the above Lemma, $M$ is a free $R$-module of rank $\dim_{\mathbb K}V.$\par
Let $M_1$ and $M_2$ be two $R$-lattices in $V.$ Let $n=\dim _{\mathbb K} V.$ Following \cite{TAE2}, we will define $[M_1:M_2]_R\in \mathbb K^\times.$ Select $e_1, \ldots, e_n$ in $M_1$ and $f_1, \ldots, f_n\in M_2$ such that:
$$M_1= \oplus_{j=1}^n Re_j\, {\rm and}\, M_2=\oplus_{j=1}^n Rf_j.$$ 
Then, observe that:
$$M_1= \oplus_{j=1}^n \mathbb K e_j=\oplus_{j=1}^n \mathbb K f_j.$$
Let $f:V\rightarrow V$ be the $\mathbb K$-linear map such that $f(e_i)= f_i , i=1, \ldots, n.$  Let $B$ be the matrix of $f$ in the basis $(e_i)_{i=1, \ldots, n}.$ then $B\in GL_n(\mathbb K)$ and $(\det B)k^\times$ does not depend on the choices of $(e_i)_{i=1,\ldots, n}$ and $(f_i)_{i=1, \ldots n}.$ Let $x\in \mathbb  K^\times$ be the monic element such that $(\det B)k^\times = xk^\times,$ then we set:
$$[M_1:M_2]_R=x.$$
Observe that $[M_1:M_2]_R$ is well-defined.\par
If $M_2\subset M_1$ then $\frac{M_1}{M_2}$ is a finite dimensional $k$-vector space. Since $R$ is a principal ideal domain, there exists an $R$-basis $(e_1, \ldots, e_n)$ of $M_1,$ $r_1, \ldots, r_t,$ $t$ monic elements in $R,$ $t\leq n,$ such that if we set $f_i=r_ie_i, $ $i=1, \ldots, t,$ $f_i=e_i,$ $t+1\leq i\leq n,$ then $(f_1, \ldots, f_n)$ is an $R$-basis of $M_2.$  Thus, we have:
$$[M_1:M_2]_R= [\frac{M_1}{M_2}]_R.$$\par
If $M_1,M_2, M_3$ are three $R$-lattices in $V,$ then:
$$[M_1:M_3]_R=[M_1:M_2]_R[M_2:M_3]_R.$$\par
Let $M_1, M_2$ be two $R$-lattices in $V.$ Let $W$ be a sub $\mathbb K$-vector space of $V$ and let $N_1, N_2$ be two $R$-lattices in $W$ such that:
$$M_i\cap W=N_i, i=1,2.$$
Then $\frac{M_1}{N_1}$ and $\frac{M_2}{N_2}$ are two $R$-lattices in $\frac{V}{W}$ and we have:
$$[\frac{M_1}{N_1}:\frac{M_2}{N_2}]_R=\frac{[M_1:M_2]_R}{[N_1:N_2]_R}.$$
${}$\par
${}$\par

Let $G$ be a finite abelian group, we assume that $\mid G\mid$ is prime to the characteristic of $k.$ Let $\overline{k}$ be an algebraic closure of $k.$  We set $\widehat{G}={\rm Hom}(G, \overline{k}^\times).$ For $\chi \in \widehat G,$ let $k(\chi)=k(\chi(g), g\in G)\subset \overline{k},$ $R(\chi)=k(\chi)[\theta],$ and   $\mathbb K(\chi)= k(\chi)((\frac{1}{\theta})).$\par
For $\chi \in \widehat{G},$ we set:
$$e_{\chi}=\frac{1}{\mid G\mid }\sum_{g\in G} \chi( g) g^{-1} \in k(\chi)[G].$$
We also set:
$$[\chi]=\{ \psi \in \widehat{G}, \psi =\sigma\circ \chi \, {\rm for \, some\, } \sigma \in {\rm Gal}(k(\chi)/k)\},$$
and:
$$e_{[\chi]}=\sum_{\psi \in [\chi]}e_{\psi}\in k[G].$$
Let $M$ be a $k[G]$-module. Then, for $\chi \in \widehat{G},$ we set:
$$M(\chi)=e_{\chi} (M\otimes_kk(\chi)).$$
We have :
$$\oplus_{\psi \in [\chi]}M(\psi)=(e_{[\chi]}M)\otimes_kk(\chi).$$\par
Let $f\in  R[G]\cap {\rm Frac}(R)[G]^\times $ and let $\rho_f:R[G]\rightarrow R[G]$ be the morphism of $R$-modules given by the multiplication by
$f.$ Let $\det_R\rho_f\in R\cap {\rm Frac}(R)^\times $  be the determinant of the matrix of $\rho_f$ with respect to the $R$-basis given by the elements of $G ,$ then there exists a unique monic element $\det_Rf \in R$ such that:
$$(\det_Rf)k^\times =  (\det_R\rho_f) k^\times.$$
Observe that $\det_Rf$ is well-defined. For $\chi \in \widehat G,$ we define $f(\chi)\in R(\chi)$ by the equality:
$$e_{\chi} f=f(\chi)e_{\chi}.$$
We get:
$$(\det_Rf)k^\times =( \prod_{\chi\in \widehat{G}}f(\chi))k^\times.$$\par

 ${}$\par
Let $M$ be a finite dimensional $k$-vector space which is an $R[G]$-module.  Let $L/k$ be a finite extension containing $k(\chi),$ then $[e_{\chi}(M\otimes_kL)]_{L[\theta]}$ does not depend on $L.$ Observe that:
$$\forall \sigma \in {\rm Gal}(k(\chi)/k), \sigma([M(\chi)]_{R(\chi)})= [M(\sigma\circ \chi)]_{R(\chi)}.$$
 We set:
$$[M]_{R[G]}= \sum_{\chi\in \widehat{G}}[M(\chi)]_{R(\chi)}e_{\chi} \in R[G]\cap {\rm Frac}(R)[G]^\times.$$
Then:
$$\det_R[M]_{R[G]}= \prod_{\chi \in \widehat{G}}[M(\chi)]_{R(\chi)}= [M]_R.$$
Finally, observe that if $M_1,M, M_2$ are finite dimensional $k$-vector spaces that are also $R[G]$-modules and such that we have an exact sequence of $R[G]$-modules:
$$O\rightarrow M_1\rightarrow M\rightarrow M_2\rightarrow 0,$$
then:
$$[M]_{R[G]}=[M_1]_{R[G]} [M_2]_{R[G]}.$$ ${}$\par
${}$\par
Let $V$ be a finite dimensional $\mathbb K$-vector space which is a free $\mathbb K[G]$-module. An $R[G]$-lattice in $V,$  is  an $R$-lattice in $V,$  $M,$ such that $M$ is an $R[G]$-module. Observe that $M$ is a free $R[G]$-module: for $\chi \in \widehat{G}, $ $e_{[\chi]}M$ is a finitely generated $e_{[\chi]}R[G]$-module without torsion and thus a free $e_{[\chi]}R[G]$-module ($e_{[\chi]}R[G]$ is a principal ideal domain). Let $\chi \in \widehat{G},$ then $M(\chi)$ is an $R(\chi)$-lattice in the finite dimensional $\mathbb K(\chi)$-vector space $V(\chi).$ \par
Let $M_1,M_2$ be two $R[G]$-lattices in $V.$ Observe that:
$$\forall \sigma \in {\rm Gal}(k(\chi)/k), \sigma([M(\chi):M_2(\chi)]_{R(\chi)})= [M_1(\sigma\circ \chi):M_2(\sigma \circ \chi)]_{R(\chi)}.$$
 We set:
$$[M_1:M_2]_{R[G]}= \sum_{\chi\in \widehat{G}} [M_1(\chi):M_2(\chi)]_{R(\chi)}e_{\chi}\in \mathbb K[G]^\times.$$
We have:
$$[M_1:M_2]_R=\prod_{\chi \in \widehat{G}}[M_1(\chi):M_2(\chi)]_{R(\chi)}.$$
Furthermore, if $M_2\subset M_1,$ we have:
$$[M_1:M_2]_{R[G]}=[\frac{M_1}{M_2}]_{R[G]}.$$
Finally if $M_1, M_2, M_3$ are three $R[G]$-lattices in $V, $ then:
$$[M_1:M_3]_{R[G]}= [M_1:M_2]_{R[G]}[M_2:M_3]_{R[G]}.$$


\subsection{$\mathbb F_q(z)[\theta]$-Drinfeld modules}\label{Background}${}$\par
Let $\mathbb F_q$ be a finite field having $q$ elements and let $p$ be the characteristic of $\mathbb F_q.$ Let $\theta$ be an indeterminate over $\mathbb F_q$ and let $A=\mathbb F_q[\theta].$ Let $A_+$ be the set of monic polynomials in $A.$  For $n\geq 0,$ we denote by $A_{+,n}$ the set of elements of $A_+$ of degree $n.$ Let $K=\mathbb F_q(\theta)$ and $K_\infty=\mathbb F_q((\frac{1}{\theta})).$ Let $\mathbb C_\infty$ be the completion of a fixed algebraic closure of $K_\infty$ and let $v_\infty: \mathbb C_\infty \rightarrow \mathbb Q\cup\{+\infty\}$ be the valuation on $\mathbb C_\infty$ normalized such that $v_\infty(\theta)=-1.$\par
Let $z$ be an indeterminate over $\mathbb C_{\infty}.$  We still denote by $v_{\infty}$ the $\infty$-adic Gauss valuation on $\mathbb C_{\infty}(z),$ i.e. if $f\in \mathbb C_{\infty}[z],$ $f=\sum_{k=0}^n f_n z^n,$ $f_n\in \mathbb C_{\infty},$ $v_{\infty}(f)={\rm Inf}\{ v_{\infty}(f_k), k\in \{ 0, \ldots, n\}\}.$
Let $F$ be a subfield of $\mathbb C_{\infty}$  containing $K_{\infty}$ and complete for $v_{\infty},$ we denote by $\mathbb T_z(F)$ the completion of $F[z]$ for $v_{\infty}$ and $\widetilde{F}$ the completion of $F(z)$ for $v_{\infty}.$ Then the $\mathbb F_q(z)$-vector space generated by $\mathbb T_z(F)$ is dense in $\widetilde{F}.$\par

Let $L/K$ be a finite extension and let $L_{\infty}=L\otimes_KK_{\infty}.$ Let $O_L$ be the integral closure of $A$ in $L.$  Let $S_{\infty}(L)$ be the set of places of $L$ above $\infty.$ For each $v\in S_{\infty}(L),$ let $\iota_v: L\rightarrow \mathbb C_{\infty}$ be the morphism of $K$-algebras corresponding to $v,$ we denote by $L_v$ the field generated over $K_{\infty}$ by $\iota_v(L).$ Let's set:
$$\widetilde{L_{\infty}}= L\otimes_K\widetilde{K_{\infty}}.$$
Observe that we have an isomorphism of $\widetilde{K_{\infty}}$-algebras:
$$\widetilde{L_{\infty}}\simeq \prod_{v\in S_{\infty}(L)} \widetilde{L_v}.$$
Let $\tau:\widetilde{L_{\infty}}\rightarrow \widetilde{L_{\infty}}$ be the continuous morphism of $\mathbb F_q(z)$-algebras such that $\forall x\in L, \tau(x)=x^q.$ Let $\widetilde{O_L}$ be the $\mathbb F_q(z)$-vector space generated by $O_L$ in $\widetilde{L_{\infty}},$ and set $\widetilde{A}=\mathbb F_q(z)[\theta].$  \par
${}$\par

A Drinfeld $\widetilde{A}$-module defined over $\widetilde{O_L},$  $\varphi,$ is a morphism of $\mathbb F_q(z)$-algebras $\varphi : \widetilde{A}\rightarrow \widetilde{L_{\infty}}\{\tau\}$ such that $\varphi_{\theta}\equiv \theta \pmod{\tau}.$ We denote by $\varphi^0$ the trivial $\widetilde{A}$-Drinfeld module, i.e. $\varphi^0_{\theta}=\theta.$\par
${}$\par
Let $\varphi$ be an $\widetilde{A}$-Drinfeld module defined over $\widetilde{O_L},$ then there exists a unique element $\exp_{\varphi}\in L(z)\{\{ \tau\}\}$ such that (\cite{DEM}, Proposition 2.3):
$$\exp_{\varphi}\equiv 1\pmod{\tau},$$
$$\exp_{\varphi}\theta=\varphi_{\theta} \exp_{\varphi}.$$
Let's set:
$$H(\varphi/\widetilde{O_L})= \frac{\widetilde{L_{\infty}}}{\widetilde{O_L}+\exp_{\varphi}(\widetilde{L_{\infty}})}.$$
Then, by \cite{DEM}, Proposition 2.6, $H(\varphi/\widetilde{O_L})$ is a finite dimensional $\mathbb F_q(z)$-vector space  and an $\widetilde A$-module via $\varphi.$ Furthermore, by \cite{DEM}, Proposition 2.6, $\exp_{\varphi}^{-1}(\widetilde{O_L})$ is an $\widetilde{A}$-lattice in $\widetilde{L_{\infty}}.$\par
Let $\frak P$ be a non zero prime ideal of $O_L$ and let us denote by  $\varphi(\frac{\widetilde{O_L}}{\frak P\widetilde{O_L}})$ the $\mathbb F_q(z)$-vector space $\frac{\widetilde{O_L}}{\frak P\widetilde{O_L}}$ viewed as an $\widetilde{A}$-module via $\varphi.$ Then, by \cite{DEM}, Theorem 2.7, the following infinite product converges in $\widetilde{K_{\infty}}^\times:$
$$\mathcal L(\varphi/\widetilde{O_L}):=\prod_{\frak P {\rm \, prime\, ideal \, of \, }O_L, \frak P\not =(0)}\frac{[\frac{\widetilde{O_L}}{\frak P\widetilde{O_L}}]_{\widetilde{A}}}{[\varphi(\frac{\widetilde{O_L}}{\frak P\widetilde{O_L}})]_{\widetilde{A}}}\in \widetilde{K_{\infty}}^\times .$$
Furthermore, we have (\cite{DEM}, Theorem 2.7):
\begin{eqnarray}\label{eqSt1}\mathcal L(\varphi/\widetilde{O_L})= [\widetilde{O_L}:\exp_{\varphi}^{-1}(\widetilde{O_L})]_{\widetilde{A}}[H(\varphi/\widetilde{O_L})]_{\widetilde{A}}.\end{eqnarray}
Observe that $\mathcal L(\varphi^0/\widetilde{O_L})=1,$ $\exp_{\varphi^0}=1,$ $\exp_{\varphi^0}^{-1}(\widetilde{O_L})=\widetilde{O_L},$ $H(\varphi^0/\widetilde{O_L})=\{0\}.$

\subsection{Integrality results}\label{Integrality}${}$\par

Let $\varphi$ be an $\widetilde{A}$-Drinfeld module defined over $\widetilde{O_L}$ such that $\varphi_{\theta}\in O_L[z]\{\tau\}.$
Denote by $\mathbb T_z(L_{\infty})$ the closure of $L_{\infty}[z]$ in $\widetilde{L_{\infty}}.$ Then, we have an isomorphism of $\mathbb T_z(K_{\infty})$-algebras:
$$ \mathbb T_z(L_{\infty})\simeq \prod_{v\in S_{\infty}(L)}\mathbb T_z(L_v).$$
Observe that $\tau\mid_{\mathbb T_z(L_{\infty})}:\mathbb T_z(L_{\infty})\rightarrow \mathbb T_z(L_{\infty})$ is a continuous morphism of $\mathbb F_q[z]$-algebras.
Since $\exp_{\varphi}\theta= \varphi_{\theta} \exp_{\varphi}$ and $\exp_{\varphi}\equiv 1\pmod{\tau},$ we deduce:
$$\exp_{\varphi}\in \mathbb T_z(L_{\infty})\{\{ \tau\}\}.$$
Let's set:
$$H(\varphi/O_L[z])= \frac{\mathbb T_z(L_{\infty})}{O_L[z]+\exp_{\varphi}(\mathbb T_z(L_{\infty}))}.$$
Let $n=\dim_KL$ and let $e_1, \ldots, e_n\in O_L$ such that $O_L=\oplus_{i=1}^n Ae_i.$ Then:
$$\mathbb T_z(L_{\infty})= \oplus_{i=1}^n \mathbb T_z(K_{\infty})e_i.$$
Let $M=\{ x\in \mathbb T_z(K_{\infty}), v_{\infty}(x)\geq 1\},$ then:
$$\mathbb T_z(L_{\infty})= O_L[z]\oplus \oplus_{i=1}^nMe_i.$$
Since, by \cite{DEM}, Proposition 2.3,   there exists an integer $N\geq 0 $ such that $\exp_{\varphi}$ induces an isomorphism of $\mathbb F_q[z]$-module on $ \oplus_{i=1}^n\frac{1}{\theta^N}Me_i,$ we deduce that $H(\varphi /O_L[z])$ is a finitely generated $\mathbb F_q[z]$-module.
\begin{lemma}\label{lemmaSt2}
$$\mathcal L(\varphi/\widetilde{O_L})\in \mathbb T_z(K_{\infty})^\times.$$
\end{lemma}
\begin{proof} Let $\frak P$ be a non zero prime ideal in $O_L.$ Then we have an isomorphism of $\widetilde{A}$-modules:
$$\frac{\widetilde{O_L}}{\frak P\widetilde{O_L}}\simeq \frac{O_L}{\frak P}[z]\otimes_{\mathbb F_q[z]}\mathbb F_q(z).$$
Since $\frac{O_L}{\frak P}[z]$ is a free $\mathbb F_q[z]$-module, we deduce that:
$$[\varphi(\frac{\widetilde{O_L}}{\frak P\widetilde{O_L}})]_{\widetilde{A}}=\det\, _{\mathbb F_q[z,X]}(X{\rm Id}-\varphi_{\theta}\mid_{\frac{O_L}{\frak P}[z]})\mid_{X=\theta} \in A[z].$$
But :
$$\deg_{\theta} [\varphi(\frac{\widetilde{O_L}}{\frak P\widetilde{O_L}})]_{\widetilde{A}}=\dim_{\mathbb F_q}\frac{\widetilde{O_L}}{\frak P}=\deg_{\theta}[\frac{\widetilde{O_L}}{\frak P}]_{\widetilde{A}}.$$
Therefore $[\varphi(\frac{\widetilde{O_L}}{\frak P\widetilde{O_L}})]_{\widetilde{A}}$ is a monic polynomial in $\theta$ in $A[z]$ having the same degree as $[\frac{\widetilde{O_L}}{\frak P}]_{\widetilde{A}}$ which is a monic polynomial in $A.$ We get:
$$\frac{[\frac{\widetilde{O_L}}{\frak P\widetilde{O_L}}]_{\widetilde{A}}}{[\varphi(\frac{\widetilde{O_L}}{\frak P\widetilde{O_L}})]_{\widetilde{A}}}\in 1+\frac{1}{\theta} \mathbb F_q[z][[\frac{1}{\theta}]].$$
Since , by \cite{DEM}, Theorem 2.7, the infinite product $\prod_{\frak P {\rm \, prime\, ideal \, of \, }O_L, \frak P\not =(0)}\frac{[\frac{\widetilde{O_L}}{\frak P\widetilde{O_L}}]_{\widetilde{A}}}{[\varphi(\frac{\widetilde{O_L}}{\frak P\widetilde{O_L}})]_{\widetilde{A}}} $ converges in $\widetilde{K_{\infty}},$ we deduce that:
$$\mathcal L(\varphi/\widetilde{O_L})\in 1+\frac{1}{\theta}\mathbb F_q[z][[\frac{1}{\theta}]].$$
\end{proof}

\begin{proposition}\label{propositionSt1}${}$\par
\noindent 1) Let $U(\varphi/O_L[z])=\{ x\in \mathbb T_z(L_{\infty}), \exp_{\varphi}(x)\in O_L[z]\}.$  Then $\exp_{\varphi}^{-1}(\widetilde{O_L})$ is the $\mathbb F_q(z)$-vector space generated by $U(\varphi/O_L[z]).$ \par
\noindent 2) $U(\varphi/O_L[z])$ is a finitely generated $A[z]$-module.
\end{proposition}
\begin{proof} ${}$\par
\noindent 1) The proof is similar to that of \cite{APTR}, Proposition 5.4.  For simplicity let's denote by $V$ the $\mathbb F_q(z)$-vector space generated by $U(\varphi/O_L[z]).$ Then clearly $V\subset \exp_{\varphi}^{-1}(\widetilde{O_L}).$ Let $N$ be a suitable neighborhood of zero in $\widetilde{L_{\infty}}$ such that $\exp_{\varphi}: N\rightarrow N$ is an isomorphism of $\mathbb F_q(z)$-vector spaces. Let $W$ be the $\mathbb F_q(z)$-vector space generated by $\mathbb T_z(L_{\infty}),$ then:
$$\widetilde{L_{\infty}}= W+N.$$
Let $f\in \exp_{\varphi}^{-1}(\widetilde{O_L}),$ then we can write:
$$f=g+h, g\in W, h\in N.$$
Then:
$$\exp_{\varphi}(h)= \exp_{\varphi}(g)-\exp_{\varphi}(f)\in W\cap N.$$
Therefore $h\in W.$ This implies that $f\in V.$\par
\noindent 2) Since $\exp_{\varphi}^{-1}(\widetilde{O_L})$ is discrete in $\widetilde{L_{\infty}},$ we deduce that $U(\varphi/O_L[z])$ is discrete in $\mathbb T_z(L_{\infty}).$ Furthermore, the  $\widetilde{K_{\infty}}$-vector space generated by $\exp_{\varphi}^{-1}(\widetilde{O_L})$ is $\widetilde{L_{\infty}},$ thus we can select $f_1, \ldots, f_n\in U(\varphi/O_L[z])$ such that:
$$\exp_{\varphi}^{-1}(\widetilde{O_L})=\oplus_{i=1}^n \widetilde{A}f_i,$$
$$\widetilde{L_{\infty}}=\oplus_{i=1}^n \widetilde{K_{\infty}} f_i.$$
Set $N=\oplus_{i=1}^n A[z]f_i.$ Let $V$ be the $\mathbb T_z(K_{\infty})$-module generated by $N,$ then $V= \oplus_{i=1}^n \mathbb T_z(K_{\infty}) f_i.$ Let $W$ be the  $\mathbb T_z(K_{\infty})$-module generated by $U(\varphi/O_L[z]).$ Then : $V\subset W\subset \mathbb T_z(L_{\infty}).$ In particular, $W$ is a free $\mathbb T_z(K_{\infty})$-module of rank $n=[L:K].$
Observe that, because of 1), $U(\varphi/O_L[z]) \subset \mathbb F_q(z)N$, that is, if $x\in U(\varphi/O_L[z]),$ then there exists $\delta\in \mathbb F_q[z]\setminus\{ 0\}$ such that $\delta x\in N.$ This implies that there exists $\delta\in \mathbb F_q[z]\setminus\{ 0\}$ such that :
$$\delta W\subset V.$$
Thus:
$$\delta U(\varphi/O_L[z])\subset V \cap \mathbb F_q(z)N = N.$$
This implies that $U(\varphi/O_L[z])$ is a finitely generated $A[z]$-module.
\end{proof}

\subsection{The canonical $z$-deformation of a Drinfeld module}${}$\par
Let $\phi :A\rightarrow O_L\{\tau\}$ be a Drinfeld $A$-module, i.e. it is a morphism of $\mathbb F_q$-algebras such that:
$$\phi_{\theta}\equiv \theta\pmod{\tau}.$$
Let $\exp_\phi$ be the unique element in $L\{Ê\{Ê\tau\}\}$ such that $\exp_{\phi}\equiv 1\pmod{\tau }$ and:
$$\exp_{\phi} \theta=\phi_{\theta}\exp_{\phi}.$$
 Following Taelman (\cite{TAE1}), let's set:
$$H(\phi/O_L)=\frac{L_{\infty}}{O_L+\exp_\phi(L_{\infty})}.$$
Then $H(\phi /O_L)$ is a finite $A$-module (via $\phi$). Let's also set:
$$U(\phi /O_L)= \{ x\in L_{\infty}, \exp_\phi (x)\in O_L\}.$$
Then $U(\phi/O_L)$ is an $A$-lattice in $L_{\infty}$ and we have  (see \cite{TAE2}):
\begin{eqnarray}\label{eqSt2}\mathcal L(\phi /O_L)=[O_L:U(\phi /O_L)]_A[H(\phi /O_L)]_A,\end{eqnarray}
where:
$$\mathcal L(\phi /O_L)=\prod_{\frak P {\rm \, prime\, ideal \, of \, }O_L, \frak P\not =(0)}\frac{[\frac{O_L}{\frak P}]_{A}}{[\phi(\frac{O_L}{\frak P})]_{A} }\in  K_{\infty}^\times.$$
Write:
$$\phi_{\theta}=\sum_{j=0}^r \alpha_j \tau^j, \alpha_j \in O_L,j=0, \ldots, r,  \alpha_0=\theta.$$
The canonical $z$-deformation of $\phi,$ $\widetilde{\phi},$ is the $\widetilde {A}$-module defined over $\widetilde{O_L}$ given by:
$$\widetilde{\phi}_{\theta}=\sum_{j=0}^r z^j\alpha_j \tau^j\in O_L[z]\{ \tau\}.$$
Then, one can easily verify that:
$$\exp_{\widetilde{\phi}}=\sum_{j\geq 0} e_jz^j\tau^j,$$
where $\exp_{\phi}=\sum_{j\geq 0} e_j\tau ^j, e_j\in L, j\geq 0.$

\begin{proposition}\label{propositionSt2} $H(\widetilde{\phi}/O_L[z])$ is a finitely generated and torsion $\mathbb F_q[z]$-module. In particular:
$$H(\widetilde{\phi}/\widetilde{O_L})=\{0\}.$$

\end{proposition}
\begin{proof} By section \ref{Integrality}, $H(\widetilde{\phi}/O_L[z])$ is a finitely generated  $\mathbb F_q[z]$-module. Observe that, since $\exp_{\widetilde \phi} \equiv 1 \mod z$ : 
$$\mathbb T_z(L_{\infty})= z\mathbb T_z(L_{\infty})+\exp_{\widetilde{\phi}}(\mathbb T_z(L_{\infty})).$$
Therefore, the multiplication by $z$ on $H(\widetilde{\phi}/O_L[z])$ gives rise to an exact sequnece of finitely generated $\mathbb F_q[z]$-modules:
$$0\rightarrow H(\widetilde{\phi}/O_L[z])[z]\rightarrow H(\widetilde{\phi}/O_L[z])\rightarrow H(\widetilde{\phi}/O_L[z])\rightarrow 0,$$
where 
$$H(\widetilde{\phi}/O_L[z])[z]=\{x\in H(\widetilde{\phi}/O_L[z]), zx=0\}.$$
 This implies that $H(\widetilde{\phi}/O_L[z])[z]=\{0\}$ and that $H(\widetilde{\phi}/O_L[z])$ is a torsion $\mathbb F_q[z]$-module. Since the $\mathbb F_q(z)$-module generated by $\mathbb T_z(L_{\infty})$ is dense in $\widetilde{L_{\infty}},$ we deduce that the inclusion $\mathbb T_z(L_{\infty})\subset \widetilde{L_{\infty}}$ induces an isomorphism of $\widetilde{A}$-modules:
$$H(\widetilde{\phi}/O_L[z])[z]\otimes_{\mathbb F_q[z]}\mathbb F_q(z)\simeq H(\widetilde{\phi}/\widetilde{O_L}).$$
This implies the second assertion of the Lemma.

\end{proof}
\begin{remark} Let $\phi$ be a Drinfeld $A$-module defined over $O_L.$ Let $\frak P$ be a non zero prime ideal of $O_L.$   If $\phi_{\theta} \equiv \theta \pmod{\frak P}$ then :
$$[\phi(\frac{O_L}{\frak P})]_A=[\frac{O_L}{\frak P}]_A= N_{L/K}(\frak P),$$
where for $I\not =\{0\}$ an ideal of $O_L$ we set $N_{L/K}(I)= [\frac{O_L}{I}]_A.$ We also have:
$$[\widetilde{\phi}(\frac{\widetilde{O_L}}{\frak P\widetilde{O_L}})]_{\widetilde{A}}=N_{L/K}(\frak P),$$
In this case, we set:
$$f_{\frak P}(z)=0.$$
Thus we assume that $\phi_{\theta}\not \equiv \theta \pmod{\frak P}.$ Let $\overline{\phi}:A\rightarrow  \frac{O_L}{\frak P}\{\tau\}$ be the reduction modulo $\frak P$ of $\phi.$ Set $\frak L=\frac{O_L}{\frak P}.$ Let $P$ be the monic irreducible element in  $A$ such that $\frak P\cap A= PA.$  Then:
$$N_{L/K}(\frak P)=P^m,$$
where $\dim_{\mathbb F_q} \frak L= m\deg_{\theta} P.$ We have: $\frac{\widetilde{O_L}}{\frak P\widetilde{O_L}}=\frak L(z).$ Let's observe that:
$$g_{\frak P}(z):=[\widetilde{\phi}(\frak L(z))]_{\widetilde{A}}\in A[z].$$
Furthermore:
$$g_{\frak P}(0)= P^m,$$
$$g_{\frak P}(1)=[\phi(\frak L)]_A.$$
Let's make a simple observation. Let $M$ be a finitely generated and free $A$-module. Let $u:M\rightarrow M$ be an injective morphism of $A$-modules. Then:
$$(\det_A u ) A=[\frac{M}{u(M)}]_A A.$$
Let's view $\frak L\{\tau\}$ as a left $A$-module via $\overline{\phi}.$ Then it is a finitely generated  and free $A$-module.  We have an isomorphism of $A$-modules:
$$\frac{\frak L\{\tau \}}{\frak L\{ \tau\} (\tau-1)}\simeq \phi (\frak L).$$
Let $\rho : \frak L\{Ê\tau \} \rightarrow \frak L\{\tau \}$ be the morphism of $A$-modules given by :
$$\rho (x) = x\tau .$$
Then:
$$(\det_A({\rm Id}-\rho)) A=[\phi(\frak L)]_A A.$$
Let $X$ be a  variable, and set:
$$F(X)=\det_{A[X]}((1\otimes X){\rm Id}-\rho\otimes 1\mid_{\frak L\{\tau\}\otimes_{\mathbb F_q}\mathbb F_q[X]})\in A[X].$$
Since $\frac{\frak L\{\tau\} }{\frak L\{Ê\tau \} \tau}$ is isomorphic as an $A$-module to $\frak L,$  we have:
$$F(0)A= P^mA.$$
Now, observe that we have an isomorphism of $A$-modules:
$$\frac{\frak L[z]\{\tau \}}{\frak L[z]\{ \tau\} (\tau-z)}\simeq \widetilde{\phi} (\frak L[z]).$$
Thus:
$$g_{\frak P}(z) A[z]= F(z)A[z].$$
Let's set $f_{\frak P}(z)= P^m-g_{\frak P}(z) \in zA[z].$ If $I$ is a non zero ideal of $O_L,$ write $I=\prod_{j=1}^t\frak P_j^{n_j}$ its decomposition into prime ideals of $O_L,$ and write:
$$f_{I}(z)=\prod_{j=1}^t f_{\frak P_j}(z)^{n_j} \in A[z].$$
Then, we have:
$$\mathcal L(\widetilde{\phi}/\widetilde{O_L})= \sum_{I\not = \{ 0\}}\frac{f_I(z)}{N_{L/K}(I)}\in \mathbb T_z(K_{\infty})^\times.$$
In particular:
$$\mathcal L(\widetilde{\phi}/\widetilde{O_L})\mid_{z=0}=1,$$
$$\mathcal L(\widetilde{\phi}/\widetilde{O_L})\mid_{z=1}= \mathcal L(\phi/O_L)= \sum_{I\not = \{ 0\}}\frac{f_I(1)}{N_{L/K}(I)}\in K_{\infty}^\times.$$
\end{remark}

\subsection{Stark units for Drinfeld modules}${}$\par
Let $\phi:A\rightarrow O_L\{\tau\}$ be a Drinfeld $A$-module. Let $ev:\mathbb T_z(L_{\infty})\rightarrow L_{\infty}$ be the surjective morphism of $\mathbb F_q$-vector spaces given by $\forall f\in \mathbb T_z(L_{\infty}), ev(f)=f\mid_{z=1}.$ We set:
$$U_{St}(\phi/O_L)=ev(U(\widetilde{\phi}/O_L[z])).$$
We call $U_{St}(\phi/O_L)$ the $A$-module of Stark units relative to $\phi$ and $O_L.$ Since $ev(\exp_{\widetilde{\phi}})= \exp_\phi,$ $U_{St}(\phi/O_L)$ is a sub-$A$-module of $U(\phi/O_L).$

Let $\alpha : \mathbb T_z(L_{\infty})\rightarrow \mathbb T_z(L_{\infty})$ be the morphism of $\mathbb F_q[z]$-modules given by:
\begin{eqnarray}\label{eqSt3}\forall x\in\mathbb T_z(L_{\infty}), \alpha(x)=\frac{\exp_{\widetilde{\phi}}(x)-\exp_\phi(x)}{z-1}\in \mathbb T_z(L_{\infty}).\end{eqnarray}
Recall that $H(\widetilde{\phi}/O_L[z])$ is an $A[z]$-module via $\widetilde{\phi}.$ For $f\in A[z],$ we set 
$$H(\widetilde{\phi}/O_L[z])[f]=\{ x\in H(\widetilde{\phi}/O_L[z]), fx=0\}.$$
\begin{proposition}\label{propositionSt3} The map $\alpha$ induces an isomorphism of $A$-modules:
$$\overline{\alpha}: \frac{U(\phi/O_L)}{U_{St}(\phi/O_L)}\simeq H(\widetilde{\phi}/O_L[z])[z-1].$$ 
\end{proposition}
\begin{proof}${}$\par
 Let $x\in U(\phi/O_L),$ then:
$$(z-1) \alpha (x) = \exp_{\widetilde{\phi}}(x)-\exp_\phi(x)\in O_L+\exp_{\widetilde{\phi}}(\mathbb  T_z(L_{\infty})).$$
Therefore $(z-1)\alpha(x) =0$ in $H(\widetilde{\phi}/O_L[z]),$ i.e. the image $\alpha(x)$ in $H(\widetilde{\phi}/O_L[z])$ lies in  $H(\widetilde{\phi}/O_L[z])[z-1].$ Furthermore:
$$\forall x\in \mathbb T_z(L_{\infty}), \alpha(\theta x)=\widetilde{\phi}_{\theta}(\alpha(x))+ (\sum_{j=1}^r\alpha_j\frac{z^j-1}{z-1}\tau^j )(\exp_\phi(x)),$$
where $\phi_{\theta}=\sum_{j=0}^r \alpha_j \tau^j, \alpha_j \in O_L.$
Thus $\alpha$ induces  a morphism of $A$-modules $\overline{\alpha}:  U(\phi/O_L)\rightarrow H(\widetilde{\phi}/O_L[z])[z-1].$\par
Let $x\in \mathbb T_z(L_{\infty})$ such that the image of $x$ in $H(\widetilde{\phi}/O_L[z])$ lies in $H(\widetilde{\phi}/O_L[z])[z-1].$ Then:
$$(z-1)x=a+\exp_{\widetilde{\phi}}(h), a\in O_L[z], h\in \mathbb T_z(L_{\infty}).$$
Write $a=b+(z-1)c,$ $b\in O_L, c\in O_L[z],$ $h=g+(z-1)v,$ $g\in L_{\infty}, v\in \mathbb T_z(L_{\infty}).$ Then:
$$0= b+\exp_\phi(g).$$
Thus:
$$b=-\exp_\phi(g)\in O_L.$$
Thus $g\in U(\phi/O_L).$ We get:
$$(z-1)(x-c-\exp_{\widetilde{\phi}}(v))= \exp_{\widetilde{\phi}}(g)-\exp_\phi(g)=(z-1)\alpha(g).$$
Thus the image of $x$ in $H(\widetilde{\phi}/O_L[z])$ is equal to the image of $\alpha(g)$ in $H(\widetilde{\phi}/O_L[z]).$ This implies that  
$\overline{\alpha}:  U(\phi/O_L)\rightarrow H(\widetilde{\phi}/O_L[z])[z-1]$ is a surjective morphism of $A$-modules.\par
Let $x\in U_{St}(\phi/O_L).$ Then, there exist $u\in U(\widetilde{\phi}/O_L[z])$ and $h\in \mathbb T_z(L_{\infty})$ such that:
$$x= u+(z-1)h.$$
We get:
$$\exp_{\widetilde{\phi}}(x)= \exp_{\widetilde{\phi}}(u)+(z-1)\exp_{\widetilde{\phi}}(h).$$
Since $\exp_{\widetilde{\phi}}(u)\in O_L[z]$ and  $ev(\exp_{\widetilde{\phi}}(u))=\exp_\phi (x)\in O_L,$ we get:
\begin{eqnarray*}
\exp_{\widetilde{\phi}}(x)-\exp_\phi(x)&=&(\exp_{\widetilde{\phi}}(u)-\exp_\phi(x))+(z-1)\exp_{\widetilde{\phi}}(h)\\
&\in& (z-1)(O_L[z]+\exp_{\widetilde \phi}(\mathbb T_z(L_{\infty}))),
\end{eqnarray*}
i.e. :
$$\alpha(x)\in O_L[z]+\exp_{\widetilde{\phi}}(\mathbb T_z(L_{\infty})).$$
Thus $U_{St}(\phi/O_L)\subset {\rm Ker}\, \overline{\alpha}.$\par
Now, let $x\in U(\phi/O_L)$ such that $\alpha(x)\in O_L[z]+\exp_{\widetilde{\phi}}( \mathbb T_z(L_{\infty})).$ Then:
$$\exp_{\widetilde{\phi}}(x)\in O_L[z]+\exp_{\widetilde{\phi}}((z-1)\mathbb T_z(L_{\infty})).$$
Thus $x\in U(\widetilde{\phi}/O_L[z])+(z-1)\mathbb T_z(L_{\infty}).$ Thus:
$$x=ev(x)\in ev(U(\widetilde{\phi}/O_L[z]))=U_{St}(\phi/O_L).$$
Therefore ${\rm Ker} \,\overline{\alpha}=U_{St}(\phi/O_L).$
\end{proof}


\begin{theorem}\label{theoremSt1} $\frac{U(\phi/O_L)}{U_{St}(\phi/O_L)}$ is a finite $A$-module  and:
$$[H(\phi/O_L)]_A=[\frac{U(\phi/O_L)}{U_{St}(\phi/O_L)}]_A.$$
In particular:
$$\mathcal L(\phi/O_L)=[O_L:U_{St}(\phi/O_L)]_A.$$
\end{theorem}
\begin{proof}  ${}$\par
By Proposition  \ref{propositionSt2}, $H(\widetilde{\phi}/O_L[z])$ is a finite dimensional $\mathbb F_q$-vector space, thus a finite $A$-module (via $\widetilde{\phi}$). In particular $H(\widetilde{\phi}/O_L[z])[z-1]$ is a finite $A$-module. Thus Proposition \ref{propositionSt3} implies that $\frac{U(\phi/O_L)}{U_{St}(\phi/O_L)}$ is a finite $A$-module.\par
Observe that the map $ev$ induces an exact sequence of $A$-modules:
$$0\rightarrow (z-1)H(\widetilde{\phi}/O_L[z])\rightarrow H(\widetilde{\phi}/O_L[z])\rightarrow H(\phi/O_L)\rightarrow 0.$$
Therefore the multiplication by $z-1$ on $H(\widetilde{\phi}/O_L[z])$ gives rise to an exact sequence of finite $A$-modules:
$$0\rightarrow H(\widetilde{\phi}/O_L[z])[z-1]\rightarrow H(\widetilde{\phi}/O_L[z])\rightarrow H(\widetilde{\phi}/O_L[z])\rightarrow H(\phi/O_L)\rightarrow 0.$$
This implies that (recall  that if $0\rightarrow M_2\rightarrow M_1\rightarrow M_3\rightarrow 0$ is an exact sequence of finite $A$-modules then $[M_1]_A=[M_2]_A[M_3]_A$):
$$[H(\widetilde{\phi}/O_L[z])[z-1]]_A=[H(\phi/O_L)]_A.$$
It remains to apply Proposition \ref{propositionSt3} and formula (\ref{eqSt2}).
\end{proof}
\begin{lemma}\label{lemmaSt3} Let $\phi$ be an $A$-Drinfeld module defined over $O_L.$ Let $E/L$ be a finite extension. The  inclusion $L\subset E$ induces a natural injective morphism of $K_{\infty}$-algebras: $\iota_{E/L}: L_{\infty}\hookrightarrow E_{\infty}.$ Then:\par
\noindent 1) $\iota_{E/L}(U_{St}(\phi/O_L))\subset U_{St}(\phi/O_E);$\par
\noindent 2) $U_{St}(\phi/O_E)\cap \iota_{E/L}(L_{\infty})= \iota_{E/L}(U_{St}(\phi/O_L));$\par
\noindent 3) $Tr_{E/L}(U_{St}(\phi/O_E))\subset U_{St}(\phi/O_L).$

\end{lemma}
\begin{proof} Let's denote the natural morphism of $\mathbb T_z(K_{\infty})$-algebras induced by the inclusion $L\subset E$ by  $\iota'_{E/L}: \mathbb T_z(L_{\infty})\hookrightarrow \mathbb T_z(E_{\infty}).$ Then:
$$\iota'_{E/L}(U(\widetilde{\phi}/O_L[z]))\subset U(\widetilde{\phi}/O_E[z]).$$
Since $ev\circ \iota'_{E/L}=\iota_{E/L}\circ ev,$ we get 1). Let $Tr_{E/L}: E\rightarrow L$ be the trace map relative to the extension $E/L.$ Then $Tr_{E/L}$ induces a continuous morphism of $K_{\infty}$-algebras still denoted by $Tr_{E/L}:E_{\infty}\rightarrow L_{\infty}$ and a continuous morphism of $\mathbb T_z(K_{\infty})$-algebras $Tr'_{E/L}: \mathbb T_z(E_{\infty})\rightarrow \mathbb T_z(L_{\infty}).$ Observe that:
$$ev\circ Tr'_{E/L}=Tr_{E/L}\circ ev.$$
For $x\in \mathbb T_z(E_{\infty}),$ we have:
$$Tr'_{E/L}(\exp_{\widetilde{\phi}}(x))= \exp_{\widetilde{\phi}}(Tr'_{E/L}(x)).$$
Since $Tr'_{E/L}(O_E[z])\subset O_L[z],$ we get :
$$Tr'_{E/L}(U(\widetilde{\phi}/O_E[z]))\subset U(\widetilde{\phi}/O_L[z]).$$
Thus we get 3). Now, we observe that:
$$O_E[z]\cap \iota'_{E/L}(\mathbb T_z(L_{\infty}))=  \iota'_{E/L}(O_L[z]).$$
This is an easy consequence of the following fact: let $e_1, \ldots, e_k, \ldots e_n$ be a basis of $E$ over $K$ where $e_1, \ldots, e_k$ is a basis of $L$. Then one can write any $x\in \widetilde E_\infty$ in a unique way as $x=\sum_{i=1}^n e_i\otimes \lambda_i$ with $\lambda_i\in \widetilde K_\infty$.
This implies:
$$U(\widetilde{\phi}/O_E[z])\cap \iota'_{E/L}(\mathbb T_z(L_{\infty}))=\iota'_{E/L}(U(\widetilde{\phi}/O_L[z])).$$
Therefore, we get 2).
\end{proof}

\begin{corollary}\label{corollaryStark1}
Let $\phi$ be a Drinfeld $A$-module defined over $O_L.$ let $E/L$ be a finite extension. Then:
$$\frac{[H(\phi/O_E)]_A}{[H(\phi/O_L)]_A}\in A.$$ 
Furthermore $\frac{O_E}{\iota_{E/L}(O_L)}$ and $\frac {U_{St}(\phi/O_E)}{\iota_{E/L}(U_{St}(\phi/O_L))}$ are two $A$-lattices in the finite dimensional $K_{\infty}$-vector space $\frac{E_{\infty}}{\iota_{E/L}(L_{\infty})}$ and we have:
$$[\frac{O_E}{\iota_{E/L}(O_L)}:\frac {U_{St}(\phi/O_E)}{\iota_{E/L}(U_{St}(\phi/O_L))} ]_A=\frac{\mathcal L(\phi/O_E)}{\mathcal L(\phi/O_L)}.$$
\end{corollary}
\begin{proof} 
 By Lemma \ref{lemmaSt3}, $\iota_{E/L}$ induces an injective morphism of $A$-modules:
\begin{eqnarray}\label{eqnSt4}\frac{U(\phi/O_L)}{U_{St}(\phi/O_L)}\hookrightarrow\frac{U(\phi/O_E)}{U_{St}(\phi/O_E)} .\end{eqnarray}
 By Theorem \ref{theoremSt1}, the above $A$-modules are finite, thus:
 $$\frac{[\frac{U(\phi/O_E)}{U_{St}(\phi/O_E)} ]_A}{[ \frac{U(\phi/O_L)}{U_{St}(\phi/O_L)}]_A}\in A.$$
 Applying again Theorem \ref{theoremSt1}, we get:
 $$\frac{[H(\phi/O_E)]_A}{[H(\phi/O_L)]_A}\in A.$$
  Furthermore, by (\ref{eqnSt4}), we also have:
 $$[\frac{U(\phi/O_E)}{U_{St}(\phi/O_E)+\iota_{E/L}(U(\phi/O_L))}]_A=\frac{[H(\phi/O_E)]_A}{[H(\phi/O_L)]_A}.$$
 
 Observe that $O_E\cap \iota_{E/L}(L_{\infty})= \iota_{E/L}(O_L).$ 
 Now $\frac{O_E}{\iota_{E/L}(O_L)}$  is an $A$-lattice in $\frac{E_{\infty}}{\iota_{E/L}(L_{\infty})}.$ By Lemma \ref{lemmaSt3} and Theorem \ref{theoremSt1}, we also have that $\frac {U_{St}(\phi/O_E)}{\iota_{E/L}(U_{St}(\phi/O_L))}$  is an $A$-lattice in $\frac{E_{\infty}}{\iota_{E/L}(L_{\infty})}.$ We have:
 $$ [\frac{O_E}{\iota_{E/L}(O_L)}:\frac {U_{St}(\phi/O_E)}{\iota_{E/L}(U_{St}(\phi/O_L))} ]_A=\frac{[O_E:U_{St}(\phi/O_E)]_A}{[O_L:U_{St}(\phi/O_L)]_A}.$$
It remains to apply Theorem \ref{theoremSt1}.
\end{proof}
\begin{lemma}\label{lemmaEqui1} Let $\phi$ be a Drinfeld $A$-module defined over $O_L.$ Let $E/L$ be a finite Galois extension and let $G={\rm Gal}(E/L).$ Then $U(\phi/O_E)$ and $U_{St}(\phi/O_E)$ are $A[G]$-modules and:
$$U(\phi/O_E)^G= \iota_{E/L}(U(\phi/O_L)),$$
$$U_{St}(\phi/O_E)^G= \iota_{E/L}(U_{St}(\phi/O_L)).$$
\end{lemma}
\begin{proof} We prove the assertion for  Stark units. Observe that $\exp_{\widetilde{\phi}}:\mathbb T_z(E_{\infty})\rightarrow \mathbb T_z(E_{\infty})$ is a morphism of $\mathbb F_q[z][G]$-modules. Thus $U(\widetilde{C}/O_E[z])$ is an $A[z][G]$-module.  Since $ev:\mathbb T_z(E_{\infty})\rightarrow E_{\infty}$ is a morphism of $A[G]$-modules, we deduce that $U_{St}(\phi/O_E)$ is a  $A[G]$-module. Now, observe that:
$$E_{\infty}^G=\iota_{E/L}(L_{\infty}).$$
Therefore, by Lemma \ref{lemmaSt3}, we get:
$$U_{St}(\phi/O_E)^G=U_{st}(\phi/O_E)\cap \iota_{E/L}(L_{\infty})=\iota_{E/L}(U_{St}(\phi/O_L)).$$
\end{proof}

\begin{proposition}\label{PropositionEqui1}
Let $\varphi$ be a Drinfeld $\widetilde{A}$-module defined over $\widetilde{O_L}.$ Let $E/L$ be a finite abelian extension of degree prime to the characteristic of $\mathbb F_q,$ and let $G={\rm Gal}(E/L).$ Then:
$$[\widetilde{O_E}:U(\varphi/\widetilde{O_E})]_{\widetilde{A}[G]} [H(\varphi/\widetilde{O_E})]_{\widetilde{A}[G]}=\prod\limits_{\substack{\frak P\, {\rm prime\,  ideal\,  of \, } O_L,\\ \frak P\not =\{0\}}}\frac{[\frac{\widetilde{O_E}}{\frak P\widetilde{O_E}}]_{\widetilde{A}[G]}}{[\varphi(\frac{\widetilde{O_E}}{\frak P\widetilde{O_E}})]_{\widetilde{A}[G]}}\in \widetilde{K_{\infty}}[G]^\times.$$
\end{proposition}
\begin{proof} The above result is a consequence of the proof of \cite{ANG&TAE1}, Theorem A. We refer the reader to \cite{TAE2}, \cite{ANG&TAE1}, \cite{DEM}, \cite{FAN}, for the background on the class number formula. We only give a sketch of the proof. We also suspect that this equivariant  formula can also be deduced from \cite{FAN2}, working over $\mathbb F_q(z)$ instead of working over $\mathbb F_q.$ \par
 Recall that $O_E$ is a free $A[G]$-module.   Set $k=\mathbb  F_q(z).$ Let $F/k$ be the finite extension obtained by adjoining to $k$ the elements $\chi(g),$  $\forall \chi \in \widehat{G},\forall  g\in G.$\par
We have an exact sequence of $\widetilde{A}[G]$-modules (the module in the center  is a $\widetilde{A}$-module via ${\varphi}$):
$$O\rightarrow \frac{\widetilde{E_{\infty}}}{U(\varphi/\widetilde{O_E})}\rightarrow \frac{\widetilde{E_{\infty}}}{\widetilde{O_E}}\rightarrow H({\varphi}/\widetilde{O_E})\rightarrow 0.$$
Since $\frac{\widetilde{E_{\infty}}}{U({\varphi}/\widetilde{O_E})}$ is a divisible $\widetilde{A}[G]$-module and $\widetilde{A}[G]$ is a principal ideal ring, the above sequence of $\widetilde{A}[G]$-modules splits. We obtain an isomorphism of $\widetilde{A}[G]$-modules:
$$\gamma : \frac{\widetilde{E_{\infty}}}{U({\varphi}/\widetilde{O_E})}\times H({\varphi}/\widetilde{O_E})
\simeq \frac{\widetilde{E_{\infty}}}{\widetilde{O_E}}.$$
We still denote by $\tau$ the continuous morphism of $F(z)$-algebras: $\tau\otimes 1: \widetilde{E_{\infty}}\otimes_k F \rightarrow \widetilde{E_{\infty}}\otimes_k F.$  Let $\chi \in \widehat{G},$ by \cite{ANG&TAE1}, proof of Theorem A and \cite{DEM}, Proposition 3.7, we deduce that
$$[e_{\chi}(H({\varphi}/\widetilde{O_E})\otimes_kF)]_{F[\theta]}[e_{\chi} (\widetilde{O_E}\otimes_kF):e_{\chi}(U({\varphi}/\widetilde{O_E})\otimes_kF)]_{F[\theta]}$$
is equal to
$$\det_{F[[Z]]}(1+\Theta\mid \frac{\widetilde{L_{\infty}}\otimes_{\widetilde{A}}e_{\chi} (\widetilde{O_E}\otimes_kF)}{e_{\chi} (\widetilde{O_E}\otimes_kF)})\mid_{Z=\theta^{-1}}, $$
where we refer the reader to \cite{DEM}, section 3.1, for the definition of $\det_{F[[Z]]} (.),$ and where :
$$\Theta=\sum_{n\geq 1} (\theta-{\varphi}_{\theta}) \theta^{n-1} Z^n \in \widetilde{O_L}\{\tau\} [[Z]] .$$
Now, observe that $e_{\chi} (\widetilde{O_E}\otimes_kF)$ is a finitely generated and projective $\widetilde{O_L}$-module. Therefore, by an adaptation of \cite{TAE2} , Theorem 3,  similar to that made in  \cite{DEM},  and by \cite{ANG&TAE1}, paragraph 6.4, if we set $M=e_{\chi} (\widetilde{O_E}\otimes_kF),$ we get:
$$\det_{F[[Z]]}(1+\Theta\mid \frac{\widetilde{L_{\infty}}\otimes_{\widetilde{A}}M}{M})=\prod_{\frak P}\det_{F[[Z]]}(1+\Theta\mid \frac{M}{\frak P M})^{-1},$$
where the product is over the maximal ideals of $O_L.$ But observe that:
$$\prod_{\frak P}(\det_{F[[Z]]}(1+\Theta\mid \frac{M}{\frak P M})\mid_{Z=\theta^{-1}})^{-1}=\prod_{\frak P}\frac{[\frac{M}{\frak PM}]_{F[\theta]}}{[{\varphi}(\frac{M}{\frak PM})]_{F[\theta]}}.$$
Thus, we get the desired equivariant formula. The reader will notice that the above reasoning works also for Drinfeld $A$-modules leading to Fang's relative equivariant class number formula (\cite{FAN2}, Theorem 1.13) .
\end{proof}

\begin{corollary}\label{corollaryStark2} Let $\phi$ be a Drinfeld $A$-module defined over $O_L.$  Let $E/L$ be a finite abelian extension of degree prime to the characteristic of $\mathbb F_q,$ and let $G={\rm Gal}(E/L).$ Then:
$$[\frac{U(\phi/O_E)}{U_{St}(\phi/O_E)}]_{A[G]}= [H(\phi/O_E)]_{A[G]}.$$
We also have:
$$[O_E:U_{St}(\phi/O_E)]_{A[G]}=\prod_{\frak P\, {\rm prime\,  ideal\,  of \, } O_L, \frak P\not =\{0\}}\frac{[\frac{O_E}{\frak PO_E}]_{A[G]}}{[\phi(\frac{O_E}{\frak PO_E})]_{A[G]}}\in K_{\infty}[G]^\times.$$

\end{corollary}
\begin{proof} By Lemma \ref{lemmaEqui1}, we have that  $U(\phi/O_E)$ and  $U_{St}(\phi/O_E)$ are $A[G]$-lattices in $E_{\infty}.$ Now, observe that the map $\overline{\alpha}$ of Proposition \ref{propositionSt3} induces an isomorphism of $A[G]$-modules:
$$\frac{U(\phi/O_E)}{U_{St}(\phi/O_E)}\simeq H(\widetilde{\phi}/O_E[z])[z-1].$$
Thus, using similar arguments as that used   in the proof of Theorem \ref{theoremSt1}, we get the first assertion. The second assertion is a consequence of the first assertion and \cite{FAN2}, Theorem 1.13 (see also the proof of Proposition \ref{PropositionEqui1}).
\end{proof}
 Let $\phi$ be a Drinfeld $A$-module defined over $O_L.$ Let $P$ be a monic irreducible element in $A,$ and let $ \frak P_1, \ldots ,\frak P_t$ be the distinct maximal ideals of $O_L$ above $P.$ Write $PO_L =\frak P_1^{e_1}\cdots \frak P_t^{e_t}.$  Then we have an exact sequence of $A$-modules: 
 $$0\rightarrow \prod_{j=1}^t\frac{\frak P_j}{\frak P_j ^{e_j}}\rightarrow \frac{O_L}{ PO_L}\rightarrow \prod_{j=1}^t\frac{O_L}{\frak P_j}\rightarrow 0.$$
Now, observe that $\forall n\geq 1,$ for $j=1, \ldots, t,$ $\tau (\frak P_j^n)\subset \frak P_j^{n+1}.$ This implies:
$$\frac{[\frac{O_L}{PO_L}]_A}{[\phi (\frac{O_L}{PO_L})]_A}= \prod_{j=1}^t\frac{[\frac{O_L}{\frak P_j}]_A}{[\phi (\frac{O_L}{\frak P_j})]_A}.$$
Therefore:
$$\mathcal L(\phi/O_L)=\prod_{P}\frac{[\frac{O_L}{PO_L}]_A}{[\phi (\frac{O_L}{PO_L})]_A},$$
where $P$ runs through the monic irreducible elements in $A.$ Now let $E/L$ be a finite abelian extension of degree prime to the characteristic of $\mathbb F_q.$ Then, the above discussion implies:
$$\mathcal L(\phi/(O_E/O_L),G):= \prod_{\frak P\, {\rm prime\,  ideal\,  of \, } O_L, \frak P\not =\{0\}}\frac{[\frac{O_E}{\frak PO_E}]_{A[G]}}{[\phi(\frac{O_E}{\frak PO_E})]_{A[G]}}=\prod_{P}\frac{[\frac{O_E}{ PO_E}]_{A[G]}}{[\phi(\frac{O_E}{ PO_E})]_{A[G]}}.$$
Let $P$ be a monic irreducible element in $A,$ since $O_E$ is a free $A[G]$-module, we have:
$$[\frac{O_E}{ PO_E}]_{A[G]}=P^{[L:K]}1_G,$$
$$[\phi(\frac{O_E}{ PO_E})]_{A[G]}=\det_{\mathbb F_q[G][Z]}((1\otimes Z) {\rm Id}-\phi_{\theta}\otimes 1\mid_{\frac{O_E}{ PO_E}\otimes_{\mathbb F_q}\mathbb F_q[Z]})_{Z=\theta}.$$
If $\widetilde{\phi}$ is the canonical $z$-deformation of $\phi,$ then:
$$\mathcal L(\widetilde{\phi}/\widetilde{O_L})= \prod_{P}\frac{[\frac{O_L[z]}{PO_L[z]}]_{A[z]}}{[\phi (\frac{O_L[z]}{PO_L[z]})]_{A[z]}},$$
where, for $P$ a monic irreducible prime of $A,$ we have:
$$[\frac{O_L[z]}{PO_L[z]}]_{A[z]}= P^{[L:K]},$$
$$[\phi (\frac{O_L[z]}{PO_L[z]})]_{A[z]}=\det_{\mathbb F_q[z][Z]}((1\otimes Z) {\rm Id}-\phi_{\theta}\otimes 1\mid_{\frac{O_L[z]}{ PO_L[z]}\otimes_{\mathbb F_q}\mathbb F_q[Z]})_{Z=\theta}.$$
And also:
$$\mathcal L(\widetilde{\phi}/(\widetilde{O_E}/\widetilde{O_L}),G):=\prod_{\frak P\, {\rm prime\,  ideal\,  of \, } O_L, \frak P\not =\{0\}}\frac{[\frac{\widetilde{O_E}}{\frak P\widetilde{O_E}}]_{\widetilde{A}[G]}}{[\widetilde{\phi}(\frac{\widetilde{O_E}}{\frak P\widetilde{O_E}})]_{\widetilde{A}[G]}}=\prod_{P}\frac{[\frac{O_E[z]}{ PO_E[z]}]_{A[z][G]}}{[\widetilde{\phi}(\frac{O_E[z]}{ PO_E[z]})]_{A[z][G]}},$$
where, for a monic irreducible element $P$ of $A,$ we have:
$$[\frac{O_E[z]}{ PO_E[z]}]_{A[z][G]}=P^{[L/K]}1_G,$$
$$[\widetilde{\phi}(\frac{O_E[z]}{ PO_E[z]})]_{A[z][G]}=\det_{\mathbb F_q[G][z,Z]}((1\otimes Z) {\rm Id}-\widetilde{\phi}_{\theta}\otimes 1\mid_{\frac{O_E[z]}{ PO_E[z]}\otimes_{\mathbb F_q}\mathbb F_q[Z]})_{Z=\theta}.$$
In particular, since $\mathcal L(\widetilde{\phi}/(\widetilde{O_E}/\widetilde{O_L}),G)$ converges in $\widetilde{K_{\infty}}[G] ,$  we get that:
$$\mathcal L(\widetilde{\phi}/(\widetilde{O_E}/\widetilde{O_L}),G)\in (\mathbb T_z(K_{\infty})[G])^\times.$$
\begin{theorem}\label{TheoremSt2}
Let $\phi$ be a Drinfeld $A$-module defined over $A.$ Let $E/K$ be a finite abelian extension of degree prime to the characteristic of $\mathbb F_q.$ Let $G={\rm Gal}(E/K).$ Then:
$$\mathcal L(\widetilde{\phi}/(\widetilde{O_E}/\widetilde{A}), G)\, O_E[z]= U(\widetilde{\phi}/O_E[z]).$$
In particular:
$$U_{St}(\phi/O_E)=\mathcal L(\phi/(O_E/A), G) \, O_E.$$
\end{theorem}
\begin{proof} Since $O_E$ is a free $A[G]$-module of rank one, we observe that $\widetilde{O_E}$ is a free $\widetilde{A}[G]$-module of rank one. It is also the case for $U(\widetilde{\phi}/\widetilde{O_E}).$  Write:
$$\widetilde{O_E}= \widetilde{A}[G] \eta,$$ for some $\eta \in O_E,$ and:
$$U(\widetilde{\phi}/\widetilde{O_E})=\widetilde{A}[G] \varepsilon,$$
for some $\varepsilon \in U(\widetilde{\phi}/O_E[z])$ (we can find such an element by Proposition \ref{propositionSt1}). Then:
$$\widetilde{E_{\infty}}= \widetilde{K_{\infty}}[G] \eta = \widetilde{K_{\infty}}[G]\varepsilon.$$
This implies that there exists $f\in (\widetilde{K_{\infty}}[G])^\times$ such that:
$$\varepsilon = f\eta.$$
We have:
$$\mathbb T_z(E_{\infty})=\mathbb T_z(K_{\infty})[G]\eta.$$
Thus $f\in \mathbb T_z(K_{\infty})[G].$ 
Furthermore, by Proposition \ref{PropositionEqui1}, we have:
$$\mathcal L(\widetilde{\phi}/(\widetilde{O_E}/\widetilde{A}), G)=\delta f,$$
for some $\delta \in (\mathbb F_q(z)[G])^\times.$ But we have that $\mathcal L(\widetilde{\phi}/(\widetilde{O_E}/\widetilde{A}), G)\in (\mathbb T_z(K_{\infty})[G])^\times.$  This implies that $x: = \frac{1}{\delta}\in \mathbb  F_q[z][G] .$ We get:
$$\exp_{\widetilde{\phi}}(\varepsilon)\in x\mathbb T_z(E_{\infty})\cap O_E[z]= xO_E[z].$$
The  above equality comes from the fact that $x\mathbb T_z(K_{\infty})[G]\cap A[z][G]= xA[z][G].$ 
Set:
$$u=\frac{1}{x}\varepsilon=\delta \varepsilon \in U(\widetilde{\phi}/\widetilde{O_E}),$$
then:
$$\exp_{\widetilde{\phi}}(u)\in O_E[z].$$
We have:
$$u= \mathcal L(\widetilde{\phi}/(\widetilde{O_E}/\widetilde{A}), G) \eta\in \mathbb T_z(E_{\infty}).$$
Thus:
$$u\in U(\widetilde{\phi}/O_E[z]).$$
Therefore:
$$\mathcal L(\widetilde{\phi}/(\widetilde{O_E}/\widetilde{A}), G)O_E[z]= A[z][G] u\subset U(\widetilde{\phi}/O_E[z]).$$
This also implies that:
$$\mathbb T_z(E_{\infty})=\mathbb T_z(K_{\infty})[G]u.$$
Now, let $m\in U(\widetilde{\phi}/O_E[z]).$ Since $U(\widetilde{\phi}/\widetilde{O_E})=\widetilde{A}[G]u,$ by Proposition \ref{propositionSt1}, there exists $x\in \mathbb F_q[z]\setminus \{0\}$ such that $xm \in A[z][G]u.$ But :
$$A[z][G]\cap x\mathbb T_z(K_{\infty})[G]= xA[z][G].$$
Thus $m\in A[z][G]u.$ Therefore :
$$U(\widetilde{\phi}/O_E[z])=A[z][G] u=\mathcal L(\widetilde{\phi}/(\widetilde{O_E}/\widetilde{A}), G)O_E[z].$$
\end{proof}


\subsection{An example}\label{Anexample}${}$\par

Let $\widetilde{C}: \widetilde{A}\rightarrow \widetilde{A}\{\tau \}$ be the Drinfeld $\widetilde A$-module defined over $\widetilde A$  given by:
$$\widetilde{C}_{\theta}=z\tau +\theta.$$
Then:
 $$\exp_{\widetilde{C}}=\sum_{j\geq 0}\frac{z^j}{D_j}\tau^j \in K(z)\{ \{ \tau\}\},$$
 where $D_0=1,$ and for $i\geq 1,$ $D_i=(\theta^{q^i}-\theta) D_{i-1}^q.$\par
   
 Let $L/K$ be a finite extension and let $E/L$ be a finite abelian extension of degree prime to the characteristic of $\mathbb F_q.$ Let $G={\rm Gal}(E/L).$  Let $\frak P$ be a maximal ideal of $O_L.$ Let $G_{\frak P}\subset G$ be the decomposition group of $\frak P$ in $E/L.$ Fix $\frak M$ a maximal ideal of $O_E$ above $\frak P.$ Le $I_{\frak P}$ be the inertia group of $\frak P$ in $E/L,$ and let's define:
 $$\sigma_{\frak P}=\frac{1}{\mid I_{\frak P}\mid }\sum {\delta} \in  \mathbb F_q[G_{\frak P}],$$
 where the sum is over the elements $\delta \in G_{\frak P}$ such that the image of $\delta $ in $\frac{G_{\frak P}}{I_{\frak P}}$ is the Frobenius of the extension $\frac{O_E}{\frak M}/\frac{O_L}{\frak P}.$

\begin{lemma}\label{lemmaSt1} Let $\frak P$ be a non zero prime ideal of $O_L$ and let $P$ be the prime of $A$ such that $PA=\frak P\cap A.$ Write $[\frac{O_L}{\frak P}]_A=P^m,$ $m\geq 1.$ Then:
$$[\frac{\widetilde{O_E}}{\frak P\widetilde{O_E}}]_{\widetilde{A}[G]}= P^m1_G,$$
$$[\widetilde{C}(\frac{\widetilde{O_E}}{\frak P\widetilde{O_E}})]_{\widetilde{A}[G]}= P^m1_G-z^{m{\rm deg}_{\theta} P}\sigma_{\frak P}.$$
Furthermore, if $\frak P$ is unramified in $E,$ then we have an isomorphism of $\widetilde{A}[G]$-modules:
$$\widetilde{C}(\frac{\widetilde{O_E}}{\frak P\widetilde{O_E}})\simeq \frac{\widetilde{A}[G]}{ (P^m 1_G-z^{m\deg_{\theta} P} \sigma_{\frak P})\widetilde{A}[G]} .$$
\end{lemma}
\begin{proof} The computation of equivariant Fitting ideals can be done by an adaptation of the results in  \cite{FAN2}, section 6. However, we present here an alternative proof.  Fix a maximal ideal $\frak M$ of $O_E$ above $\frak P$  and let $e$ be the ramification index of $\frak M/\frak P.$Then we have an isomorphism of $\widetilde{A}[G]$-modules:
$$\frac{O_E}{\frak P O_E}\simeq \frac{O_E}{\frak M ^e}\otimes _{\mathbb F_q[G_{\frak P}]}\mathbb  F_q[G].$$
Observe that we have an isomorphism of $\widetilde{A}$-modules:
$$\widetilde{O_E}\simeq O_E[z]\otimes_{\mathbb F_q[z]}\mathbb F_q(z).$$
Therefore, we have an isomorphism of $\widetilde{A}$-modules:
$$\frac{\widetilde{O_E}}{\frak M^e\widetilde{O_E}}\simeq \frac{O_E}{\frak M^e}[z]\otimes_{\mathbb F_q[z]}\mathbb F_q(z).$$
Now, by \cite{NOE},  $\frac{O_E}{\frak M^e}[z]$ is a free $\frac{O_L}{\frak P}[z][G_{\frak P}]$-module. Thus:
$$[\frac{\widetilde{O_E}}{\frak P\widetilde{O_E}}]_{\widetilde{A}[G]}=P^m1_G.$$

We observe that:
$$\widetilde{C}_{P}\equiv z^d \tau ^d \pmod{P\widetilde A\{\tau\}},$$
where $d=deg_{\theta}P.$  Set:
$$\overline{G_{\frak P}}= \frac {G_{\frak P}}{I_{\frak P}}.$$
The natural morphism $G_{\frak P}\rightarrow \overline{G_{\frak P}}$ induces a surjective morphism of $\mathbb F_q$-algebras $\mathbb F_q[G_{\frak P}]\rightarrow \mathbb F_q[\overline{G_{\frak P}}].$ If $x \in \mathbb F_q[G_{\frak P}],$ we denote its image in $\mathbb F_q[\overline{G_{\frak P}}]$ by $\overline x.$
We have :
$$(\widetilde{C}_{P^m}-z^{dm}\overline{\sigma_{\frak P}})(\frac{\widetilde{O_E}}{\frak M\widetilde{O_E}})=\{0\}.$$
Furthermore, observe that:
$$\{ x\in \frac{\widetilde{O_E}}{\frak M\widetilde{O_E}}, \tau(x)=x\}=\mathbb F_q(z).$$
Since $\frac{\widetilde{O_E}}{\frak M\widetilde{O_E}}$ is a free $\frac{\widetilde{O_L}}{\frak P\widetilde{O_L}}[\overline{G_{\frak P}}] $-module  of rank one   and $md=\dim_{\mathbb F_q(z)}\frac{\widetilde{O_L}}{\frak P\widetilde{O_L}},$ by  an adaptation to our case of the proof of  \cite{APTR}, Lemma 5.7, and the  proof of \cite{APTR}, Lemma 5.8, we  deduce that there exists a finite extension $F/\mathbb F_q$ (simply take the finite extension of $\mathbb F_q$ obtained by adjoining the values of the characters of $\overline{G_{\frak P}}$) such that  we have an isomorphism of $\widetilde{A}[\overline{G_{\frak P}}]\otimes_{ \mathbb F_q} F$-modules:
$$\widetilde{C}(\frac{\widetilde{O_E}}{\frak M\widetilde{O_E}})\otimes _{\mathbb F_q}F\simeq \frac{\widetilde{A}[\overline{G_{\frak P}}]\otimes_{\mathbb F_q}F}{((P^m1_{\overline{G_{\frak P}}}-z^{dm}\overline{\sigma_{\frak P}})\otimes 1)\widetilde{A}[\overline{G_{\frak P}}]\otimes_{\mathbb F_q} F}.$$
Therefore, we have an isomorphism of $\widetilde{A}[\overline{G_{\frak P}}]$-modules:
$$\widetilde{C}(\frac{\widetilde{O_E}}{\frak M\widetilde{O_E}})\simeq \frac{\widetilde{A}[\overline{G_{\frak P}}]}{(P^m1_{\overline{G_{\frak P}}}-z^{dm}\overline{\sigma_{\frak P}})\widetilde{A}[\overline{G_{\frak P}}]} .$$
This implies that we have  an isomorphism of $\widetilde{A}[G_{\frak P}]$-modules:
$$\widetilde{C}(\frac{\widetilde{O_E}}{\frak M\widetilde{O_E}})\simeq \frac{\widetilde{A}[{G_{\frak P}}]}{(1_{G_\frak P}+(P^m-1)u-z^{dm}{\sigma_{\frak P}})\widetilde{A}[{G_{\frak P}}]},$$
where $u=\frac{1}{\mid I_{\frak P}\mid } \sum_{\delta\in I_{\frak P}} \delta \in \mathbb  F_q[I_{\frak P}].$
Since, for $n\geq 1,$ $\tau (\frak M^n) \subset \frak M^{n+1}, $  we have:
$$[\widetilde{C}(\frac{\frak M\widetilde{O_E}}{\frak M^e\widetilde{O_E}})]_{\widetilde{A}[G_{\frak P}]}=[\frac{\frak M\widetilde{O_E}}{\frak M^e\widetilde{O_E}}]_{\widetilde{A}[G_{\frak P}]}.$$
Now, observe that:
$$[\frac{\widetilde{O_E}}{\frak M\widetilde{O_E}}]_{\widetilde{A}[G_{\frak P}]}= 1_{G_{\frak P}} +(P^m-1) u,$$
and:
$$[\frac{\frak M\widetilde{O_E}}{\frak M^e\widetilde{O_E}}]_{\widetilde{A}[G_{\frak P}]}= (P^m-1)(1_{G_{\frak P}}-u)+1_{G_{\frak P}}.$$
Thus:
$$[\widetilde{C}(\frac{\widetilde{O_E}}{\frak M^e\widetilde{O_E}})]_{\widetilde{A}[G_{\frak P}]}=P^m 1_{G_{\frak P}}-z^{dm}{\sigma_{\frak P}}.$$
But we have an isomorphism of $\widetilde{A}[G]$-modules:
$$\widetilde{C}(\frac{\widetilde{O_E}}{\frak P \widetilde{O_E}})\simeq \widetilde{C}(\frac{\widetilde{O_E}}{\frak M ^e\widetilde{O_E}})\otimes _{\mathbb F_q[G_{\frak P}]}\mathbb  F_q[G].$$
We have:
$$[\widetilde{C}(\frac{\widetilde{O_E}}{\frak P \widetilde{O_E}})]_{\widetilde{A}[G]}=\det_{\mathbb F_q[G][Z,z]}((1\otimes Z) {\rm Id}-\widetilde{C}_{\theta} \otimes 1\mid_{\frac{\widetilde{O_E}}{\frak P \widetilde{O_E}}\otimes_{\mathbb F_q}\mathbb F_q[Z]})_{Z=\theta}.$$
Thus:
$$[\widetilde{C}(\frac{\widetilde{O_E}}{\frak P \widetilde{O_E}})]_{\widetilde{A}[G]}=\det_{\mathbb F_q[G_{\frak P}][Z,z]}((1\otimes Z) {\rm Id}-\widetilde{C}_{\theta} \otimes 1\mid_{\frac{\widetilde{O_E}}{\frak M^e \widetilde{O_E}}\otimes_{\mathbb F_q}\mathbb F_q[Z]})_{Z=\theta}\otimes 1_G.$$


\end{proof}
If $I$ is a non-zero ideal of $O_L,$ recall that we have set:
$$N_{L/K}(I)=[\frac{O_L}{I}]_A.$$
Write $I=\frak P_1^{e_1}\cdots, \frak P_t^{e_t}$ its decomposition into maximal ideals of $O_L,$ we set:
$$\sigma_I=\prod_{j=1}^t \sigma_{\frak P_j}^{e_j} \in \mathbb F_q[G].$$
By the above Lemma, we get:
$$\mathcal L(\widetilde{C}/(\widetilde{O_E}/\widetilde{O_L}), G)=\sum_{I{\rm \, ideal\,  of \,} O_L, I\not =(0)}\frac{z^{\deg_{\theta}N_{L/K}( I)}\sigma_I}{N_{L/K}(I)} \in \mathbb (T_z(K_{\infty})[G])^\times.$$
${}$\par
Let $C$ be the Carlitz module, i.e. $C: A\rightarrow A\{\tau \}$ is the morphism of $\mathbb F_q$-algebras given by:
$$C_{\theta}= \tau +\theta.$$
Then Theorem \ref{TheoremSt2} implies Anderson's log-algebraicity Theorem (\cite{AND}, see also \cite{APTR}, \cite{APTR2}). More precisely we have:
\begin{corollary}\label{CorollarySt3} Let $L/K$ be a finite extension and let $m\geq 1$ be an integer. Let $x\in L_{\infty},$ and set:
$$\ell_m(x)=\sum_{a\in A_{+}}\frac{C_a(x)^m}{a} z^{\deg_{\theta} a}\in L_{\infty}[[z]].$$
If $\ell_m(x)$ converges in $\mathbb T_z(L_{\infty}),$ then:
$$\exp_{\widetilde{C}}(\ell_m(x))\in A[z][x]\subset \mathbb  T_z(L_{\infty}).$$
In that case:
$$ev(\exp_{\widetilde{C}}(\ell_m(x)))= \exp_C (ev(\ell_m(x))) \in A[x]\subset L_{\infty}.$$
\end{corollary}
\begin{proof} Let $P$ be a monic irreducible element in $A$ and let $M$ be the $P$th cyclotomic function field (\cite{ROS}, chapter 12). Recall that $M/K$ is a finite abelian extension of degree $q^d-1$ where $d=\deg_{\theta}P.$ Select $\lambda_P\in M^\times$  such that $C_P(\lambda_P)=0.$ Then (\cite{ROS}, Proposition 12.9):
$$O_M=A[\lambda_P].$$
Furthermore, for $a\in A\setminus PA,$ we have (\cite{ROS}, Theorem 12.10):
$$\sigma_{aA}(\lambda_P)=C_a(\lambda_P).$$
Also observe, that :
$$\sigma _{PA}= -\sum_{g\in G} g,$$
where $G={\rm Gal}(M/K).$ For $a \in A_+,$ we set:
$$\sigma _a=\sigma_{aA}.$$
We get:
$$\mathcal L(\widetilde{C}/(\widetilde{O_M}/\widetilde{A}), G)=
(\sum_{a\in A_+\setminus PA_+} \frac{\sigma_a z^{\deg_{\theta} a}}{a})(1+ \frac{z^d \sum_{g\in G }g}{P})^{-1}.$$
Let $n\in \{ 1, \ldots, q^d-2\},$ then :
$$(1+ \frac{z^d \sum_{g\in G }g}{P})\lambda_P^n= \lambda_P^n .$$
Therefore, for $n\in \{ 1, \ldots, q^d-2\},$ by Theorem \ref{TheoremSt2}, we have:
$$\mathcal L(\widetilde{C}/(\widetilde{O_M}/\widetilde{A}), G)\lambda_P^n=\sum_{a\in A_+} \frac{C_a(\lambda_P)^nz^{\deg_{\theta}a}}{a}\in U(\widetilde{C}/O_M[z]).$$
This implies that, for $n\in \{ 1, \ldots, q^d-2\},$ we have:
$$\forall m\geq 0, \sum_{i+j=m} \frac{1}{D_i} \sum_{a\in A_{+,j}}\frac{(C_a(\lambda_P)^n)^{q^i}}{a^{q^i}} \in A[\lambda_P].$$
Now, let $X$ be an undeterminate over $K.$ Let $\tau : K[ X]\rightarrow K[ X]$ be the morphism of $\mathbb F_q$-algebras such that $\forall x\in K[X],\tau (x)=x^q.$ Let $n\geq 1,$  and set:
$$\forall m\geq 0, f_m(X)=\sum_{i+j=m} \frac{1}{D_i} \sum_{a\in A_{+,j}}\frac{(C_a(X)^n)^{q^i}}{a^{q^i}}\in K[X].$$ 
Fix an integer $m\geq 0.$ Then, for  all   irreducible monic polynomials $P$ in $A:$
$$f_m(\lambda_P) \in A[\lambda_P].$$
This easily implies that:
$$ f_m(X) \in A[X].$$
We have:
$$\exp_{\widetilde{C}}(\sum_{a\in A_+}\frac{C_a(X)^nz^{\deg_{\theta}a}}{a})=\sum_{m\geq 0} f_m(X) z^m.$$
Therefore, for all $n\geq 1:$
$$\exp_{\widetilde{C}}(\sum_{a\in A_+}\frac{C_a(X)^nz^{\deg_{\theta}a}}{a}) \in X^nA[X][[z]].$$
Now, we work in $\mathbb T_z(\mathbb C_{\infty}).$ Fix an integer $n\geq 1.$ Then if $\lambda$ is a torsion point of the Carlitz module:
$$\ell_n(\lambda)\in \mathbb  T_z(\mathbb C_{\infty}).$$
Observe that, for all $a\in A_+,$ we have:
$$v_{\infty}(C_a(\lambda))\geq \frac{-1}{q-1}.$$
Thus there exist $x_{\lambda}\in \mathbb T_z(\mathbb  C_{\infty})$ and $y(X)\in K[X][z]$ which does not depend on $\lambda,$  such that :
$$v_{\infty}(x_{\lambda}) \geq 1,$$
and:
$$\exp_{\widetilde{C}}(\ell_n(\lambda))=x_{\lambda}+y\mid_{X=\lambda}.$$
Thus, we have:
$$\forall m\geq  C(n),\forall \lambda,  v_{\infty}(f_m(\lambda))\geq 1,$$
where $C(n)$ is a constant depending only on $n.$ But, since for all $m\geq 0$  $f_m(X) \in A[X],$ we get:
$$\forall m\geq  C(n), f_m(X)=0.$$
\end{proof}


\section{Deformation of Drinfeld modules over several variable Tate algebras}\label{section2}${}$\par
 The notion of deformation of a Drinfeld module over Tate algebras  as introduced in section \ref{Basicproperties} has its roots in a remarkable formula obtained by F. Pellarin in \cite{PEL}. This formula links a certain one variable $L$-series to Anderson-Thakur special function. This formula was fully understood when  F. Pellarin and the authors (\cite{APTR}) found a  connection with Taelman's  work (\cite{TAE2}). In this section, we show how the ideas developed in section \ref{section1} can be extended to the situation of deformations of Drinfeld modules   and we will study in details the case of the Carlitz module leading to the proof of the discrete Greenberg Conjectures. In the last section, we will construct $P$-adic $L$-series attached to deformations of Drinfeld modules.\par

\subsection{Basic properties}\label{Basicproperties}${}$\par
We fix some notation. Let $n\geq 0$ be a fixed integer and let $t_1, \ldots, t_n, z$ be $n+1$ indeterminates over $K_{\infty}.$ We set $k=\mathbb F_q(t_1, \ldots, t_n), $ $\mathbb A= k[\theta],$ $\widetilde{\mathbb A}= k(z)[\theta],$ $\mathbb K_{\infty}=k((\frac{1}{\theta})),$ $\widetilde{\mathbb K_{\infty}}=k(z)((\frac{1}{\theta})).$ We denote by $v_{\infty}$ the normalized $\frac{1}{\theta}$-adic valuation on $\widetilde{\mathbb K_{\infty}}.$ Let 
$\mathbb T_{n,z}(K_{\infty})$ be the closure of $K_{\infty}[t_1, \ldots, t_n,z]$ in $\widetilde{\mathbb K_{\infty}},$  and $\mathbb T_n(K_{\infty})$ the closure of $K_{\infty}[t_1, \ldots, t_n]$ in $\mathbb K_{\infty}.$ Let $\tau : \widetilde{\mathbb K_{\infty}}\rightarrow \widetilde{\mathbb K_{\infty}}$ be the continuous morphism of $k(z)$-algebras given by $\tau (\theta)= \theta^q.$ Finally, we set $b_0(z)=1,$ and for $m\geq 1,$ $b_m(z)=\prod_{j=0}^{m-1} (z-\theta^{q^j}).$\par
${}$\par
Let $\phi: A\rightarrow A\{\tau\}$ be a Drinfeld $A$-module defined over $A,$ and write:
$$\phi_{\theta}= \sum_{j=0}^r \alpha_j \tau^j,\, \alpha_0=\theta, \alpha_j\in A, \alpha_r\not =0,  r\geq 1.$$
${}$\par
Let $ \varphi: \mathbb A\rightarrow \mathbb A\{\tau\}$ be the morphism of $k$-algebras given by:
$$\varphi_{\theta}= \sum_{j=0}^r \alpha_j b_j(t_1)\cdots b_j(t_n)\tau^j.$$
We call $\varphi$ the canonical deformation of $\phi$ over the Tate algebra $\mathbb T_n(K_{\infty}).$\par
There exists a unique element $\exp_{\varphi}\in \mathbb K_{\infty}\{\{ \tau\}\}$ such that $\exp_{\varphi}\equiv 1\pmod{\tau}$ and:
$$\exp_{\varphi} \theta = \varphi_{\theta} \exp_{\varphi}.$$
If we write $\exp_\phi=\sum_{j\geq 0} e_j \tau^j,$ we get:
$$\exp_{\varphi}=\sum_{j\geq 0} b_j(t_1)\cdots b_j(t_n)e_j \tau^j.$$
Observe that by \cite{DEM}, Proposition 2.3, $\exp_{\varphi}$ converges on $\mathbb K_{\infty}$ and furthermore $\exp_{\varphi}(\mathbb T_n(K_{\infty}))Ê\subset \mathbb T_n(K_{\infty}).$ Let's set:
$$U(\varphi/\mathbb A)=\{ x\in \mathbb K_{\infty}, \exp_{\varphi}(x)\in \mathbb A\}.$$
Then, by \cite{DEM}, Proposition 2.6, $U(\varphi/\mathbb A)$  is a $\mathbb A$-lattice in $\mathbb K_{\infty},$ thus a free $\mathbb A$-module of rank one. Let's also set:
$$H(\varphi/\mathbb A)= \frac{\mathbb K_{\infty}}{\mathbb A+\exp_{\varphi}(\mathbb K_{\infty})}.$$
Then, by \cite{DEM}, Proposition 2.6, $H(\varphi/\mathbb A)$ is a finite dimensional $k$-vector space and also a $\mathbb A$-module via $\varphi.$ 
By \cite{DEM}, Theorem 2.7, the following infinite product converges in $\mathbb  K_{\infty}:$
$$\mathcal L(\varphi/\mathbb A):=\prod_P\frac{[\frac{\mathbb A}{P\mathbb A}]_{\mathbb A}}{[\varphi(\frac{\mathbb A}{P\mathbb A})]_{\mathbb A}}\in \mathbb K_{\infty}^\times,$$
where the product is over the monic irreducible polynomials in $A.$ Since $\varphi_{\theta} \in A[t_1, \ldots, t_n]\{Ê\tau\},$ we deduce that in fact:
$$\mathcal L(\varphi/\mathbb A)\in \mathbb T_n(K_{\infty})^\times.$$
We have the following result (\cite{DEM}, theorem 2.7):
\begin{eqnarray}\label{EqS2-1}\mathcal L(\varphi/\mathbb A) \mathbb A= [H(\varphi/\mathbb A)]_{\mathbb A}U(\varphi/\mathbb A)\end{eqnarray} 
Let's set:
$$U(\varphi/A[t_1, \ldots, t_n])=\{ x\in \mathbb T_n(K_{\infty}), \exp_{\varphi}(x)\in A[t_1, \ldots, t_n]\}.$$
\begin{lemma}\label{LemmaS2-1} We have that $[H(\varphi/\mathbb A)]_{\mathbb A} \in A[t_1, \ldots, t_n],$ and: 
$$[H(\varphi/\mathbb A)]_{\mathbb A} U(\varphi/A[t_1, \ldots, t_n])=\mathcal L(\varphi/\mathbb A) A[t_1, \ldots, t_n].$$
\end{lemma}
\begin{proof} By a similar argument to that used in the proof of Proposition \ref{propositionSt1}, we have that $U(\varphi/\mathbb A)$ if equal to the $k$-vector space spanned by $U(\varphi/A[t_1, \ldots, t_n]).$  Recall that $\mathbb T_n(K_{\infty})$ is a unique factorization domain (\cite{PUT}, chapter 3). Now, select $\varepsilon \in U(\varphi/A[t_1, \ldots, t_n])$ such that:
$$U(\varphi/\mathbb A)= \mathbb A \varepsilon.$$
Observe that for $\delta\in \mathbb F_q[t_1, \ldots, t_n]\setminus\{0\}, $ we have:
$$A[t_1, \ldots, t_n]\cap \delta \mathbb T_n(K_{\infty})= \delta A[t_1, \ldots, t_n].$$
This implies that we can assume that $\varepsilon$ is a primitive element in $\mathbb T_n(K_{\infty}),$ i.e. the elements of $\mathbb F_q[t_1, \ldots, t_n]\setminus\{0\}$ that divides $\varepsilon$ in $\mathbb T_n(K_{\infty})$ are only the elements in $\mathbb F_q^\times.$
Then by formula (\ref{EqS2-1}), there exists $x\in k^\times$ such that:
$$\mathcal L(\varphi/\mathbb A) = [H(\varphi/\mathbb A)]_{\mathbb A} x\varepsilon.$$
But $\mathcal L( \varphi/\mathbb  A) \in \mathbb T_n(K_{\infty})^\times$ and $\varepsilon$ is primitive, since $[H(\varphi/\mathbb A)]_{\mathbb A}$ is monic, this implies that $[H(\varphi/\mathbb A)]_{\mathbb A} \in A[t_1, \ldots, t_n],$ $  \varepsilon \in \mathbb T_n(K_{\infty})^\times ,$ and: 
$$ \mathcal L(\varphi/\mathbb A) = \lambda[H(\varphi/\mathbb A)]_{\mathbb A} \varepsilon,$$
for some $\lambda \in \mathbb F_q^\times.$
 Now let $m\in U(\varphi/A[t_1, \ldots, t_n]).$ By formula (\ref{EqS2-1}), there exists $x\in \mathbb F_q[t_1, \ldots, t_n]\setminus\{0\}$ such that $x[H(\varphi/\mathbb A)]_{\mathbb A} m\in A[t_1, \ldots, t_n] \mathcal L(\varphi/\mathbb A).$ But $\mathcal L(\varphi/\mathbb A) \in \mathbb T_n(K_{\infty})^\times$ and $\mathbb T_n(K_{\infty})$ is a unique factorization domain, thus $[H(\varphi/\mathbb A)]_{\mathbb A}m\in A[t_1, \ldots, t_n]\mathcal L(\varphi/\mathbb A).$
\end{proof}
Let's set:
$$H(\varphi/A[t_1, \ldots, t_n])= \frac{\mathbb T_n(K_{\infty})}{A[t_1, \ldots, t_n]+\exp_{\varphi}(\mathbb T_n(K_{\infty})}.$$
We observe that there exists an integer $N\geq 1$ such that $\exp_{\varphi}$ induces a bijective morphism of $\mathbb F_q[t_1, \ldots, t_n]$-modules on $\frac{1}{\theta^N} \mathbb F_q[t_1, \ldots, t_n][[\frac{1}{\theta}]].$ Since we have $\mathbb T_n(K_{\infty})= A[t_1, \ldots, t_n]\oplus \frac 1 \theta \mathbb F_q[t_1, \ldots, t_n][[\frac{1}{\theta}]],$ we deduce that $H(\varphi/A[t_1, \ldots, t_n])$ is a finitely generated $\mathbb F_q[t_1, \ldots, t_n]$-module and also an $A[t_1, \ldots, t_n]$-module via $\varphi.$ Furthermore, the $k$-vector space spanned by $\mathbb T_n(K_{\infty})$ is dense in $\mathbb K_{\infty},$ thus the natural inclusion $\mathbb T_n(K_{\infty})\subset \mathbb K_{\infty}$ induces an isomorphism of $\mathbb A$-modules:
\begin{eqnarray}\label{EqS2-2}H(\varphi/A[t_1, \ldots, t_n])\otimes_{\mathbb F_q[t_1, \ldots, t_n]}k\simeq H(\varphi/\mathbb A).\end{eqnarray}
However, note that the map $H(\varphi/A[t_1, \ldots, t_n])\rightarrow H(\varphi/\mathbb A)$ induced by the inclusion $ \mathbb T_n(K_{\infty})\subset \mathbb K_{\infty}$  is not injective in general.\par
${}$\par
Let $\widetilde{\varphi}$ be the canonical $z$-deformation of $\varphi,$ i.e. : $\widetilde{\varphi}: \widetilde{\mathbb A}\rightarrow \widetilde{\mathbb A}\{ \tau\}$ is the morphism of $k(z)$-algebras given by:
$$\widetilde{\varphi}_{\theta}= \sum_{j=0}^r\alpha_j b_j(t_1)\cdots b_j(t_n)z^j \tau^j.$$ 
There exists a unique element $\exp_{\widetilde{\varphi}}\in \widetilde{\mathbb K_{\infty}}\{\{ \tau\}\}$ such that $\exp_{\widetilde{\varphi}}\equiv 1\pmod{\tau}$ and:
$$\exp_{\widetilde{\varphi}} \theta = \widetilde{\varphi}_{\theta} \exp_{\varphi}.$$
We have:
$$\exp_{\widetilde{\varphi}}=\sum_{j\geq 0} b_j(t_1)\cdots b_j(t_n)e_jz^j \tau^j.$$
By \cite{DEM}, Proposition 2.3, $\exp_{\widetilde{\varphi}}$ converges on $\widetilde{\mathbb K_{\infty}}$ and furthermore $\exp_{\widetilde{\varphi}}(\mathbb T_{n,z}(K_{\infty}))Ê\subset \mathbb T_{n,z}(K_{\infty}).$ Let's set:
$$U(\widetilde{\varphi}/\widetilde{\mathbb A})=\{ x\in \widetilde{\mathbb K_{\infty}}, \exp_{\widetilde{\varphi}}(x)\in \widetilde{\mathbb A}\}.$$
Then, by \cite{DEM}, Proposition 2.6, $U(\widetilde{\varphi}/\widetilde{\mathbb A})$  is an $\widetilde{\mathbb A}$-lattice in $\widetilde{\mathbb K_{\infty}}.$  Let's also set:
$$H(\widetilde{\varphi}/\widetilde{\mathbb A})= \frac{\widetilde{\mathbb K_{\infty}}}{\widetilde{\mathbb A}+\exp_{\widetilde{\varphi}}(\widetilde{\mathbb K_{\infty}})}.$$
Then, by \cite{DEM}, Proposition 2.6, $H(\widetilde{\varphi}/\widetilde{\mathbb A})$ is a finite dimensional $k(z)$-vector space. By \cite{DEM}, Theorem 2.7, the following infinite product converges in $\widetilde{\mathbb  K_{\infty}}:$
$$\mathcal L(\widetilde{\varphi}/\widetilde{\mathbb A}):=\prod_P\frac{[\frac{\widetilde{\mathbb A}}{P\widetilde{\mathbb A}}]_{\widetilde{\mathbb A}}}{[\widetilde{\varphi}(\frac{\widetilde{\mathbb A}}{P\widetilde{\mathbb A}})]_{\widetilde{\mathbb A}}}\in \widetilde{\mathbb K_{\infty}}^\times,$$
where the product is over the monic irreducible polynomials in $A.$ Since $\widetilde{\varphi}_{\theta} \in A[t_1, \ldots, t_n,z]\{Ê\tau\},$ we  have:
$$\mathcal L(\widetilde{\varphi}/\widetilde{\mathbb A})\in \mathbb T_{n,z}(K_{\infty})^\times.$$
We have the following formula (\cite{DEM}, theorem 2.7):
\begin{eqnarray}\label{EqS2-3}\mathcal L(\widetilde{\varphi}/\widetilde{\mathbb A}) \widetilde{\mathbb A}= [H(\widetilde{\varphi}/\widetilde{\mathbb A})]_{\widetilde{\mathbb A}}U(\widetilde{\varphi}/\widetilde{\mathbb A})\end{eqnarray} 
Let's set:
$$U(\widetilde{\varphi}/A[t_1, \ldots, t_n,z])=\{ x\in \mathbb T_{n,z}(K_{\infty}), \exp_{\widetilde{\varphi}}(x)\in A[t_1, \ldots, t_n,z]\},$$
and:
$$H(\widetilde{\varphi}/A[t_1, \ldots, t_n,z])= \frac{\mathbb T_{n,z}(K_{\infty})}{A[t_1, \ldots, t_n,z]+\exp_{\widetilde{\varphi}}(\mathbb T_{n,z}(K_{\infty}))}.$$
As for $\varphi,$ $H(\widetilde{\varphi}/A[t_1, \ldots, t_nz])$ is a finitely generated $\mathbb F_q[t_1, \ldots, t_n, z]$-module and the inclusion $\mathbb T_{n,z}(K_{\infty}) \subset \widetilde{\mathbb K_{\infty}}$ induces  an isomorphism of $\widetilde{A}$-modules:
$$H(\widetilde{\varphi}/A[t_1, \ldots, t_nz])\otimes_{\mathbb F_q[t_1, \ldots, t_n,z]}k(z)\simeq H(\widetilde{\varphi}/ \widetilde{\mathbb A}).$$
\begin{proposition}\label{PropositionS2-1}  We have:
$$H(\widetilde{\varphi}/ \widetilde{\mathbb A})=\{0\}.$$ 
Furthermore:
$$U(\widetilde{\varphi}/ A[t_1, \ldots, t_n,z])= \mathcal L(\widetilde{\varphi}/\widetilde{\mathbb A}) A[t_1, \ldots, t_n,z].$$
\end{proposition} 
\begin{proof} The proof of the first assertion is similar to that of Proposition \ref{propositionSt2}. For the convenience of the reader, we give its proof.\par
We have:
$$\mathbb T_{n,z}(K_{\infty})= \mathbb T_n(K_{\infty})\oplus z\mathbb T_{n,z}(K_{\infty}).$$
Since $\forall x\in \mathbb T_{n,z}(K_{\infty}), \exp_{\widetilde{\varphi}}(x)\equiv x\mid_{z=0} \pmod{ z\mathbb T_{n,z}(K_{\infty})},$ we deduce that the multiplication by $z$ gives rise to  an exact sequence of $A[t_1, \ldots, t_n,z]$-modules:
$$0\rightarrow M \rightarrow H(\widetilde{\varphi}/A[t_1, \ldots, t_n,z]) \rightarrow H(\widetilde{\varphi}/A[t_1, \ldots, t_n,z]) \rightarrow 0,$$
where $M=\{ x\in H(\widetilde{\varphi}/A[t_1, \ldots, t_n,z]), zx=0\}.$\par
 \noindent This implies that $H(\widetilde{\varphi}/A[t_1, \ldots, t_n,z])\otimes_{\mathbb F_q[t_1, \ldots, t_n]}k$ is a finitely generated and torsion $k[z]$-module. Thus we get the desired assertion.\par
We deduce the second assertion  by formula (\ref{EqS2-3}) and a similar proof to that  of Lemma \ref{LemmaS2-1}.
\end{proof}
We set:
\begin{eqnarray}\label{unitz} u_{\phi}(t_1, \ldots, t_n; z)=\exp_{\widetilde{\varphi}}(\mathcal L(\widetilde{\varphi}/\widetilde{\mathbb A}))\in A[t_1, \ldots, t_n,z].\end{eqnarray}
Since $\exp_{\widetilde{\phi}}: \mathbb T_{n,z}(K_{\infty})\rightarrow \mathbb T_{n,z}(K_{\infty})$ is an injective morphism of $\mathbb F_q[t_1, \ldots, t_n,z]$-modules, we have that:
$$u_{\phi}(t_1, \ldots, t_n; z)\not =0.$$
\begin{remark}\label{remark2} Select $\widetilde{\pi}$ a period of the Carlitz module (well-defined modulo $\mathbb F_q^\times$). Let $h\in \{ 0, \ldots, q-2\}$ such that $h\equiv n\pmod{q-1}.$ Observe that $\widetilde{\pi}^hK_{\infty}$ is an $A$-module via $\phi$ and that $\exp_{\phi}(\widetilde{\pi}^h K_{\infty})\subset \widetilde{\pi}^h K_{\infty}.$ Therefore if $\exp_{\phi}$ is not injective on $\widetilde{\pi}^hK_{\infty},$ there exists $\pi_{\phi, h} \in \widetilde{\pi}^h K_{\infty}\setminus\{0\}$ such that:
$${\rm Ker}(\exp_{\phi}:\widetilde{\pi}^hK_{\infty}\rightarrow  \widetilde{\pi}^h K_{\infty})=A\pi_{\phi,h}.$$
Now observe that:
$$\frac{\omega(t_1)\cdots \omega(t_n)}{\widetilde{\pi}^h} \in \mathbb T_n(K_{\infty}),$$
where for $j=1, \ldots, n,$ we have set:
$$\omega(t_j)= \exp_C(\frac{\widetilde{\pi}}{\theta})\prod_{k\geq 0} (1-\frac{t_j}{\theta^{q^k}})^{-1} \in \widetilde{\pi}\mathbb T_n(K_{\infty})^\times.$$
Furthermore, for $f\in \mathbb T_n(K_{\infty}),$ we have:
$$\omega(t_1)\cdots \omega(t_n)\exp_{\varphi}(f)= \exp_{\phi}(\omega(t_1)\cdots \omega(t_n)f).$$
Therefore, $\exp_{\varphi}$ is injective on $\mathbb T_n(K_{\infty})$ if and only if $\exp_{\phi}$ is injective on $ \widetilde{\pi}^h K_{\infty}.$ Otherwise :
$${\rm Ker}( \exp_{\varphi}: \mathbb T_n(K_{\infty})\rightarrow \mathbb T_n(K_{\infty}))= A[t_1, \ldots, t_n] \frac{\pi_{\phi,h}}{\omega(t_1)\cdots \omega(t_n)}.$$
We assume now that $\exp_{\varphi}$ is not injective on $\mathbb T_n(K_{\infty}).$ 
Observe that:
$$\frac{\pi_{\phi, h}}{\omega(t_1)\cdots \omega(t_n)}\in \mathbb T_n(K_{\infty})^\times.$$
By similar arguments than those used in the proof of Lemma \ref{LemmaS2-1}, we deduce that there exists $a\in A[t_1, \ldots, t_n]$ which is monic as a polynomial in $\theta$ such that:
$$U(\varphi/A[t_1, \ldots, t_n])=A[t_1, \ldots, t_n] \frac{\pi_{\phi,h}}{a\omega(t_1)\cdots \omega(t_n)}.$$
 Therefore, there exists $\lambda \in \mathbb F_q^\times$ such that:
 $$\frac{\mathcal L(\varphi/\mathbb A)}{[H(\varphi/\mathbb A)]_{\mathbb A}}= \lambda \frac {\pi_{\phi,h}}{a\omega(t_1)\cdots \omega(t_n)}.$$
 \end{remark}

\subsection{Evaluation at $z=1$}\label{Evaluationatz=1}${}$\par
We keep the hypothesis and notation of section \ref{Basicproperties}. Let's consider the continuous morphism of $\mathbb F_q[t_1, \ldots,t_n]$-algebras $ev: \mathbb T_{n,z}(K_{\infty})\rightarrow \mathbb T_n(K_{\infty})$ given by:
$$\forall f\in \mathbb T_{n,z}(K_{\infty}), ev(f)=f\mid_{z=1}.$$
Then:
$${\rm Ker}\, ev= (z-1) \mathbb T_{n,z}(K_{\infty}).$$
Furthermore, for $f\in \mathbb T_{n,z}(K_{\infty}),$ we have:
$$ev(\exp_{\widetilde{\varphi}}(f))= \exp_{\varphi}(ev(f)),$$
$$ev( \widetilde{\varphi}_{\theta}(f))=\varphi_{\theta}(ev(f)).$$
This implies that $ev$ gives rise to an isomorphism of $A[t_1, \ldots, t_n]$-modules:
$$\frac{H(\widetilde{\varphi}/A[t_1, \ldots, t_n,z])}{(z-1)H(\widetilde{\varphi}/A[t_1, \ldots, t_n,z])}\simeq H(\varphi/A[t_1, \ldots, t_n]).$$ 
Let's set:
$$U_{St}(\varphi/A[t_1, \ldots, t_n])=ev(U(\widetilde{\varphi}/A[t_1, \ldots, t_n,z])).$$
Clearly $U_{St}(\varphi/A[t_1, \ldots, t_n])$ is a sub-$A[t_1, \ldots, t_n]$-module of $U(\varphi/A[t_1, \ldots,t_n]),$ and by Proposition \ref{PropositionS2-1}, we have:
$$U_{St}(\varphi/A[t_1, \ldots, t_n])= \mathcal L(\varphi/\mathbb A)A[t_1, \ldots, t_n].$$
In particular, by Lemma \ref{LemmaS2-1}, we have an isomorphism of $A[t_1, \ldots, t_n]$-modules:
$$\frac{U(\varphi/A[t_1, \ldots,t_n])}{U_{St}(\varphi/A[t_1, \ldots, t_n])}\simeq \frac{A[t_1, \ldots, t_n]}{[H(\varphi/\mathbb A)]_{\mathbb A}A[t_1, \ldots, t_n]}.$$
Let $\alpha: \mathbb T_{n,z}(K_{\infty})\rightarrow \mathbb T_{nz}(K_{\infty})$ be the morphism of $\mathbb F_q[t_1, \ldots, t_n,z]$-modules given by:
\begin{eqnarray}\label{alpha}\forall f\in \mathbb T_{n,z}(K_{\infty}), \alpha(f) =\frac{\exp_{\widetilde{\varphi}}(f)-\exp_{\varphi}(f)}{z-1}.\end{eqnarray}
Let's set:
$$H(\widetilde{\varphi}/A[t_1, \ldots, t_n,z])[z-1]=\{ x\in H(\widetilde{\varphi}/A[t_1, \ldots, t_n,z]), (z-1)x=0\}.$$
Observe that $H(\widetilde{\varphi}/A[t_1, \ldots, t_n,z])[z-1]$ is an $A[t_1, \ldots, t_n]$-module via $\widetilde{\varphi}.$
\begin{proposition}\label{PropositionS2-2} The map $\alpha$ induces an isomorphism of $A[t_1, \ldots, t_n]$-modules:
$$\frac{U(\varphi/A[t_1, \ldots,t_n])}{U_{St}(\varphi/A[t_1, \ldots, t_n])}\simeq H(\widetilde{\varphi}/A[t_1, \ldots, t_n,z])[z-1].$$
\end{proposition}
\begin{proof} 
The  above result can be proved along the same lines as that of the proof of Proposition \ref{propositionSt3}. For the convenience of the reader, we give a proof. The map $\exp_{\widetilde{\varphi}}$ induces the   exact sequence of $A[t_1, \ldots, t_n, z]$-modules :
$$0\rightarrow M_1\rightarrow M_2 \rightarrow M_3 \rightarrow 0.$$
where
$$M_1:= \frac{\mathbb T_{n,z}(K_{\infty})}{U(\widetilde{\varphi}/A[t_1, \ldots, t_n,z])}; \ M_2:=\frac{\mathbb T_{n,z}(K_{\infty})}{A[t_1, \ldots, t_n,z]}; \ M_3:=H(\widetilde{\varphi}/A[t_1, \ldots, t_n,z]).$$
Thus, by the snake Lemma, we get a long exact sequence:
\begin{align*}
0\longrightarrow M_1[z-1]\longrightarrow M_2[z-1]\longrightarrow M_3[z-1]\longrightarrow\\ \longrightarrow \frac{M_1}{(z-1)M_1}\longrightarrow \frac{M_2}{(z-1)M_2} \longrightarrow \frac{M_3}{(z-1)M_3}\longrightarrow 0.
\end{align*}
Now, we observe that the map $ev$ induces  isomorphisms of $A[t_1, \ldots, t_n]$-modules:
$$\frac{M_1}{(z-1)M_1} \simeq \frac{\mathbb T_n(K_{\infty})}{U_{St}(\varphi/A[t_1, \ldots, t_n])},$$
$$\frac{M_2}{(z-1)M_2}\simeq \frac{\mathbb T_n(K_{\infty})}{A[t_1, \ldots, t_n]},$$
$$\frac{M_3}{(z-1)M_3}\simeq H(\varphi /A[t_1, \ldots, t_n]).$$
Observe also  that:
$$(z-1) \mathbb T_{n,z}(K_{\infty})\cap A[t_1, \ldots, t_n,z]= (z-1)A[t_1, \ldots, t_n,z].$$
Therefore:
$$M_2[z-1]=\{0\}.$$
Let $\delta: M_3[z-1]\rightarrow \frac{M_1}{(z-1)M_1}$ be the map given by the snake Lemma. Let $x\in M_3[z-1]$ and select $y\in \mathbb T_{n,z}(K_{\infty})$ such that $x\equiv y\pmod{A[t_1, \ldots, t_n,z]+\exp_{\widetilde{\varphi}}(\mathbb T_{n,z}(K_{\infty})}.$ Then:
$$(z-1)y=\exp_{\widetilde{\varphi}}(m)+ h,$$
where $m\in \mathbb T_{n,z}(K_{\infty})$ and $h\in A[t_1, \ldots, t_n,z].$ We get:
$$\delta(x)= m\pmod{(z-1)\mathbb T_{n,z}(K_{\infty})+U(\widetilde{\varphi}/A[t_1, \ldots, t_n,z]) }.$$
But we have:
$$ev((z-1)y)=\exp_{\varphi}(ev(m))+ev(h)=0.$$
Thus $\beta=ev(m)\in U(\varphi/A[t_1, \ldots, t_n]). $ Thus:
$$\delta(x)\equiv \beta\pmod{(z-1)\mathbb T_{n,z}(K_{\infty})+U(\widetilde{\varphi}/A[t_1, \ldots, t_n,z]) }.$$
Furthermore:
$$(z-1)y\equiv \exp_{\widetilde{\varphi}}(\beta)-\exp_{\varphi}(\beta) \pmod{(z-1)\exp_{\widetilde{\varphi}}( \mathbb T_{n,z}(K_{\infty}))+(z-1) A[t_1, \ldots, t_n,z]},$$
and therefore:
$$\alpha(\beta)\equiv x\pmod{\exp_{\widetilde{\varphi}}(\mathbb T_{n,z}(K_{\infty}))+ A[t_1, \ldots,t_n,z]}.$$
Therefore, we get an exact sequence of $A[t_1, \ldots, t_n]$-modules:
\begin{align*}
0\longrightarrow H(\widetilde{\varphi}/A[t_1, \ldots, t_n,z])[z-1] \longrightarrow \frac{\mathbb T_n(K_{\infty})}{U_{St}(\varphi/A[t_1, \ldots, t_n])} \longrightarrow\\ \longrightarrow\frac{\mathbb T_n(K_{\infty})}{A[t_1, \ldots, t_n]} \longrightarrow  H(\varphi /A[t_1, \ldots, t_n])\longrightarrow 0,
\end{align*}
where the exact sequence $\frac{\mathbb T_n(K_{\infty})}{U_{St}(\varphi/A[t_1, \ldots, t_n])} \rightarrow\frac{\mathbb T_n(K_{\infty})}{A[t_1, \ldots, t_n]} \rightarrow  H(\varphi /A[t_1, \ldots, t_n])\rightarrow 0$  is induced by $\exp_{\varphi},$ and where the map $H(\widetilde{\varphi}/A[t_1, \ldots, t_n,z])[z-1] \rightarrow \frac{\mathbb T_n(K_{\infty})}{U_{St}(\varphi/A[t_1, \ldots, t_n])} $  sends $x\in H(\widetilde{\varphi}/A[t_1, \ldots, t_n,z])[z-1]$ to $\beta$ for some $\beta \in U(\varphi/A[t_1, \ldots, t_n])$ such that $\alpha(\beta)\equiv x\pmod{\mathbb T_{n,z}(K_{\infty})+ A[t_1, \ldots,t_n,z]}.$
\end{proof}

Let $\frac{d}{dz}: \mathbb T_{n,z}(K_{\infty})\{\{\tau\}\} \rightarrow \mathbb T_{n,z}(K_{\infty})\{\{\tau\}\}$ be the map defined by:
$$\frac{d}{dz}(\sum_{i\geq 0} f_i \tau^i)=\sum_{i\geq 0}\frac{d}{dz}(f_i) \tau^i, f_i\in \mathbb T_{n,z},$$
where $\frac{d}{dz}(\sum_{i\geq 0} x_i z^i)= \sum_{i\geq 1} ix_i z^{i-1}, x_i \in \mathbb T_n(K_{\infty}).$ We set:
\begin{eqnarray}\label{exp1} \exp_{\varphi}^{(1)}=\sum_{j\geq 1} jb_j(t_1)\cdots b_j(t_n)e_j  \tau^j.\end{eqnarray}
\begin{corollary}\label{CorollaryS2-1} The map $\exp_{\varphi}^{(1)}$ induces a morphism of $A[t_1, \ldots, t_n]$-modules:
$$\frac{U(\varphi/A[t_1, \ldots, t_n])}{U_{St}(\varphi/A[t_1, \ldots, t_n])}\rightarrow H(\varphi/A[t_1, \ldots, t_n]).$$
Let $H^{(1)}(\varphi/A[t_1, \ldots, t_n])$ be the image of the above map. The kernel of the above map is isomorphic as an $A[t_1, \ldots, t_n]$-module to:
$$\frac{H(\widetilde{\varphi}/A[t_1, \ldots, t_n,z])[(z-1)^2]}{H(\widetilde{\varphi}/A[t_1, \ldots, t_n,z])[z-1] },$$
where  $H(\widetilde{\varphi}/A[t_1, \ldots, t_n,z])[(z-1)^2]=
\{x\in H(\widetilde{\varphi}/A[t_1, \ldots, t_n,z]), (z-1)^2x=0\}.$
In particular if $H(\widetilde{\varphi}/A[t_1, \ldots, t_n,z])[(z-1)^2]=H(\widetilde{\varphi}/A[t_1, \ldots, t_n,z])[z-1],$ then we have an isomorphism of $A[t_1, \ldots,   t_n]$-modules:
$$H^{(1)}(\varphi/A[t_1, \ldots, t_n])\simeq \frac{A[t_1, \ldots, t_n]}{[H(\varphi/\mathbb A)]_{\mathbb A} A[t_1, \ldots, t_n]}.$$
Furthermore, in this case, $\displaystyle\frac{H(\varphi/A[t_1, \ldots, t_n])}{H^{(1)}(\varphi/A[t_1, \ldots, t_n])}$ is a finitely generated and torsion $\mathbb F_q[t_1, \ldots, t_n]$-module.
\end{corollary}
\begin{proof} Observe that $ev$ induces a morphism of $A[t_1, \ldots, t_n]$-modules:
$$H(\widetilde{\varphi}/A[t_1, \ldots, t_n,z])[z-1]\rightarrow H(\varphi/A[t_1, \ldots, t_n]).$$
The kernel of the above map being:
$$(z-1)H(\widetilde{\varphi}/A[t_1, \ldots, t_n,z])[(z-1)^2].$$
Now observe that, by formula (\ref{alpha}), we have :
$$\forall \beta \in U(\varphi/A[t_1, \ldots, t_n]), 
ev\circ \alpha (\beta) = \exp_{\varphi}^{(1)}(\beta).$$
It remains to apply Proposition \ref{PropositionS2-2} to obtain the first three assertions. The last  assertion is a  direct consequence of  (\ref{EqS2-2})
\end{proof}


\begin{question} Let $n\geq 1.$ ${}$\par
\noindent 1) Is it true that $H(\widetilde{\varphi}/A[t_1, \ldots, t_n,z])[(z-1)^2]=H(\widetilde{\varphi}/A[t_1, \ldots, t_n,z])[z-1]?$\par
\noindent 2) Is $H(\varphi/\mathbb A)$  a cyclic $\mathbb A$-module?\par
Observe that by the above Corollary if 1) is true then   2) is true. Also observe that 1) has a positive answer if and only if $\frac{H(\varphi/A[t_1, \ldots, t_n])}{H^{(1)}(\varphi/A[t_1, \ldots, t_n])}$ is a finitely generated and torsion $\mathbb F_q[t_1, \ldots, t_n]$-module. We are going to show in the next section  that 1) is true  when $\phi$ is the Carlitz module.
\end{question}
\subsection{The case of the Carlitz module}$\label{Thecaseofthecarlitzmodule}{}$\par
In this paragraph $\phi=C$ is the Carlitz module, $\varphi$ is the canonical deformation of $C$ over $\mathbb T_n(K_{\infty}),$ and $\widetilde{\varphi}$ is the canonical $z$-deformation of $\varphi.$
\subsubsection{Some  results on  units}${}$\par
Observe that in our case (by \cite{APTR}, Proposition 5.9):
$$ \mathcal L(\widetilde{\varphi}/\widetilde{\mathbb A})=\sum_{a\in A_+} \frac{a(t_1)\cdots a(t_n)}{a} z^{\deg_{\theta} a}\in \mathbb T_{n,z}(K_{\infty})^\times,$$
and:
$$\mathcal L(\varphi/\mathbb A)=ev(\mathcal L(\widetilde{\varphi}/\mathbb A))= \sum_{a\in A_+} \frac{a(t_1)\cdots a(t_n)}{a} \in \mathbb T_{n}(K_{\infty})^\times.$$

Recall that we have  set:
\begin{eqnarray}\label{unit}u_C(t_1, \ldots, t_n;z):= \exp_{\widetilde{\varphi}}(\mathcal L(\widetilde{\varphi}/\widetilde{\mathbb A}))\in A[t_1, \ldots, t_n,z].\end{eqnarray}
\begin{lemma}\label{LemmaS2-2} If $n\in \{ 0, \ldots, q-1\},$ then:
$$u_C(t_1, \ldots, t_n; z)= 1.$$
\end{lemma}
\begin{proof} Note that:
$$\forall j\geq 0, v_{\infty} (\frac{b_j(t_1)\cdots b_j(t_n)}{D_j} z^j )=jq^j-n\frac{q^j-1}{q-1},$$
and
$$v_{\infty}(\mathcal L(\widetilde{\varphi}/\widetilde{\mathbb A})-1)\geq 1.$$
Thus, under the assumption of the Lemma, a direct computation gives:
$$v_{\infty}(\exp_{\widetilde{\varphi}}(\mathcal L(\widetilde{\varphi}/\widetilde{\mathbb A}))-1)\geq 1.$$
Since $u_C(t_1, \ldots, t_n;z)\in A[t_1, \ldots, t_n,z],$ we get the desired result.
\end{proof}
\begin{proposition}\label{PropositionS2-3} 
\noindent Let $n\geq q.$ Write $n=q+r+\ell (q-1),$ $\ell \in \mathbb N,$ $r\in \{ 0, \ldots, q-2\}.$Then $u_C(t_1, \ldots, t_n;z)$ viewed as a polynomial in $\theta$ is of degree $n(\frac{q^{\ell}-1}{q-1})-\ell q^{\ell}$ if $r=0$ and $n(\frac{q^{\ell +1}-1}{q-1})-(\ell +1)q^{\ell +1}$ if $r\not =0.$ Furthermore the leading coefficient of $u_C(t_1,\ldots, t_n;z)$ viewed as a polynomial in $\theta $ is :
$$ {\rm if }\, r=0, (-1)^{\ell}z^{\ell}(1-z).$$
$${\rm if}\, r\not =0, (-1)^{n(\ell +1)}z^{\ell +1}     .$$
\end{proposition}
\begin{proof} We prove the result for $r=0,$ the proof of the Proposition being similar in the remaining cases.  For $i\geq 0,$ observe that $\tau^i(\mathcal L(\widetilde{\varphi}/\widetilde{\mathbb A}))$ and $
(-1)^i \frac{b_i(t_1)\cdots b_i(t_n)}{D_i}$ are  monic. We have:
$$\forall i\geq 0, v_{\infty}(\frac{b_i(t_1)\cdots b_i(t_n)z^i}{D_i}\tau^i (\mathcal L(\widetilde{\varphi}/\widetilde{\mathbb A})))=i q^i-n(\frac{q^i-1}{q-1}) .$$
Set, for $i\geq 0,$  $\alpha_i= (-1)^{ i}\frac{b_i(t_1)\cdots b_i(t_n)}{D_i}\tau^i (\mathcal L(\widetilde{\varphi}/\widetilde{\mathbb A})),$  then $\alpha_i$ is  monic. We have:
$$u(t_1, \ldots, t_n;z)=\sum_{i\geq 0}(-1)^{ i}\alpha_i z ^i.$$
For $i\geq 0,$ we have :
$$((i+1) q^{i+1}-n(\frac{q^{i+1}-1}{q-1})) -(i q^i-n(\frac{q^i-1}{q-1}))=q^i(q+i(q-1)-n)    .$$
Recall that  $n=q+\ell (q-1), $ $\ell \in \mathbb N.$ We get :
$$v_{\infty}(u_C(t_1, \ldots, t_n;z)-((-1)^{\ell }\alpha_{\ell} z^{\ell }+ (-1)^{(\ell+1)}\alpha_{\ell +1}z^{\ell +1}))>v_{\infty}(u(t_1, \ldots, t_n;z)).$$
But $\alpha_{\ell}$ and $\alpha_{\ell +1}$ are monic, thus:
$$v_{\infty}(u_C(t_1, \ldots, t_n;z)-(z^{\ell}(1-z)(-1)^{ \ell}\theta^{n(\frac{q^{\ell}-1}{q-1})- \ell q^{\ell}})>v_{\infty}(u(t_1, \ldots, t_n;z)).$$
This proves the  assertion. \end{proof}

\begin{lemma}\label{LemmaS2-3} Let $n\geq q,$ $n\equiv 1\pmod{q-1}.$ Then:
$$u_C(t_1, \ldots, t_n; z)\in (z-1)A[t_1, \ldots, t_n,z].$$
\end{lemma}
\begin{proof} By \cite{APTR}, Lemma 6.8 and Remark \ref{remark2},  $\exp_{\varphi}: \mathbb T_n(K_{\infty})\rightarrow \mathbb T_n(K_{\infty})$ is not injective if and only if $n \equiv 1\pmod{q-1}.$ If $n\equiv 1\pmod{q-1},$ then:
$${\rm Ker}(\exp_{\varphi}: \mathbb T_n(K_{\infty})\rightarrow \mathbb T_n(K_{\infty}))=A[t_1, \ldots, t_n]\frac{\widetilde{\pi}}{\omega(t_1)\cdots \omega(t_n)}.$$
Furthermore, if $n\geq q,$ $n\equiv 1\pmod{q-1},$  we have (\cite{APTR}, Proposition 7.2):
\begin{eqnarray}\label{kernel}U(\varphi/A[t_1, \ldots, t_n])= {\rm Ker}\exp_{\varphi}\mid_{\mathbb T_n(K_{\infty})}=A[t_1, \ldots, t_n]\frac{\widetilde{\pi}}{\omega(t_1)\cdots \omega(t_n)}.\end{eqnarray}
Recall that  $\mathcal L(\varphi/\mathbb A)A[t_1, \ldots, t_n]=U_{St}(\varphi/A[t_1, \ldots, t_n])\subset U(\varphi/A[t_1, \ldots, t_n])  .$ For $n\geq q,$ $n\equiv 1\pmod{q-1},$ we get by (\ref{unit}):
$$u_C(t_1, \ldots, t_n; 1)=ev(u_C(t_1, \ldots,t_n; z))= \exp_{\varphi}(\mathcal L(\varphi/\mathbb A))=0.$$
\end{proof}
Let's set:
$$\mathbb B(t_1, \ldots, t_n)=[H(\varphi/\mathbb A)]_{\mathbb A} \in A[t_1, \ldots, t_n].$$
We recall that if $n\leq 2q-2,$ by \cite {APTR}, Remark 5.3,  we have $H(\phi/A[t_1, \ldots, t_n])=\{ 0\}$ and therefore $\mathbb B(t_1, \ldots, t_n)=1.$
\subsubsection{The case $n\not \equiv 1\pmod{q-1}$}${}$\par
In this section, we assume that $n\not \equiv 1\pmod{q-1}.$ 
\begin{theorem}\label{TheoremS2-1}
$H(\varphi/A[t_1, \ldots, t_n])$ is a finitely generated and torsion $\mathbb F_q[t_1, \ldots, t_n]$-module. In particular:
$$\mathbb B(t_1, \ldots,t_n)=1.$$
\end{theorem}
\begin{proof} ${}$\par
Let $r\in \{ 2, \ldots, q-1\}$ 
such that $n\equiv r\pmod{q-1}.$ We can assume $n>r.$ Set :
$$N=\{ x\in \mathbb T_n(K_{\infty}), v_{\infty}(x)\geq \frac{n-r}{q-1}\}.$$
Then, since $\frac{n-r}{q-1}>\frac{n-q}{q-1},$  we have :
$$N=\exp_{\varphi}(N)\subset \exp_{\varphi}(\mathbb T_n(K_{\infty})).$$
Observe that:
$$\mathbb T_n(K_{\infty})= \frac{1}{\theta^{\frac{n-r}{q-1}-1}}A[t_1, \ldots, t_n]\oplus N.$$
Set :
$$E=\frac{\mathbb T_n(K_{\infty})}{A[t_1, \ldots, t_n]+N}.$$ 
Observe that $E$ is a free and finitely generated $\mathbb F_q[t_1, \ldots, t_n]$-module of rank $\frac{n-r}{q-1}-1.$ In order to prove the Theorem, it is enough to prove that there exists $\lambda \in \mathbb T_n(K_{\infty}),$ $v_{\infty}(\lambda)=0,$ such that if we write in $E:$
$$k\in \{ 1, \ldots, \frac{n-r}{q-1}-1\}, \, \exp_{\varphi}(\theta^{-k}\lambda)=\sum_{j=1}^{\frac{n-r}{q-1}-1} \beta_{j,k}(\lambda)\theta^{-j}, \, \beta_{j,k}(\lambda)\in \mathbb F_q[t_1, \ldots, t_n],$$
then $det((\beta_{j,k}(\lambda))_{j,k})\in \mathbb  F_q[t_1, \ldots,t_n]\setminus \{0\}.$\par
${}$\par

 Let $\psi:A\rightarrow \mathbb T_n(K_{\infty})\{Ê\tau \}$ be the morphism of $\mathbb F_q$-algebras given by :
$$\psi_{\theta}= (t_1-\theta)\cdots (t_r-\theta)\tau +\theta.$$
Let $\exp_{\psi}= 1+\sum_{j\geq 1} \frac{b_j(t_1)\cdots b_j(t_r)}{D_j}\tau^j .$ By Lemma \ref{LemmaS2-2}, we have:
$$\exp_{\psi}(\mathcal L(\psi/\mathbb A))=1.$$
Set:
$$\eta=\frac{1}{\omega(t_{r+1})\cdots \omega(t_n)}\in \mathbb T_n(K_{\infty})^\times.$$
We refer the reader to Remark \ref{remark2} for the definition of $\omega(t).$
Observe that $v_{\infty}(\eta)=\frac{n-r}{q-1}.$
We have :
$$\forall x \in \mathbb T_n(K_{\infty}), \exp_{\phi}(\eta x)=\eta\exp_{\psi}(x).$$ 
Thus, for $k\geq 0:$
$$\exp_{\phi}(\theta^k  \mathcal L(\psi/\mathbb A)\eta)=\eta {\psi_{\theta^k}(1)}.$$
We set $\lambda = (-\theta) ^{\frac{n-r}{q-1}} \mathcal L(\psi/\mathbb A) \eta,$ note that $\lambda\in \mathbb T_n(K_{\infty}),$ and $v_{\infty}(\lambda)=0.$ For $k\in \{ 1, \ldots, \frac{n-r}{q-1}-1\},$ recall that  we write in $E:$ 
$$\exp_{\phi}(\theta^{-k}\lambda)=\sum_{j=1}^{\frac{n-r}{q-1}-1} \beta_{j,k}(\lambda)\theta^{-j}, \, \beta_{j,k}(\lambda)\in \mathbb F_q[t_1, \ldots, t_n].$$
${}$\par
Let $ev_0:\mathbb T_n(K_{\infty})\rightarrow K_{\infty}$ be the surjective morphism of $\mathbb F_q$-algebras given by $ev_0(f)=f\mid_{t_1=\cdots= t_n=0}.$ Then $ev_0((-\theta) ^{\frac{n-r}{q-1}}\eta )=1.$  Let $\varrho : A\rightarrow A\{\tau\}$ be the morphism of $\mathbb F_q$-algebras given by:
$$ \varrho_{\theta}=\varphi_{\theta}\mid_{t_1=\cdots= t_n=0}= (-1)^n \theta^n \tau +\theta.$$
Then clearly:
$$\exp_{\varrho}=\exp_{\varphi}\mid_{t_1=\cdots= t_n=0}.$$
We have:
$$ \{ x\in K_{\infty}, v_{\infty}(x)\geq \frac{n-r}{q-1}\}=ev_0(N)\subset \exp_{\varrho}(K_{\infty}).$$
Furthermore $E':=\frac{K_{\infty}}{A\oplus ev_0(N)}$ is a finite dimensional $\mathbb F_q$-vector space of dimension $\frac{n-r}{q-1}-1.$ Therefore, for $k\in \{ 1, \ldots, \frac{n-r}{q-1}-1\},$ we have   in $E':$ $$\exp_{\varrho}(\theta^{-k})=\sum_{j=1}^{\frac{n-r}{q-1}-1} ev_0(\beta_{j,k}(\lambda))\theta^{-j}.$$
A direct computation yields for $k\geq 0$ :
$$ev_0(\exp_{\varphi}(\theta^k  \mathcal L(\psi/\mathbb A)\eta))\in (-\theta)^{-\frac{n-r}{q-1}}\theta^k+(-\theta)^{k+1-\frac{n-r}{q-1}}A.$$
 This implies: for $k,j\in \{ 1, \ldots, \frac{n-r}{q-1}-1\},$ $ev_0(\beta_{k,k}(\lambda))\not =0$ and $ev_0(\beta_{j,k}(\lambda))=0$ if $j>k.$
 Thus $\det((ev_0(\beta_{j,k}(\lambda))_{j,k})\in \mathbb F_q^\times,$ and therefore:
 $$det((\beta_{j,k}(\lambda))_{j,k})\in \mathbb  F_q[t_1, \ldots,t_n]\setminus \{0\}.$$
 The last assertion of the Theorem is then a consequence of (\ref{EqS2-2}).
  
\end{proof}
Let's set:
$$\mathcal M(\varphi/A[t_1, \ldots, t_n])=\{ \varphi_a(u_C(t_1, \ldots, t_n; 1)), a\in A[t_1, \ldots, t_n]\},$$
and its radical $\sqrt{\mathcal M(\varphi/A[t_1, \ldots,t_n])}$ is defined as
$$\{ b\in A[t_1, \ldots, t_n], \exists a\in A[t_1, \ldots, t_n]\setminus\{ 0\}, \varphi_a(b)\in \mathcal M(\phi/A[t_1, \ldots, t_n])\}.$$
Note that, for $n\geq 1,$ $\mathcal M(\varphi/A[t_1, \ldots, t_n])$ is a free $A[t_1, \ldots, t_n]$-module of rank one. We will need the following Lemma:
\begin{lemma}\label{LemmaS2-4}  Let $n\geq 2$ be an integer (we don't assume in this Lemma that $n\not \equiv 1\pmod{q-1}$). Let $u\in A[t_1, \ldots,t_n]\setminus \{0\}$ be such that its leading coefficient as a polynomial in $\theta$ is in $\mathbb F_q^\times.$ Let $c\in \mathbb A$ such that $\varphi_c(u)\in A[t_1, \ldots, t_n],$ then $c\in A[t_1, \ldots,t_n].$
\end{lemma}
\begin{proof}${}$\par
For $k\geq 1,$ let's write:
$$\varphi_{\theta^k}=\sum_{j=0}^k[\theta^k,j]\tau^j,\, [\theta^k,j]\in A[t_1, \ldots, t_n].$$
Then:
$${\rm deg}_{\theta}[\theta^k, j]= q^j(k-j)+n(\frac{q^j-1}{q-1}),$$
and the leading coefficient of $[\theta^k, j]$ (viewed as a polynomial in $\theta$) lies in $\mathbb F_q^\times.$ Observe that, since $n\geq 2,$ we have:
$$j=0, \ldots, k-1, {\rm deg}_{\theta}[\theta^k, k]>{\rm deg}_{\theta}[\theta^k, j].$$
This implies that if $a,b\in A[t_1, \ldots, t_n]\setminus\{0\},$ we have:
$${\rm deg}_{\theta} \varphi_a(b) = q^{{\rm deg}_{\theta}(a)}{\rm deg}_{\theta} b+ n(\frac{q^{\rm deg_{\theta} a}-1}{q-1}),$$
and the leading coefficient of $\varphi_{a}(b)$ viewed as a polynomial in $\theta$ is (up to an element in $\mathbb F_q^\times$) the leading coefficient of $a$ times the leading coefficient of $b.$\par
Write $c=\delta^{-1} a,$ $\delta\in \mathbb F_q[t_1, \ldots, t_n]\setminus\{0\}$ and $a\in A[t_1, \ldots, t_n].$
Write $a=\delta b+d,$  where $b,d\in A[t_1, \ldots, t_n]$ such that $d=0$ or the leading coefficient of $d$ is not divisible by $\delta.$ Then:
$$\varphi_{d}(u)\in \delta A[t_1, \ldots, t_n].$$
But, by the above discussion, if $d\not =0,$ the leading coefficient of $\varphi_d(u)$ is equal (up to an element in $\mathbb F_q^\times$) to the leading coefficient of $d.$ Thus $\delta$ would divide the leading coefficient of $d$ which is a contradiction and therefore $d=0$ and $c\in A[t_1, \ldots,t_n].$
\end{proof}

\begin{corollary}\label{CorollaryS2-2} 
Let $n\geq 1,$  $n\not \equiv 1\pmod{q-1}.$ We have:  $$\sqrt{\mathcal M(\varphi/A[t_1, \ldots,t_n])}=\mathcal M(\varphi/A[t_1, \ldots,t_n]).$$
\end{corollary}
\begin{proof}
Let $b\in \sqrt{\mathcal M(\varphi/A[t_1, \ldots, t_n])}\setminus\{ 0\}, $ then there exist $a,c \in A[t_1, \ldots, t_n]\setminus \{0\},$ such that:
$$\varphi_a(b)=\varphi_c(u_C(t_1, \ldots, t_n; 1)).$$
By Proposition \ref{PropositionS2-3}, the leading coefficient of $u_C(t_1, \ldots, t_n; 1)$ is in $\mathbb F_q^\times.$ By (\ref{unit}):
$$u_C(t_1, \ldots, t_n; 1)=\exp_{\varphi}(\mathcal L(\varphi/\mathbb A)).$$
Since (recall that $n\not \equiv 1\pmod{q-1}$) $\varphi_a: \mathbb K_{\infty}\rightarrow \mathbb K_{\infty}$ is an injective map, we get:
$$b=\exp_{\varphi} (\frac{c\mathcal L(\varphi/\mathbb A)}{a}).$$
 But by Theorem \ref{TheoremS2-1} and (\ref{EqS2-1}), we have:
$$\exp_{\varphi^{-1}}(\mathbb A)=\mathcal L(\varphi/\mathbb A)\mathbb A.$$
This implies:
$$\frac{c\mathcal L(\varphi/\mathbb A)}{a}\in \mathcal L(\varphi/\mathbb A)\mathbb A. $$
Thus $a$ must divide $c$ in $\mathbb A.$ By Lemma \ref{LemmaS2-4}, we have that $a$ divides $c$ in $A[t_1, \ldots, t_n]$ and therefore $b\in  \mathcal M(\varphi/A[t_1, \ldots, t_n]).$
\end{proof}
\begin{remark}\label{Remark3} We briefly treat the case $n=0.$ In this case $\varphi=C$ is the Carlitz module and $\mathcal M(C/A)=\{ C_a(1),\, a\in A\}.$ Observe that:
$$K_{\infty}=\log_C(1) A\oplus M_{\infty},$$
where $M_{\infty}=\{x\in K_{\infty}, v_{\infty}(x)\geq 1\}$ and where $\log_C\in K\{\{ \tau\}\}$ is the Carlitz logarithm (see \cite{GOS}, section 3.4). Therefore:
$$\exp_C(K_{\infty})=\mathcal M(C/A)\oplus M_{\infty}.$$
This implies:
$$A\cap \exp_C(K_{\infty})=\mathcal M(C/A).$$
Thus:
$$\sqrt{\mathcal M(C/A)}=\mathcal M(C/A).$$
Furthermore observe that if $q\geq 3,$ $\mathcal M(C/A)$ is a free $A$-module of rank one and if $q=2,$ we have an isomorphismm of $A$-modules:
$$\mathcal M(C/A)\simeq \frac {A}{(\theta(\theta+1))A}.$$
\end{remark}

\subsubsection{The case $n\equiv 1\pmod{q-1}$}\label{nodd}${}$\par
In this section, we assume that $n\geq q$ and $n\equiv 1\pmod{q-1}$. Recall that, by Lemma \ref{LemmaS2-3}, we have :
$$u_C(t_1, \ldots, t_n; 1)=0.$$
Therefore, for $n\geq q,$ we set:
$$u_C^{(1)}(t_1, \ldots, t_n; z)=\frac{d}{dz} u_C(t_1, \ldots, t_n; z)\in A[t_1, \ldots, t_n,z].$$ Observe  that, if we set $N= \{ x\in \mathbb T_n( K_{\infty}), v_{\infty}(x)>\frac{n-q}{q-1}\},$ then:
$$ \mathbb T_n(K_{\infty})= \frac{\widetilde{\pi}}{\omega(t_1)\cdots \omega (t_n)} A[t_1, \ldots, t_n]\oplus N.$$
Furthermore , by (\ref{kernel}), we get:
$$N=\exp_{\varphi}(N)=\exp_{\varphi}(\mathbb T_n(K_{\infty})).$$
\begin{proposition}\label{PropositionS2-4}${}$\par
We  set $E=\frac{\mathbb T_n(K_{\infty})}{N}$ viewed as an $A[t_1, \ldots, t_n]$-module via $\varphi.$ Then   $\exp_{\varphi}^{(1)}$ induces an injective morphism of $A[t_1, \ldots, t_n]$-modules : $$\exp_{\varphi}^{(1)}: E\hookrightarrow E.$$
\end{proposition}
\begin{proof} Recall that by (\ref{kernel}), we have:
$$U(\varphi/A[t_1, \ldots,t_n])=\frac{\widetilde{\pi}}{\omega(t_1)\cdots \omega (t_n)}A[t_1, \ldots, t_n].$$
By (\ref{unit}), we have:
$$\exp_{\varphi}^{(1)}(\mathcal L(\varphi/\mathbb A))\equiv u_C^{(1)}(t_1, \ldots, t_n; 1)\pmod{N}.$$
Furthermore, for $x\in U(\varphi/A[t_1, \ldots,t_n]),$ using the fact that $\exp_{\widetilde{\varphi}} \theta =\widetilde{\varphi}_{\theta} \exp_{\widetilde{\varphi}}$ and that   $\exp_{\varphi}(x)=0,$  we have:
$$\exp_{\varphi}^{(1)}(\theta x)\equiv \varphi_{\theta} (\exp_{\varphi}^{(1)}(x))\pmod{N}.$$
 Also observe  that:
$$\exp_{\varphi}^{(1)}(N)\subset N.$$
 Let $x\in \mathbb T_n(K_{\infty})$ such that there exists $a\in A[t_1, \ldots,t_n]\setminus\{0\},$ $a$ monic as a polynomial in $\theta,$  with $\varphi_a(x)=0$ in $E.$ Then, by \cite{APTR}, proof of Corollary 6.5,   we  have:
$$x\in \exp_{\varphi}(T_n(K_{\infty}))=N.$$
 By Proposition \ref{PropositionS2-3},we have that  $u_C^{(1)}(t_1, \ldots, t_n; 1)\in A[t_1, \ldots, t_n]\setminus \{ 0\},$ and since 
 $$\varphi_{\mathbb B(t_1, \ldots, t_n)}(\exp_{\varphi}^{(1)}(U(\varphi/A[t_1, \ldots, t_n]))\equiv u_C^{(1)}(t_1, \ldots, t_n; 1)A[t_1, \ldots, t_n])\pmod{N},$$
 and $A[t_1, \ldots,t_n]\cap N=\{0\},$   we deduce the assertion of the Proposition.
\end{proof}

\begin{lemma}\label{LemmaS2-5} 
 We have:
$$\exp_{\varphi}^{(1)}(\frac{\widetilde{\pi}}{\omega(t_1)\cdots \omega(t_n)})\equiv 
\frac{(-1)^{\frac{n-q}{q-1}}}{\theta^{\frac{n-q}{q-1}}}\pmod{ N}.$$
\end{lemma}
\begin{proof} We first notice  that :
$$v_{\infty}(\frac{\widetilde{\pi}}{\omega(t_1)\cdots \omega(t_n)} -\frac{(-1)^{\frac{n-q}{q-1}}}{\theta^{\frac{n-q}{q-1}}})> \frac{n-q}{q-1}.$$
For $i\geq 2,$ we also have that:
$$v_{\infty}( \frac{b_i(t_1)\cdots b_i(t_n)\tau^i(\frac{\widetilde{\pi}}{\omega(t_1)\cdots \omega(t_n)}  )}{D_i})>\frac{n-q}{q-1}.$$
Furthermore, observe that:
$$ \tau (\frac{\widetilde{\pi}}{\omega(t_1)\ldots \omega(t_n)})=\frac{\widetilde{\pi}^q}{(t_1-\theta)\cdots (t_n-\theta) \omega(t_1)\cdots \omega(t_n)}.$$
Finally, observe that :
$$v_{\infty}(\frac{\widetilde{\pi}^q}{(\theta^q-\theta)\omega(t_1)\cdots \omega(t_n)}-\frac{(-1)^{\frac{n-q}{q-1}}}{\theta^{\frac{n-q}{q-1}}})> \frac{n-q}{q-1}.$$
\end{proof}

\begin{lemma}\label{LemmaS2-6}  Let $a\in A[t_1, \ldots, t_n]\setminus\{ 0\},$ $a$ monic as a polynomial in $\theta.$  If 
$$\varphi_{a}(\frac{1}{\theta^{\frac{n-q}{q-1}}})\in A[t_1, \ldots, t_n]+\exp_{\varphi}(\mathbb T_n(K_{\infty})),$$
then $deg_{\theta}(a)\geq \frac{n-q}{q-1}.$
\end{lemma}
\begin{proof} Write $a=\theta^r+\sum_{i=0}^{r-1} a_i \theta ^i, $ $ a_i \in \mathbb F_q[t_1, \ldots, t_n].$  Set $b=a\mid_{t_j=0, j= 1, \cdots n}.$ 
Let $\varrho :A\rightarrow A\{ \tau\}$ be the Drinfeld $A$-module of rank one given by $\varrho_{\theta}= (-1)^n\theta^n\tau +\theta.$ We have:
$${\rm Ker}(\exp_{\varrho}: K_{\infty}\rightarrow K_{\infty}) =\frac{\widetilde{\pi}}{\exp_C(\frac{\widetilde{\pi}}{\theta}) ^n}A$$ 
We easily deduce that:
$$\exp_{\varrho}(K_{\infty})=\{ x\in K_{\infty}, v_{\infty}(x)>\frac{n-q}{q-1}\}.$$
Since $\varphi_{a}(\frac{1}{\theta^{\frac{n-q}{q-1}}})\in A[t_1, \ldots, t_n]+\exp_{\varphi}(\mathbb T_n(K_{\infty})),$  we have :
$$\varrho_b(\frac{1}{\theta^{\frac{n-q}{q-1}}})\in A+\exp_{\varrho}(K_{\infty}).$$
We observe that, for  $\frac{n-q}{q-1}\geq \ell \geq 1,$ we have :
$$\varrho_{\theta}(\frac{1}{\theta^{\ell}})\equiv \frac{1}{\theta^{\ell-1}} \pmod{\frac{1}{\theta^{\ell-2}}A}.$$
Thus, for $0\leq i\leq \frac{n-q}{q-1}-1,$ we get:
$$\varrho_{\theta^ i}(\frac{1}{\theta^{(n-q)/(q-1)}})\equiv \frac{1}{\theta^{(n-q)/(q-1)-i}}\pmod{\frac{1}{\theta^{(n-q)/(q-1)-i-1}}A}.$$
This implies that $r\geq \frac{n-q}{q-1}.$
\end{proof}

\begin{theorem} \label{TheoremS2-2} The module $H^{(1)}(\varphi/A[t_1, \ldots, t_n])$ is the sub-$A[t_1, \ldots, t_n]$-module of $H(\varphi/A[t_1, \ldots, t_n])$ generated by the image of $\frac{1}{\theta^{(n-q)/(q-1)}}$ in $H(\varphi/A[t_1, \ldots, t_n]).$ We have an isomorphism of $A[t_1, \ldots, t_n]$-modules:
$$H^{(1)}(\varphi/A[t_1, \ldots, t_n])\simeq \frac{A[t_1, \ldots, t_n]}{\mathbb B(t_1, \ldots, t_n)A[t_1, \ldots, t_n]}.$$
Furthermore,  $\frac{H(\varphi/A[t_1, \ldots, t_n])}{H^{(1)}(\varphi/A[t_1, \ldots, t_n])}$ is a finitely generated and torsion $\mathbb F_q[t_1, \ldots, t_n]$-module.
\end{theorem}
\begin{proof} Set:
$$V=\{ x\in U(\varphi/A[t_1, \ldots, t_n]), \exp_{\varphi}^{(1)}(x)\in A[t_1, \ldots,t_n]+N\}.$$
We have:
$$U_{St}(\varphi/A[t_1, \ldots, t_n])=\mathcal L(\varphi/\mathbb A) A[t_1, \ldots, t_n]\subset V.$$
Observe that the $k$-vector space $W$ generated by $V$ is a free  $\mathbb A$-module of rank one and, by (\ref{EqS2-1}),  that $G:=[\frac{U(\varphi/\mathbb A)}{W}]_{\mathbb A}$ divides $\mathbb B(t_1, \ldots, t_n)$ in $\mathbb A.$ Thus $G\in A[t_1, \ldots, t_n].$ This implies:
$$V=GU(\varphi/A[t_1, \ldots, t_n]).$$
We have :
$$\exp_{\varphi}^{(1)} (\frac{G\widetilde{\pi}}{\omega(t_1)\cdots \omega(t_n)}) \in A[t_1, \ldots, t_n]+N.$$
Therefore, by Lemma \ref{LemmaS2-5}:
$$\varphi_{G}(\frac{1}{\theta^{\frac{n-q}{q-1}}})\in A[t_1, \ldots, t_n]+N.$$
Thus by Lemma \ref{LemmaS2-6}, we get that $deg_{\theta} G\geq \frac{n-q}{q-1},$ and since $G$ divides $\mathbb B(t_1, \ldots, t_n),$ we get $G=\mathbb B(t_1, \ldots, t_n),$ i.e.:
$$V=U_{St}(\varphi/A[t_1, \ldots, t_n]).$$
By Proposition \ref{PropositionS2-4} and the above equality, the map  $\exp_{\varphi}^{(1)}$ induces an injective morphism of $A[t_1, \ldots, t_n]$-modules:
$$\frac{U(\varphi/A[t_1, \ldots, t_n])}{U_{St}(\varphi/A[t_1, \ldots, t_n])}\hookrightarrow H(\varphi/A[t_1, \ldots, t_n]).$$
\end{proof}
Let's set (recall that $n\geq q$):
$$\mathcal M(\varphi/A[t_1, \ldots, t_n])=\{ \varphi_{a}(u_C^{(1)}(t_1, \ldots, t_n; 1)), a\in A[t_1, \ldots, t_n]\},$$
and define its radical $\sqrt{\mathcal M(\varphi/A[t_1, \ldots, t_n])}$ by
$$\{ x\in A[t_1, \ldots, t_n],\exists a\in A[t_1, \ldots, t_n]\setminus\{ 0\}, \varphi_a(x)\in \mathcal M(\varphi/A[t_1, \ldots,t_n])\}.$$
By the proof of Proposition \ref{PropositionS2-4}, $\mathcal M(\varphi/A[t_1, \ldots, t_n])$ is a free $A[t_1, \ldots, t_n]$-module of rank one.
\begin{corollary}\label{CorollaryS2-3} We have:
$$\sqrt{\mathcal M(\varphi/A[t_1, \ldots, t_n])}=\mathcal M(\varphi/A[t_1, \ldots,t_n]).$$
\end{corollary}
\begin{proof} Let $b\in \sqrt{\mathcal M(\varphi/A[t_1, \ldots, t_n])}\setminus\{ 0\}, $ then there exist $a,c \in A[t_1, \ldots, t_n]\setminus \{0\},$ such that:
$$\varphi_a(b)=\varphi_c(u_C^{(1)}(t_1, \ldots, t_n; 1)).$$
Let $E'=\frac{\mathbb K_{\infty}}{N'}$ where $N'=\{ x\in \mathbb K_{\infty}, v_{\infty}(x)> \frac{n-q}{q-1}\}.$
By an adaptation of the proofs of  Proposition \ref{PropositionS2-4} and Theorem \ref{TheoremS2-2},  we have:
$$\mathcal L(\varphi /\mathbb A)\mathbb A+N'=\{ x\in  \mathbb K_{\infty}, \exp_{\varphi}^{(1)}(x) \in \mathbb A+N'\}.$$
This implies that $\exp_{\varphi}^{(1)}$ induces an isomorphism of $k$-vector spaces:
$$\frac{\mathbb K_{\infty}}{\mathcal L(\varphi /\mathbb A)\mathbb A+N'}\simeq \frac{\mathbb K_{\infty}}{ \mathbb A+N'}.$$
Therefore,   there exists $d\in \mathbb A$ such that we have the following equality in $E':$
$$b=\varphi_d(u_C^{(1)}(t_1, \ldots, t_n; 1)).$$
Since $b, u_C^{(1)}(t_1, \ldots, t_n; 1) \in \mathbb A,$ this implies the equality in $\mathbb K_{\infty}:$ 
$$b=\varphi_d(u_C^{(1)}(t_1, \ldots, t_n; 1)).$$
 By an adaptation of the  proof of Proposition \ref{PropositionS2-4}, for $x\in \mathbb A\setminus\{0\},$ $\phi_x:E'\rightarrow E'$ is an injective morphism of $\mathbb A$-modules. In particular,  we get:
$$ad=c.$$
 Again, since $E'$ has no torsion point for $\varphi,$ we get in $E':$
$$b=\varphi_{c/a}(u_C^{(1)}(t_1, \ldots, t_n; 1)).$$
Since $\mathbb A\cap N'=\{0\},$ this equality is also true in $\mathbb K_{\infty}.$  By Proposition \ref{PropositionS2-3}, the leading coefficient of $u_C^{(1)}(t_1, \ldots, t_n; 1)$ is in $\mathbb F_q^\times.$ Now apply Lemma \ref{LemmaS2-4}, we get $\frac{c}{a} \in A[t_1, \ldots, t_n]$ and therefore $b\in \mathcal M(\varphi/A[t_1, \ldots, t_n]).$

\end{proof}

\begin{remark}\label{remark4} When $n=1,$ then we already mentioned that $H(\phi/A[t_1])=\{0\}.$ Furthermore, by Lemma \ref{LemmaS2-2}, we have:
$$u_C(t_1; 1)=1.$$
So we set, in this case:
$$\mathcal M(\varphi/A[t_1])=\{ \varphi_a(1), a\in A[t_1]\},$$
$$\sqrt{\mathcal M(\varphi/A[t_1])}=\{ x\in A[t_1], \exists b\in A[t_1]\setminus\{0\}, \varphi_b(x)\in \mathcal M(\varphi/A[t_1])\}.$$
We have:
$$\varphi_{\theta-t_1}(1) =0.$$
Furthermore:
$$ U(\varphi/\mathbb A)= \mathcal L(\varphi/\mathbb A) \mathbb A,$$
and:
$${\rm Ker}(\exp_{\phi}:\mathbb T_1(K_{\infty})\rightarrow \mathbb T_1(K_{\infty}))=\frac{\widetilde{\pi}}{\omega(t_1)}A[t_1].$$
By \cite{APTR}, Lemma 7.1 , we have an isomorphism of $A[t_1]$-modules:
$$\mathcal M(\varphi/A[t_1])\simeq \frac{A[t_1]}{(\theta-t_1)A[t_1]}.$$
In particular $\mathcal M(\varphi/A[t_1])=\mathbb F_q[t_1].$
But, since $\mathcal M(\varphi/A[t_1])\subset \exp_{\varphi}(\mathbb T_1(K_{\infty})), $ we get:  $\sqrt{\mathcal M(\varphi/A[t_1])}\subset \exp_{\varphi}(\mathbb K_{\infty})\cap \mathbb A=\exp_{\varphi}(U(\varphi/\mathbb A)).$ Thus: $\sqrt{\mathcal M(\varphi/A[t_1])}\subset \exp_{\varphi}(U(\varphi/\mathbb A)).$ Note that we have an isomorphism of $\mathbb A$-modules:
$$\exp_{\varphi}(U(\varphi/\mathbb A))\simeq \frac{\mathbb A}{(\theta-t_1)\mathbb A},$$
and $\exp_{\varphi}(U(\varphi/\mathbb A))$is the $\mathbb A$-module (via $\varphi$) generated by $1.$  In particular, we have $\exp_{\varphi}(U(\varphi/\mathbb A))=\mathbb F_q(t_1).$ Thus:
$$\exp_{\varphi}(U(\varphi/\mathbb A))\cap \mathbb T_1(K_{\infty})=\mathcal M(\varphi/A[t_1]).$$
Therefore, we also have in this case:
$$\sqrt{\mathcal M(\varphi/A[t_1])}=\mathcal M(\varphi/A[t_1]).$$

\end{remark}
\subsection{$P$-adic $L$-series}$\label{PadicLseries}{}$\par
Let $n\geq 0$ be an integer. Let $\phi:A\rightarrow A\{\tau\}$ be a Drinfeld $A$-module defined over $A,$ and let $\varphi$ be its canonical deformation over $\mathbb T_n(K_{\infty}),$ $\widetilde{\varphi}$ be the canonical $z$-deformation of $\varphi$ (see section \ref{Basicproperties}).\par
Let $\log_\phi\in K\{\{\tau\}\}$ (respectively $\log_\varphi\in \mathbb T_n(K_{\infty})\{\{\tau\}\},$ $\log_{\widetilde{\varphi}}\in \mathbb T_{n,z}(K_{\infty})\{\{\tau\}\}$)  be the unique element such that $\log_\phi \equiv 1\pmod{\tau}$ and :
$$\log_\phi \phi_{\theta}= \theta \log_\phi \, ({\rm respectively\, } \log_\varphi \varphi_{\theta}= \theta \log_\varphi,\, \log_{\widetilde{\varphi}} \widetilde{\varphi}_{\theta}= \theta \log_{\widetilde{\varphi}} ).$$
Write:
$$\log_\phi =\sum_{j\geq 0} l_j\tau ^j , l_j \in K,$$
then:
$$\log_\varphi =\sum_{j\geq 0} b_j(t_1) \ldots b_j(t_n)l_j\tau^j \, {\rm and }\, \log_{\widetilde{\varphi}} =\sum_{j\geq 0} b_j(t_1) \ldots b_j(t_n)l_jz^j\tau^j .$$
Now, we fix $P$ a monic irreducible element of $A$ of degree $d\geq 1.$ Let $v_P:K\rightarrow \mathbb Z\cup \{ +\infty\}$ be the $P$-adic valuation on $K.$
\begin{lemma}\label{LemmaS2-7} For $n\geq 0,$ we have:
$$v_P({l_n})\geq-[\frac{n}{d}] ,$$
where for $x\in \mathbb R,$ $[x]$ denotes the integer part of $x.$
\end{lemma}
\begin{proof} We have for $n\geq 1:$
$$(\theta-\theta^{q^n}){l_n}=\sum_{j=0}^{n-1} {l_j} \alpha_{n-j}^{q^j},$$
where $\phi_{\theta}= \sum_{j\geq 0} \alpha_j \tau^j,$ and $\alpha_j \in A, $ $\alpha_j=0$ for $j\geq r+1$ where $r$ is the rank of $\phi.$ Set $\ell_0=1$ and for $j\geq 1,$ $\ell_j= \prod_{k=1}^{j-1} (\theta-\theta^{q^k}).$ Then, we deduce that:
$$\forall n\geq 0, {l_n}=\frac{a_n}{\ell_n}, a_n \in A.$$
\end{proof}
Let $K_P^{ac}$ be an algebraic closure of $K_P,$ where $K_P$ is the $P$-adic completion of $K.$ Let $\mathbb C_P$ be the completion of $K_P^{ac}$ equipped with the $P$-adic valuation $v_P:\mathbb C_P\rightarrow \mathbb Q\cup \{+\infty\}$ ($v_P(P)=1$). We denote by $v_P$ the $P$-adic Gauss valuation on $\mathbb C_P[t_1, \ldots, t_n,z],$ i.e. if
$f=\sum_{i_1, \ldots, i_n ,j \in \mathbb  N} x_{i_1, \ldots, i_n,j} t_1^{i_1}\cdots t_n^{i_n} z^j,  x_{i_1, \ldots, i_n,j}\in \mathbb C_P$, then
$$v_P(f)={ \rm Inf}\{ v_P(x_{i_1, \ldots, i_n,j}),i_1, \ldots, i_n ,j \in \mathbb  N\}.$$
Let $\mathbb T_{n,z}(\mathbb C_P)$ be the completion of $\mathbb C_P[t_1, \ldots, t_n ,z]$ for $v_P,$ and let $\mathbb T_n(\mathbb C_P) $ be the closure of $\mathbb C_P[t_1, \ldots, t_n]$  in  $\mathbb T_{n,z}(\mathbb C_P).$\par
We denote by 
$\mathbb  T_{n,z}(K_P)$ the closure of $K[t_1, \ldots, t_n,z]$ in $\mathbb T_{n,z}(\mathbb C_P),$ and $\mathbb T_n(K_P)$ the closure of $K[t_1, \ldots, t_n]$ in $\mathbb T_{n,z}(\mathbb C_P).$ We also denote by $\mathbb A_P$ the closure of $A[t_1, \ldots, t_n]$ in $\mathbb T_n(K_P),$ and $\widetilde{\mathbb A_P}$ the closure of $A[t_1, \ldots, t_n,z]$ in $\mathbb T_{n,z}(K_P).$ Finally let $O=\{ x\in \mathbb C_P, v_P(x)\geq 0\},$ and let $\widetilde{\mathbb  O}$ be the closure  of $O[t_1, \ldots, t_n,z]$ in $\mathbb T_{n,z}(\mathbb C_P),$ and $\mathbb O$ be the closure of $O[t_1, \ldots, t_n]$ in $\mathbb T_{n,z}(\mathbb C_P).$\par
We still denote by  $\tau $ the  continuous morphism of $\mathbb F_q(t_1, \ldots, t_n,z)$-algebras  $\tau: \mathbb T_{n,z}(\mathbb C_P) \rightarrow \mathbb T_{n,z}(\mathbb C_P),$ $\forall x\in \mathbb C_P, \tau(x)=x^q.$\par
\begin{lemma}\label{LemmaS2-8} $\widetilde{\varphi}$ extends by continuity  to a morphism of $\mathbb F_q[t_1, \ldots, t_n,z]$-algebras~:
$$\widetilde{\varphi}: \widetilde{\mathbb A_P}\rightarrow \widetilde{\mathbb A_P}\{\{Ê\tau\}\}.$$
\end{lemma}
\begin{proof} Observe that:
$$\widetilde{\varphi}_P= P+f\tau, f\in A[t_1, \ldots, t_n,z]\{\tau\}.$$
This implies that, for $m\geq 1,$ we have:
$$\widetilde{\varphi}_{P^m}= \sum_{k=0}^{m-1} a_k \tau^k +f_m \tau^m, a_0, \ldots, a_{m-1}\in A[t_1, \ldots, t_n,z], f_m\in A[t_1, \ldots, t_n,z]\{\tau\},$$
with the property:
$$k\in \{ 0, \ldots, m-1\}, v_P(a_k) \geq m-k.$$
This implies that if $x\in P^mA[t_1, \ldots, t_n,z]$ for some $m\geq 1,$ we get:
$$\widetilde{\varphi}_{x}= \sum_{k=0}^{m-1} b_k \tau^k +g_m \tau^m, b_0, \ldots, b_{m-1}\in A[t_1, \ldots, t_n,z], g_m\in A[t_1, \ldots, t_n,z]\{\tau\},$$
with the property:
$$k\in \{ 0, \ldots, m-1\}, v_P(
b_k) \geq m-k.$$
Thus, if $(a_n)_{n\geq 0}$ is a Cauchy sequence for $v_P$ of elements in $A[t_1, \ldots, t_n,z],$ then the sequence $(\widetilde{\varphi}_{a_n})_{n\geq 0}$  converges in $\widetilde{\mathbb A_P}\{\{Ê\tau\}\}.$
\end{proof}
Observe that $\widetilde{\mathbb O}$ is an $A[t_1, \ldots, t_n,z]$-module via $\widetilde{\varphi},$ and by the above Lemma $\{x\in \widetilde{\mathbb O}, v_P(x)>0\}$ is an $\widetilde{\mathbb A_P}$-module via $\widetilde{\varphi}.$

\begin{lemma}\label{LemmaS2-9} $\log_{\widetilde{\varphi}}$ induces an injective  morphism of $\mathbb F_q[t_1, \ldots, t_n,z]$-modules:
$$\log_{\widetilde{\varphi},P}: \widetilde{\mathbb O}\rightarrow \mathbb T_{n,z}(\mathbb C_P) ,$$
such that:
$$\forall a\in A[t_1, \ldots, t_n,z] , \forall x\in  \widetilde{\mathbb O},
\log_{\widetilde{\varphi},P}(\widetilde{\varphi}_a(x))= a\log_{\widetilde{\varphi},P}(x).$$
Furthermore, if $x\in \widetilde{\mathbb O},$ $v_P(x)>0,$  we have:
$$\log_{\widetilde{\varphi},P}(x)=\sum_{j\geq 0}{b_j(t_1)\cdots b_j(t_n) z^j}{l_j}\tau^j(x),$$
and:
$$\forall a\in \widetilde{\mathbb A_P} , 
\log_{\widetilde{\varphi},P}(\widetilde{\varphi}_a(x))= a\log_{\widetilde{\varphi},P}(x).$$
\end{lemma}
\begin{proof} Let's set:
$$\mathcal M=\{ x \in \widetilde{\mathbb O}, v_P(x)>0\}.$$
Observe that we have a direct sum  of $\mathbb F_q[t_1, \ldots, t_n,z]$-modules:
$$\widetilde{\mathbb O}= \overline{\mathbb F_q}[t_1, \ldots, t_n,z]\oplus \mathcal M,$$
where $\overline{\mathbb F_q}$ is the algebraic closure of $\mathbb F_q$ in $\mathbb C_P.$ By Lemma \ref{LemmaS2-7},
$\log_{\widetilde{\varphi}}$ converges on $\mathcal M$ and we have the last assertion of the Lemma by Lemma \ref{LemmaS2-8}. Furthermore, for $x\in \mathcal M,$ we have:
$$v_P(\widetilde{\varphi}_P(x))\geq{\rm Inf}\{qv_P(x), v_P(x)+1\}.$$
This implies that there exists an integer $m$ depending on $x$ such  that:
$$\log_{\widetilde{\varphi}}(x)= \frac{1}{P^m}\log_{\widetilde{\varphi}}(\widetilde{\varphi}_{P^m}(x))\in \frac{1}{P^m}\mathcal M,$$
and again by Lemma \ref{LemmaS2-7}:
$$v_P(\log_{\widetilde{\varphi}}(\widetilde{\varphi}_{P^m}(x))= v_P(\widetilde{\varphi}_{P^m}(x)).$$
Now, we observe that $\widetilde{\mathbb  O}$ is an $A[t_1, \ldots, t_n,z]$-module via $\widetilde{\varphi}$ without torsion. This implies that $\log_{\widetilde{\varphi}}:\mathcal M\rightarrow \mathbb T_{n,z}(\mathbb C_P)$ is injective.\par
Now, let $m\geq 1, m\equiv 0\pmod{d}.$ Let:
$$\overline{M}=\frac{\mathbb F_{q^m}[t_1, \ldots, t_n,z]\oplus \mathcal M}{\mathcal M}.$$
Then $\overline{M}$ is a finitely generated  and a free $\mathbb F_q[t_1, \ldots, t_n,z]$-module and also an $A[t_1, \ldots, t_n,z]$-module via $\widetilde{\varphi}.$ This implies that $\overline{M}$ is a torsion $A[t_1, \ldots, t_n,z]$-module. Thus there exists $a\in A[t_1, \ldots, t_n,z]\cap \mathbb T_{n,z}(K_P)^\times$ such that:
$$\widetilde{\varphi}_a(\mathbb F_{q^m})\subset \mathcal M.$$
Now, let $x\in \widetilde{\mathbb O}.$ By the above discussion, there exists $b\in A[t_1, \ldots, t_n,z]\cap \mathbb T_{n,z}(K_P)^\times $ such that:
$$\widetilde{\varphi}_b(x)\in \mathcal M.$$
We set:
$$\log_{\widetilde{\varphi}, P}(x) =\frac{1}{b}\log_{\widetilde{\varphi}}(\widetilde{\varphi}_b(x)).$$
This does not depend on the choice of $b.$ 
\end{proof}
Recall that:
$$\mathcal L(\widetilde{\varphi}/\widetilde{\mathbb A})= \prod_{Q}\frac{[\frac{\widetilde{\mathbb A}}{Q\widetilde{\mathbb A}}]_{\widetilde{\mathbb A}}}{[\widetilde{\varphi}(  \frac{\widetilde{\mathbb A}}{Q\widetilde{\mathbb A}}      )]_{\widetilde{\mathbb A}}} \in \mathbb T_{n,z}(K_{\infty})^\times,$$
where $Q$ runs through the monic irreducible elements in $A.$
Thus:
$$(\frac{[\frac{\widetilde{\mathbb A}}{P\widetilde{\mathbb A}}]_{\widetilde{\mathbb A}}}{[\widetilde{\varphi}(  \frac{\widetilde{\mathbb A}}{P\widetilde{\mathbb A}}      )]_{\widetilde{\mathbb A}}})^{-1}\mathcal L(\widetilde{\varphi}/\widetilde{\mathbb A})= \prod_{Q\not =P}\frac{[\frac{\widetilde{\mathbb A}}{Q\widetilde{\mathbb A}}]_{\widetilde{\mathbb A}}}{[\widetilde{\varphi}(  \frac{\widetilde{\mathbb A}}{Q\widetilde{\mathbb A}}      )]_{\widetilde{\mathbb A}}} \in \mathbb T_{n,z}(K_{\infty})^\times.$$
Now, by (\ref{unitz}), we have the formal equality in $K(t_1, \ldots, t_n)[[z]]:$
$$\mathcal L(\widetilde{\varphi}/\widetilde{\mathbb A})=\log_{\widetilde{\varphi}}(u_\phi(t_1, \ldots, t_n; z)).$$
Therefore, we define $\mathcal L_P(\widetilde{\varphi}/\widetilde{\mathbb A})$ by the formal equality in $K(t_1, \ldots, t_n)[[z]]:$
$$\mathcal L_P(\widetilde{\varphi}/\widetilde{\mathbb A})= (\frac{[\frac{\widetilde{\mathbb A}}{P\widetilde{\mathbb A}}]_{\widetilde{\mathbb A}}}{[\widetilde{\varphi}(  \frac{\widetilde{\mathbb A}}{P\widetilde{\mathbb A}}      )]_{\widetilde{\mathbb A}}})^{-1}\mathcal L(\widetilde{\varphi}/\widetilde{\mathbb A}).$$

\begin{theorem}\label{TheoremS2-3} $\mathcal L_P(\widetilde{\varphi}/\widetilde{\mathbb A})$ converges in $\widetilde{\mathbb A_P} $ and we have the equality in $\mathbb T_{n,z}(K_P):$
$$\mathcal L_P(\widetilde{\varphi}/\widetilde{\mathbb A})= \frac{[\widetilde{\varphi}(  \frac{\widetilde{\mathbb A}}{P\widetilde{\mathbb A}}      )]_{\widetilde{\mathbb A}}} {P}\log_{\widetilde{\varphi},P}(u_{\phi}(t_1, \ldots, t_n; z)).$$
\end{theorem}
\begin{proof} Let $\rho:A\rightarrow A\{\tau\}$ be the Drinfeld $A$-module given by:
$$\rho_{\theta}=\sum_{j=0}^r \alpha_j P^{q^j -1} \tau^j,$$
where $\phi_{\theta}= \sum_{j=0}^r \alpha_j\tau^j.$ Let $\varrho$ be the canonical deformation of $\rho$ over $\mathbb T_n(K_{\infty})$ and $\widetilde{\varrho}$ be the canonical $z$-deformation of $\varrho.$ Then:
$$\log_{\widetilde{\varrho}}=\sum_{j\geq  0} {b_j(t_1) \cdots b_j(t_n) P^{q^j -1}}{l_j} z^j\tau ^j,$$
where $\log_{\phi}= \sum_{j\geq 0}{l_i} \tau ^i.$ By (\ref{unitz}), we have the formal identity in $K(t_1, \ldots, t_n)[[z]]:$
$$\mathcal L(\widetilde{\varrho}/\widetilde{\mathbb A})= \log_{\widetilde{\varrho}}(u_{\rho}(t_1, \ldots, t_n; z)).$$
But, by Lemma \ref{LemmaS2-7}, we have:
$$\forall j\geq 0, v_P({b_j(t_1) \cdots b_j(t_n) P^{q^j -1}}{l_j} )\geq q^j -1-[\frac{j}{d}].$$
This implies that $\mathcal L(\widetilde{\varrho}/\widetilde{\mathbb A})$ converges in $\widetilde{\mathbb A_P}.$ Let $Q$ be a monic irreducible element in $A$ prime to $P$.  The multiplication by $P$ gives rise to an isomorphism of $\widetilde{\mathbb A}$-modules:
$$\widetilde{\varrho}(\frac{\widetilde{\mathbb A}}{Q\widetilde{\mathbb A}})\simeq \widetilde{\varphi}(\frac{\widetilde{\mathbb A}}{Q\widetilde{\mathbb A}}).$$
This implies that we have the following equality in $K(t_1, \ldots, t_n)[[z]]:$
$$\mathcal L(\widetilde{\varrho}/\widetilde{\mathbb A})= \mathcal L_P(\widetilde{\varphi}/\widetilde{\mathbb A}).$$
Therefore we get the first assertion of the Theorem. The second assertion is a consequence of (\ref{unitz}) and Lemma \ref{LemmaS2-9}.

\end{proof}
Observe that  we have analogous constructions for $\varphi$ as that made in Lemma \ref{LemmaS2-8} and Lemma \ref{LemmaS2-9}. Thus, we have a continuous morphism of $\mathbb F_q[t_1, \ldots, t_n]$-modules: 
$$\log_{\varphi, P}: \mathbb O\rightarrow \mathbb T_n(\mathbb C_P),$$
such that:
$$\forall a \in \mathbb A, \log_{\varphi, P}(\varphi_a(x))=a\log_{\varphi, P}(x).$$
But, $\log_{\varphi,P}$  is no longer injective.  \par
Let $ev: \mathbb T_{n,z}(\mathbb C_P)\rightarrow \mathbb T_n(\mathbb C_P), y\mapsto y\mid_{z=1}.$ Then:
$$\forall y\in \widetilde{\mathbb O}, ev(\log_{\widetilde{\varphi}, P}(y))= \log_{\varphi, P}(ev(y)).$$
We set:
$$\mathcal L_P(\varphi/\mathbb A)= ev(\mathcal L_P(\widetilde{\varphi}/\widetilde{\mathbb A}))\in \mathbb A_P.$$
We get:
\begin{corollary}\label{CorollaryS2-4} We have the following equality in $\mathbb T_n(K_P):$
$$\mathcal L_P(\varphi/\mathbb A)= \frac{[\varphi(  \frac{\mathbb A}{P\mathbb A}      )]_{\mathbb A}} {P}\log_{\varphi,P}(u_{\phi}(t_1, \ldots, t_n; 1)).$$
In particular $\mathcal L_P(\varphi/\mathbb A)= 0$ if and only if $u_{\phi}(t_1, \ldots, t_n; 1)$ is a torsion point for $\varphi.$
\end{corollary}
\begin{proof} This is a direct consequence of Theorem \ref{TheoremS2-3}.
\end{proof}
We set:
$$\mathcal L_P^{(1)}(\varphi/\mathbb A)=ev(\frac{d}{dz} \mathcal L_P(\widetilde{\varphi}/\widetilde{\mathbb A}))\in \mathbb A_P,$$
$$u_{\phi}^{(1)}(t_1, \ldots, t_n; 1) =ev(\frac{d}{dz} u_{\phi}(t_1, \ldots, t_n;z))\in A[t_1, \ldots, t_n],$$
$$\forall a\in A[t_1, \ldots, t_n], \varphi^{(1)}_a= ev(\frac{d}{dz} \widetilde{\varphi}_a)\in A[t_1, \ldots, t_n]\{\tau\}.$$
\begin{proposition}\label{PropositionS2-5} Let's assume that $u_{\phi}(t_1, \ldots, t_n; 1)$ is a torsion point for $\varphi.$
Then, in $\mathbb T_n(K_P)$, $\mathcal L_P^{(1)}(\varphi/\mathbb A)$ is equal to
$$\frac{[\varphi(  \frac{\mathbb A}{P\mathbb A}      )]_{\mathbb A}} {P}\log_{\varphi,P}(u_{\phi}^{(1)}(t_1, \ldots, t_n; 1))+ \frac{[\varphi(  \frac{\mathbb A}{P\mathbb A}      )]_{\mathbb A}} {aP}\log_{\varphi,P}(\varphi_a^{(1)}(u_{\phi}(t_1, \ldots, t_n; 1))),$$
for any  $a\in A[t_1, \ldots, t_n]$ which is monic as a polynomial in $\theta$ and such that $\varphi_a(u_{\phi}(t_1, \ldots, t_n; 1))=0.$
\end{proposition}
\begin{proof} Let's start with the equality in $K(t_1, \ldots, t_n)[[z]]:$
$$\mathcal L(\widetilde{\varphi}/\widetilde{\mathbb A})=\log_{\widetilde{\varphi}}(u_\phi(t_1, \ldots, t_n; z)).$$
We get, for $a\in A[t_1, \ldots t_n]:$
$$a\mathcal L(\widetilde{\varphi}/\widetilde{\mathbb A})=\log_{\widetilde{\varphi}}(\widetilde{\varphi}_a(u_\phi(t_1, \ldots, t_n; z))).$$
Thus
\begin{eqnarray*}
a\frac{d}{dz}\mathcal L(\widetilde{\varphi}/\widetilde{\mathbb A}) &=& (\frac{d}{dz}\log_{\widetilde{\varphi}})(\widetilde{\varphi}_a(u_\phi(t_1, \ldots, t_n; z))) + \log_{\widetilde{\varphi}}(\frac{d}{dz}(\widetilde{\varphi}_a)(u_\phi(t_1, \ldots, t_n; z)))\\ && +\log_{\widetilde{\varphi}}(\widetilde{\varphi}_a(\frac{d}{dz}(u_\phi(t_1, \ldots, t_n; z)))).
\end{eqnarray*}
Now, by Remark \ref{remark2}, there exists an element $a\in A[t_1, \ldots, t_n]$ which is monic as a polynomial in $\theta$ such that:
$$\varphi_{a}(u_{\phi}(t_1, \ldots, t_n; z))=0.$$
The assertion of the Proposition follows from Corollary \ref{CorollaryS2-4} and Theorem \ref{TheoremS2-3}.
\end{proof}
\section{Arithmetic of cyclotomic function fields}\label{section3}
Let $p$ be an odd prime number and let $X$ be the $p$-Sylow subgroup of the ideal class group of $\mathbb Q(e^{\frac{2i\pi}{p}}).$ Let $\Delta={\rm Gal}(\mathbb Q(e^{\frac{2i\pi}{p}})/\mathbb Q)\simeq (\frac{\mathbb Z}{p\mathbb Z})^\times$ and let $\widehat{\Delta} ={\rm Hom}(\Delta, \mathbb Z_p^\times).$ Observe that $X$ is a $\mathbb Z_p[\Delta]$-module.  For $\chi \in \widehat {\Delta},$ let:
$$X(\chi)=e_\chi X,$$
where $e_\chi =\frac{1}{\mid \Delta\mid} \sum_{\delta \in \Delta} \chi( \delta) \delta^{-1}\in \mathbb Z_p[\Delta].$\par
Let $\omega_p \in \widehat{\Delta}$ be the $p$-adic Teichm\" uller character. Let $\chi\in \widehat{\Delta}$ be an odd character. Then by the famous Leopoldt's Spiegelungssatz (\cite{LEO}), we have:
$$\dim_{\mathbb F_p} \frac{X(\omega_p\chi^{-1})}{pX(\omega_p\chi^{-1})}\leq \dim_{\mathbb F_p} \frac{X(\chi)}{pX(\chi)}\leq \dim_{\mathbb F_p} \frac{X(\omega_p\chi^{-1})}{pX(\omega_p\chi^{-1})}+1.$$
Therefore if $X(\omega_p\chi^{-1})=\{0\}$, $X(\chi)$ is a cyclic $\mathbb Z_p$-module. In particular, if Vandiver's Conjecture is true for the prime $p$ then $X(\chi)$ is a cyclic $\mathbb Z_p$-module for every odd character $\chi$ of conductor $p.$ The cyclicity of $X$ as a $\mathbb Z_p[\Delta]$-module is still an open problem  and is conjectured to be true. This problem is sometimes referred as  "The Iwasawa-Leopoldt Conjecture" (\cite{LAN}, page 80), and by the above discussion Vandiver's Conjecture implies this cyclicity statement.\par
In this section, using the ideas developed in sections  \ref{section1} and \ref{section2}, we study an analogous question for  the case of the $P$th cyclotomic function field where $P$ is a monic irreducible element in $A.$ Note that the analogue of Vandiver's Conjecture is false in this context (\cite{ANG&TAE1}). Furthermore, we are not convinced that the isotypic components of the " $P$-Sylow subgroup " of Taelman's class module associated to the $P$th cyclotomic function field are cyclic $A_P$-modules, where $A_P$ is the $P$-adic completion of $A,$ even in the case of  odd characters (it would be very interesting to exhibit, if they exist, such examples). However we are able to prove a cyclicity result (Theorem \ref{TheoremS3-2})  that involves the "derivatives" of Goss $P$-adic $L$-series for odd characters. This result does not seem to have an analogous counterpart in the classical case.

\subsection{Dirichlet characters}$\label{Dirichletcharacters}{}$\par
Let $\overline{\mathbb F_q}$ be the algebraic closure of $\mathbb F_q$ in $\mathbb C_\infty.$ Let $\zeta \in \overline {\mathbb F_q}$ and let $Q$ be the monic irreducible element in $A$ such that $Q(\zeta)=0.$ The morphism of $\mathbb F_q$-algebras  $\rho_\zeta: A\rightarrow \mathbb C_\infty, a\mapsto a(\zeta),$ induces an injective  morphism of groups still denoted by $\rho_\zeta:$
$$\rho_\zeta: (\frac{A}{QA})^\times\rightarrow \mathbb C_\infty^\times.$$
A Dirichlet character is a morphism of groups $\chi:(\frac{A}{gA})^\times \rightarrow \mathbb C_\infty^\times,$ for some $g\in A_+.$ We observe that the conductor of $\chi,$ $f_\chi \in A_+,$ is a square-free monic polynomial. In this article all Dirichlet characters are assumed to be primitive, i.e. viewed as defined modulo their conductors. For example, if $\zeta \in \overline{\mathbb F_q},$ $\rho_\zeta$ is a Dirichlet character of conductor $Q,$ where $Q$ is the monic irreducible polynomial in $A$ such that $Q(\zeta)=0.$\par
Let $\chi$ be a (primitive) Dirichlet character.  Then there exist $\zeta_1(\chi), \ldots, \zeta_m(\chi) \in \overline{\mathbb F_q}, $ $n_1(\chi), \ldots, n_m(\chi)\in \{ 1, \ldots, q-1\},$ $m\geq 0,$ such that:
$$\forall a\in A, \chi(a) =\prod_{j=1}^m a(\zeta_j(\chi))^{n_j(\chi)}.$$
Note that the conductor of $\chi$ is the least common multiple of the conductors of the $\rho_{\zeta_j(\chi)}, j=1, \ldots,m.$
The type of $\chi,$ $s(\chi),$ is defined by:
$$s(\chi)=\sum_{j=1}^m n_j(\chi).$$
For example, the only Dirichlet character of type $0$ is the trivial character, and the Dirichlet characters of type $1$ are exactly the $\rho_\zeta,$ $\zeta\in \overline{\mathbb F_q}.$ A Dirichlet character, $\chi,$ is called odd if $s(\chi)\equiv 1\pmod{q-1},$ and even otherwise.\par
${}$\par
Let's select $\lambda_\theta \in \mathbb C_\infty$ such that $\lambda_\theta^{q-1}=-\theta.$ We set:
$$\widetilde{\pi} =\lambda_\theta \theta  \prod_{j\geq 1} (1-\theta^{1-q^j})^{-1}\in K_\infty(\lambda_\theta)^\times.$$
Recall that $\tau :\mathbb C_\infty\rightarrow \mathbb C_\infty$ is the continuous morphism of $\mathbb F_q$-algebras such that $\tau(\theta)=\theta^q.$ Let $C$ be the Carlitz module (recall that $C_\theta=\tau +\theta$). Then:
$${\rm Ker}(\exp_C: \mathbb C_\infty \rightarrow \mathbb C_\infty)= \widetilde{\pi} A.$$
Observe that:
$$\lambda_\theta= \exp_C(\frac{\widetilde{\pi}}{\theta}).$$
For $a\in A_+,$ we set: $\lambda_a=\exp_C(\frac{\widetilde{\pi}}{a})\in \mathbb F_q((\frac{1}{\lambda_\theta})).$ Let $\sigma:\overline{\mathbb F_q}((\frac{1}{\lambda_\theta}))\rightarrow \overline{\mathbb F_q}((\frac{1}{\lambda_\theta}))$ be the continuous morphism of $\overline{\mathbb F_q}$-algebras such that $\sigma (\lambda_\theta)= \lambda_\theta^q.$\par
${}$\par
Finally, we recall the definition of the Gauss-Thakur sum attached to a Dirichlet character. Let $\zeta\in \overline{\mathbb F_q}$ and let $Q$ be the monic irreducible element in $A$ such that $Q(\zeta)=0.$ We set (see also \cite{THA}):
$$g(\rho_\zeta)= -\sum_{n=0}^{\deg_\theta Q-1}\sum_{a\in A_{+,n}} \rho_\zeta (a)^{-1}C_a(\lambda_Q) \in \overline{\mathbb F_q}((\frac{1}{\lambda_\theta}))^\times.$$
One easily verifies that:
$$\sigma (g(\rho_\zeta))= (\zeta-\theta) g(\rho_\zeta).$$
In particular:
$$\tau^{\deg_\theta Q}(g(\rho_\zeta))= \sigma^{\deg_\theta Q}(g(\rho_\zeta))= (-1)^{\deg_\theta Q}Q g(\rho_\zeta).$$
Also, we have (\cite{THA}, Proposition I ):
\begin{eqnarray}\label{eqS3-1}\sum_{k=0}^{\deg_\theta Q -1} g(\rho_{\zeta^{q^k}})= \lambda_Q.\end{eqnarray}
Now, let $\chi$ be a Dirichlet character. Let $\zeta_1(\chi), \ldots, \zeta_m(\chi) \in \overline{\mathbb F_q}, $ $n_1(\chi), \ldots, n_m(\chi)\in \{ 1, \ldots, q-1\},$ $m\geq 0,$ such that: $\forall a\in A, \chi(a) =\prod_{j=1}^m a(\zeta_j(\chi))^{n_j(\chi)}.$ We set (see also \cite{ANG&PEL1}):
$$g(\chi)=\prod_{j=1}^m g(\rho_{\zeta_j(\chi)})^{n_j(\chi)}.$$
Then:
\begin{eqnarray}\label{eqS3-2}\sigma(g(\chi))=\prod_{j=1}^m(\zeta_j(\chi)-\theta)^{n_j(\chi)}\, g(\chi).\end{eqnarray}

\subsection{$P$-adic Dirichlet-Goss $L$-series}$\label{P-adicDirichlet-GossL-series}{}$\par
Let $P$ be a monic irreducible polynomial in  $A$ of degree $d.$ Let $K^{ac}$ be the algebraic closure of $K$ in $\mathbb C_\infty$ and we fix once for all a $K$-embedding of $K^{ac}$ in $\mathbb C_P.$ Let $n\geq 0$ be an integer and let $t_1, \ldots, t_n,z$ be $n+1$ indeterminates over $\mathbb C_P.$ Let's set:
$$L_P(t_1, \ldots, t_n; z)=\sum_{n\geq 0}\sum_{a\in A_{+,n}, a\not \equiv 0\pmod{P}}\frac{a(t_1)\ldots a(t_n)}{a} z^n\in K[t_1, \ldots, t_n][[z]].$$
Then by Theorem \ref{TheoremS2-3}:
$$L_P(t_1, \ldots, t_n; z)\in \mathbb T_{n,z}(K_P).$$\par
Let $\chi$ be a Dirichlet character. Let $\zeta_1(\chi), \ldots, \zeta_m(\chi) \in \overline{\mathbb F_q}, $ $n_1(\chi), \ldots, n_m(\chi)\in \{ 1, \ldots, q-1\},$ $m\geq 0,$ such that: $\forall a\in A, \chi(a) =\prod_{j=1}^m a(\zeta_j(\chi))^{n_j(\chi)}.$ Let $n=n_1(\chi)+\cdots+n_m(\chi)$ be the type of $\chi.$ Let's recall that the value at one of the Dirichlet-Goss $P$-adic $L$-series attached to $\chi$ is defined by (see \cite{GOS4}):
$$L_P(1, \chi) =\sum_{n\geq 0}\sum_{a\in A_{+,n}, a\not \equiv 0\pmod{P}}\frac{\chi(a)}{a} \in \mathbb C_P.$$
Now let $\eta_1, \ldots, \eta_n \in \overline{\mathbb F_q}$ such that  $\eta_{n_1(\chi)+\cdots n_{j-1}(\chi)+k }= \zeta_j,$ for $k\in \{ 1, \ldots, n_j(\chi)\}, $ $j=1, \ldots, m.$ Let $ev_\chi: K[t_1, \ldots, t_n][[z]]\rightarrow \overline{\mathbb F_q}(\theta)[[z]],$ be the $K$-linear map defined by:
$$ev_\chi (f)=f\mid_{t_1=\zeta_1, \ldots, t_n=\zeta_n}.$$
It is clear that we have:
$$L_P(1, \chi) =ev_\chi(L_P(t_1, \ldots, t_n;z))\mid_{z=1}.$$
This implies that, by Corollary \ref{CorollaryS2-4} and Lemma \ref{LemmaS2-2} and Lemma \ref{LemmaS2-3}, if $\chi $ is odd then:
$$L_P(1, \chi)=0.$$
Thus, if $\chi$ is odd, we set:
$$L_P^{(1)}(1, \chi)= \frac{d}{dz}(ev_\chi(L_P(t_1, \ldots, t_n; z))\mid_{z=1}\in \mathbb C_P.$$
\begin{theorem}
\label{TheoremS3-1} Let $\chi$ be a Dirichlet character. Then $L_P(1, \chi)=0$ if and only if $\chi$ is odd. Furthermore, if $\chi$ is odd, we have: 
$$L_P^{(1)}(1, \chi)\not =0.$$
\end{theorem}
Note that the case of characters of conductor $P$ and $\chi$ even  is treated in \cite{ANG&TAE1} and the   non-vanishing result  uses  Bosser's $P$-adic Baker Brumer Theorem (see the appendix of \cite{ANG&TAE1}). In what follows, we propose a new approach which does not use the $P$-adic Baker-Brumer Theorem. \par
\noindent We  will now work in $K[t_1, \ldots, t_n][[z]].$ Let 
$$\tau : K[t_1, \ldots, t_n][[z]]\rightarrow K[t_1, \ldots, t_n][[z]]$$
 be the continuous morphism (for the $z$-adic topology) of $\mathbb F_q[t_1,\ldots, t_n][[z]]$-algebras such that $$\forall x\in K, \tau(x)=x^q.$$  Set :
$$\log_{n,z}=\sum_{k\geq 0} \frac{b_i(t_1)\cdots b_i(t_n)}{\ell_k}z^k \tau ^k,$$
where for $j=1, \ldots, n,$  we recall that $b_0(t_j)=1$ and for $i\geq 1:$ $b_i(t_j)=\prod_{k=0}^{i-1} (t_j-\theta^{q ^k}),$ $\ell_0=1$ and for $n\geq 1,$ $\ell_n =(\theta-\theta ^{q^n})\ell_{n-1}.$
Then, by (\ref{unit}), there exists $u_C(t_1, \ldots, t_n;z) \in A[t_1, \ldots, t_n,z]$ such that we have in $K[t_1, \ldots, t_n][[z]]:$
\begin{eqnarray}\label{eqS3-3}\sum_{k\geq 0}\sum_{a\in A_{+,k}}\frac{a(t_1)\cdots a(t_n)}{a} z^k = \log_{n,z}( u_C(t_1, \ldots, t_n;z)).\end{eqnarray}
\noindent Let $\chi$ be a Dirichlet character of type $n$ and conductor $f.$ Recall that, by (\ref{eqS3-2}),  there exists $\eta_1, \ldots, \eta_n\in \overline{\mathbb F_q}$ such that :
$$\sigma (g(\chi))=(\eta_1-\theta)\cdots  (\eta_n-\theta) g(\chi).$$
We set :
$$u_{\chi}(z) = g(\chi)u_C(t_1, \ldots, t_n; z)\mid_{t_1=\eta_1, \ldots , t_n=\eta_n}\in g(\chi)\mathbb  F_q(\eta_1, \ldots , \eta_n)[\theta][z].$$
Let $[\chi]= \{ \chi^{q^i}, i\geq 0\}.$ Observe that for $\psi \in [\chi],$ $\psi$ is of type $n$ and conductor $f.$ We set :
$$u_{[\chi]}(z)=\sum_{\psi \in [\chi]} u_{\psi}(z) \in A[\lambda_f][z],$$
and :
$$u^{(1)}_{[\chi]}(z)= \frac{d}{dz} u_{[\chi]}(z).$$
Finally, let $\Delta_f={\rm Gal}(K(\lambda_f)/K).$\par
Let $\log_{C,P}:  O\rightarrow \mathbb C_P$ be the $P$-adic logarithm  for the Carlitz module introduced in section \ref{PadicLseries}, where $O=\{ x\in \mathbb C_P, v_P(x)\geq 0\}$.

\begin{proposition}\label{PropositionS3-1}
Let $\chi$ be a Dirichlet character of type $n$ and conductor $f.$\par
\noindent 1) If $n\not \equiv 1\pmod{q-1},$ then :
$$L_P(1,\chi)g(\chi) =(1-P^{-1}\chi (P))\frac{1}{\mid \Delta_f\mid }\sum_{\mu \in \Delta_f}\chi^{-1}(\mu)\log_{C,P} (\mu(u_{[\chi]}(1)))                           .$$
2) If $n\geq q, $ $n\equiv 1\pmod{q-1},$  then :
$$   L_P^{(1)}(1, \chi) g(\chi)=(1-P^{-1}\chi (P))\frac{1}{\mid \Delta_f\mid }\sum_{\mu \in \Delta_f}\chi^{-1}(\mu)\log_{C,P} (\mu(u^{(1)}_{[\chi]}(1)))                       .$$
3) If $n=1,$ we have :
$$   L_P^{(1)}(1,\chi)g(\chi )=(1-P^{-1}\chi (P))\frac{1}{\mid \Delta_f\mid }\frac{1}{\theta- \chi(\theta)}\sum_{\mu \in \Delta_f}\chi^{-1}(\mu) \log_{C,P}( \mu (\lambda_f^q))                        .$$
\end{proposition}
\begin{proof}  Let's set :
$$  \log_{\sigma, z}=\sum_{k\geq 0}\frac{1}{\ell_k} z^k \sigma^k.                                                            $$
If we specialize $t_i$ in $\eta_i$ in the formula (\ref{eqS3-3}) and multiply by $g(\chi),$ we have the following  equality in $\mathbb F_q(\eta_1, \ldots, \eta_n)[\theta][\lambda_f][[z]]:$
$$\sum_{k\geq 0}\sum_{a\in A_{+,k}}\frac{\chi(a)}{a} g(\chi)z^k = \log_{\sigma,z} u_{\chi}(z).$$
Now, set :
$$\log_{C,z}=\sum_{k\geq 0}\frac{1}{\ell_k} z^k \tau^k.$$
We get the following equality in $K[\lambda_f][[z]]$ ($\tau$ acts trivially on $z$):
$$\sum_{\psi \in  [\chi]} \sum_{k\geq 0}\sum_{a\in A_{+,k}}\frac{\psi(a)}{a} g(\psi)z^k= \log_{C,z} u_{[\chi]}(z).$$
For $\mu \in \Delta_f,$ we have :
$\mu (g(\chi))=\chi(\mu)g(\chi)$, we therefore  deduce that :
$$\mid \Delta_f\mid \sum_{k\geq 0}\sum_{a\in A_{+,k}}\frac{\chi(a)}{a} g(\chi)z^k= \sum_{\mu \in \Delta_f}\chi^{-1}(\mu)\log_{C,z} \mu(u_{[\chi]}(z)).$$
Thus:
$$\mid \Delta_f\mid  \, g(\chi)  \sum_{k\geq 0}\sum_{a\in A_{+,k}}\frac{\chi(a)}{a} z^k =\sum_{\mu \in \Delta_f}\chi^{-1}(\mu)\log_{C,z} \mu(u_{[\chi]}(z)).$$
And finally :
$$g(\chi ) \sum_{k\geq 0}\sum_{\substack{a\in A_{+,k} \\ a\not \equiv 0\mod{P}}}\frac{\chi(a)}{a} z^k = (1-P^{-1}\chi (P)z^d)\frac{1}{\mid \Delta_f\mid }\sum_{\mu \in \Delta_f}\chi^{-1}(\mu)\log_{C,z} \mu(u_{[\chi]}(z)).$$ 
Recall that :
$$\theta \log_{C, z}= \log_{C,z} (z\tau +\theta).$$
Now, since $u_{[\chi]}(1)$ can be a $P$-adic unit, using the functional equation of $\log_{C,z}$ if necessary, we get if $\chi$ is even (i.e. $n\not \equiv 1\pmod{q-1}$) :
$$L_P(1,\chi)g(\chi) =(1-P^{-1}\chi (P))\frac{1}{\mid \Delta_f\mid }\sum_{\mu \in \Delta_f}\chi^{-1}(\mu)\log_{C,P} (\mu(u_{[\chi]}(1))).$$ 
Let's treat the case where  $\chi$ is odd (i.e. $n\equiv 1\pmod{q-1}$). If $n\geq q,$ by Lemma \ref{LemmaS2-3}, we have  $u_{[\chi]}(1)=0.$  Using again the same technique, we get in this case:
$$L_P^{(1)}(1, \chi) g(\chi)=(1-P^{-1}\chi (P))\frac{1}{\mid \Delta_f\mid }\sum_{\mu \in \Delta_f}\chi^{-1}(\mu)\log_{C,P} (\mu(u^{(1)}_{[\chi]}(1))).$$ 
Thus, it remains to treat the case $n=1$ . First, by formula (\ref{eqS3-1})   and Lemma \ref{LemmaS2-2}:
$$u_{[\chi]}(z)=\lambda_f.$$
Therefore, in this case :
$$g(\chi) \sum_{k\geq 0}\sum_{\substack{a\in A_{+,k}\\ a\not \equiv 0\pmod{P}}}\frac{\chi(a)}{a} z^k = (1-P^{-1}\chi (P)z^d)\frac{1}{\mid \Delta_f\mid }\sum_{\mu \in \Delta_f}\chi^{-1}(\mu)\log_{C,z} \mu(\lambda_f).$$ 
Now, select $\lambda \in \mathbb F_q$ such that $f$ is relatively prime to $\theta +\lambda.$ Then :
$$\sum_{\mu \in \Delta_f}\chi^{-1}(\mu)\log_{C,z} \mu(\lambda_f) = \chi ^{-1}(\theta+\lambda)
\sum_{\mu \in \Delta_f}\chi^{-1}(\mu)\log_{C,z}C _{\theta+\lambda}(\mu(\lambda_f)).$$
Thus $g(\chi) (\chi(\theta+\lambda)-(\theta +\lambda))\sum_{k\geq 0}\sum_{a\in A_{+,k}, a\not \equiv 0\pmod{P}}\frac{\chi(a)}{a} z^k$ is equal to:
$$ (1-P^{-1}\chi (P)z^d)\frac{1}{\mid \Delta_f\mid }\sum_{\mu \in \Delta_f}\chi^{-1}(\mu)(\log_{C,z} C_{\theta+\lambda }(\mu(\lambda_f))- (\theta +\lambda) \log_{C,z} \mu(\lambda_f)).$$ 
But $n=1,$ thus  $\chi (\theta +\lambda)= \chi (\theta) +\lambda.$ Thus, 
$$g(\chi)(\chi(\theta)-\theta)\sum_{k\geq 0}\sum_{a\in A_{+,k}, a\not \equiv 0\pmod{P}}\frac{\chi(a)}{a} z^k$$
 is equal to :
$$ (1-P^{-1}\chi (P)z^d)\frac{1}{\mid \Delta_f\mid }\sum_{\mu \in \Delta_f}\chi^{-1}(\mu)(\log_{C,z} C_{\theta }(\mu(\lambda_f))- \theta  \log_{C,z} \mu(\lambda_f)).$$ 
Set :
$$\log^{(1)}_{C,z}= \sum_{k\geq 0}\frac{k}{\ell_k} z^{k-1} \tau^k.$$
Then  :
$$\log^{(1)}_{C,z} (z\tau +\theta) -\theta \log^{(1)}_{C,z}= -\log_{C,z} \tau.$$
Therefore, for $n=1,$ we get :
$$g(\chi) L_P^{(1)}(1,\chi)=(1-P^{-1}\chi (P))\frac{1}{\mid \Delta_f\mid }\frac{1}{\theta- \chi(\theta)}\sum_{\mu \in \Delta_f}\chi^{-1}(\mu) \log_{C,P}( \mu (\lambda_f^q)) .$$
\end{proof}

${}$\par
\noindent {\sl Proof of Theorem \ref{TheoremS3-1} :}${}$\par

We first treat the case where $\chi$ is even (i.e. $n\not \equiv 1\pmod{q-1}$). Note that  by Proposition \ref{PropositionS2-3} and Lemma \ref{LemmaS2-2}, we have :
$$u_{[\chi]}(1)\not =0.$$
Since $\chi$ is even this implies that $u_{[\chi]}(1)$ is not a torsion point for the Carlitz module, in particular :
 $$\log_{C,P}(u_{[\chi]}(1))\not =0.$$
 But, by the proof of Proposition \ref{PropositionS3-1}, we have  :
 $$\log_{C,P}(u_{[\chi]}(1))=\sum_{\psi \in [\chi]}(1-\frac{\psi(P)}{P})^{-1}L_P(1, \psi) g(\psi) .$$
 This implies that there exists $\psi \in [\chi]$ such that :
 $$L_P(1, \psi)\not =0.$$
 Thus, we have to prove that if $L_P(1, \chi) \not =0,$ then $L_P(1, \chi^{\frac{1}{q}})\not =0.$ We are going to prove it by performing a change of variable.\par 
 
 \noindent Set ${K^{(q)}}=\mathbb F_q(\theta^q), $ observe that $K/K^{(q)} $ is totally ramified at every place of $K^{(q)}.$ Let ${C^{(q)}}$ be the Carlitz module for $A^{(q)}:=A^q=\mathbb F_q[\theta^q], $ i.e. ${C^{(q)}}_{\theta^q}= \tau +\theta ^q.$ Let $\lambda_{f(\theta^q)}^{(q)}=\exp_{{C^{(q)}}} (\frac{\widetilde{\pi }^{(q)}}{f(\theta^q)})$ (this is  also equal to $\lambda_f^q$). Then :
 $$K(\lambda_f) =K(\lambda_{f(\theta^q)}^{(q)}).$$
 Furthermore, the restriction map induces an isomorphism of groups :
 $$\Delta_f \simeq \Delta_{f(\theta^q)}^{(q)}.$$
 This implies that :
 $$\forall \sigma \in \Delta_f, \chi (\sigma) =\chi(\sigma^{(q)}).$$
 Let $(., K(\lambda_f)/K)$ be the Artin map, then for a monic irreducible element $Q$ of $A$ prime to $f:$
 $$ (Q, K(\lambda_f)/K)\mid_{K^{(q)}(\lambda_{f(\theta^q)}^{(q)})}= (Q(\theta^q),K^{(q)}(\lambda_{f(\theta^q)}^{(q)})/K^{(q)}).$$
 Observe that:
 $$L_P(1, \chi)=\sum_{n\geq 0}\sum_{a\in A_{+,n}, a\not \equiv 0\pmod{P},  a {\rm \, prime \, to \,} f}\frac{\chi((a, K(\lambda_f)/K))}{a} .$$
 Therefore $L_P(1, \chi)\not =0$ implies that $L_P^{(q)}(1, \chi)\not =0,$
 where :
 $$L_P^{(q)}(1, \chi)=\sum_{n\geq 0}\sum_{b\in A^{(q)}_{+,n}, b\not \equiv 0\pmod{P(\theta^q)},  b {\rm \, prime \, to \,} f(\theta^q)}\frac{\chi((b, K^{(q)}( \lambda_{f(\theta^q)}^{(q)})/K^{(q)}))}{b} .$$
 But :
 $$L_P^{(q)}(1, \chi)= (L_P(1, \chi^{\frac{1}{q}}))^q.$$
 Thus :
 $$L_P(1,\chi^{\frac{1}{q}})\not =0.$$\par
Now, we turn to the case where $\chi$ is odd of type $n \geq q.$ Again, by Proposition \ref{PropositionS2-3}, we get :
$$u^{(1)}_{[\chi]}(1) \not =0.$$
Since $\chi$ is not of type $1,$ by formula (\ref{eqS3-1}),  $u^{(1)}_{[\chi]}(1)$ is not a torsion point,  in particular:
 $$\log_{C,P}(u^{(1)}_{[\chi]}(1))\not =0.$$
Now by the proof of Theorem \ref{PropositionS3-1}:
$$\log_{C,P}(u^{(1)}_{[\chi]}(1))= \sum_{\psi \in [\chi]} (1-\frac{\psi (P)}{P})^{-1}L_P^{(1)}(1, \psi) g(\psi) .$$
We can conclude as in the even case.
It remains to treat the case $n=1.$ Observe that $\lambda_f^q$ is not a torsion point for the Carlitz module. Furthermore by Proposition  \ref{PropositionS3-1}, we have :
$$     \log_{C,P}(\lambda_f^q)= \sum_{\psi\in [\chi]}(\theta-\psi(\theta))(1-\frac{\psi(P)}{P})^{-1} L_P^{(1)}(1, \psi)g(\psi).$$                                                      
Again, we can conclude as in the even case. This achieves the proof of Theorem \ref{TheoremS3-1}.
\subsection{The class module of the $P$th cyclotomic function field}\label{ClassmoduleP}${}$\par
Let $P$ be a monic irreducible element in $A$ of degree d. Recall that $\lambda_P=\exp_C(\frac{\widetilde{\pi}}{P}).$ Let $L=K(\lambda_P)$,
recall that $L/K$ is a finite abelian extension of degree $q^d-1$ which is unramified outside $P$ and $\infty$ (\cite{ROS}, Proposition 12.7). Let's set $\Delta=\Delta_P={\rm Gal}(L/K).$  Let $O_L$ be the integral closure of $A$ in $L,$ one can show that $O_L=A[\lambda_P]$ (\cite{ROS}, Proposition 12.9) and that $PO_L = \lambda_P^{q^d-1} O_L$ (\cite{ROS}, Proposition 12.7). Recall that Taelman's class module associated to $L/K$ and the Carlitz module is defined by:
$$H=H(C/O_L)=\frac{L_\infty}{O_L+\exp_C(L_\infty)},$$
where $L_\infty= L\otimes_KK_\infty.$ We observe that $H$ is a finite $A[\Delta]$-module.\par
Let $A_P$ be the $P$-adic completion of $A,$ and set:
$$X=H\otimes_AA_P.$$
Then $X$ is a finite $A_P[\Delta]$-module. For $\chi \in \widehat{\Delta}:={\rm Hom}(\Delta, A_P^\times),$ we set:
$$X(\chi)=e_\chi X,$$
where $e_\chi=\frac{1}{\mid \Delta\mid} \sum_{\delta \in \Delta} \chi(\delta) \delta^{-1} \in \mathbb F_{q^d}[\Delta].$\par
If $\chi$ is an even Dirichlet character, then by \cite{ANG&TAE1}, Theorem 9.12, we have:
$$L_P(1, \chi) X(\chi)=\{0\}.$$
Furthermore, by \cite{ANG&TAE2}, Theorem 1, and \cite{GOS&SIN} (see also \cite{ABBL}), if the character $\chi$ is odd and $v_P(L_P^{(1)}(1, \chi))=0,$ then $X(\chi)$ is a cyclic $A_P$-module. In fact, we have:
\begin{theorem}\label{TheoremS3-2} If $\chi$ is an odd Dirichlet character of prime conductor $P$, then $L_P^{(1)}(1, \chi) X(\chi)$ is a cyclic $A_P$-module.
\end{theorem}
\begin{proof} ${}$\par
If $\chi$ is of type 1, by \cite{ANG&TAE1}, Theorem 8.7, we have   that $X(\chi)=\{0\}.$ Thus we can assume that $\chi$ is of type $n\geq q,$ $n\equiv 1\pmod{q-1}.$ Let $\eta_1, \ldots, \eta_n \in \mathbb F_{q^d}$ such that:
$$\forall a\in A, \chi (a)= a(\eta_1)\cdots a(\eta_n).$$\par

Set $\exp_C^{(1)}= \sum_{j\geq 1} \frac{j}{D_j}\tau^j.$ As in Corollary \ref{CorollaryS2-1}, combining the isomorphism of $A[\Delta]$-modules given by   Proposition \ref{propositionSt3} with the evaluation at $z=1$ we obtain a morphism of $A[\Delta]$-modules induced by $\exp_C^{(1)}:$
\begin{equation}\label{equapsi}
\psi:\frac{U(C/O_L)}{U_{St}(C/O_L)}\rightarrow H(C/O_L).
\end{equation}
We will first investigate $[e_{[\chi]}{\rm Coker}\,  \psi]_{A[\Delta]}.$\par
${}$\par

 We will intensively use the results obtained in section \ref{nodd}. Let $\varphi$ be the canonical deformation of the Carlitz module over $\mathbb T_n(K_\infty)$ (see section \ref{Thecaseofthecarlitzmodule}). Let $ev_\chi: \mathbb T_n(K_\infty)\rightarrow \mathbb F_q(\eta_1, \ldots, \eta_n)((\frac{1}{\theta}))$ be the surjective map given by $ev_\chi (f)= f\mid_{t_1=\eta_1, \ldots, t_n=\eta_n}.$
Set:
$$\Omega=L_\infty\otimes_{\mathbb F_q}\mathbb F_q(\eta_1, \ldots, \eta_n).$$
Let $\sigma =\tau\otimes 1: \Omega\rightarrow \Omega.$ Then:
$$e_\chi \Omega = g(\chi) \mathbb F_q(\eta_1, \ldots, \eta_n)((\frac{1}{\theta})).$$
Then, by formula (\ref{eqS3-2}), we get:
$$\forall f\in \mathbb T_n(K_\infty), 
(\sigma +\theta) (g(\chi) ev_\chi(f))= g(\chi)ev_\chi(\varphi_\theta(f)).$$
$$\forall f\in \mathbb T_n(K_\infty), 
\exp_{C,\sigma} (g(\chi) ev_\chi(f))= g(\chi)ev_\chi(\exp_\varphi(f)),$$
where:
$$\exp_{C,\sigma} = \sum_{j\geq 0} \frac{1}{D_j}\sigma^j.$$
Now, by Theorem \ref{TheoremSt2}, we have:
$$e_\chi U_{St}(C/O_L)\otimes_{\mathbb F_q} \mathbb F_q(\eta_1, \ldots, \eta_n)= A[\eta_1, \ldots , \eta_n] g(\chi)L(1, \chi),$$
where $L(1, \chi)=\sum_{a\in A_+} \frac{\chi(a)}{a} \in \mathbb F_q(\eta_1, \ldots, \eta_n)((\frac{1}{\theta}))^\times.$ Thus:
$$e_\chi U_{St}(C/O_L)\otimes_{\mathbb F_q} \mathbb F_q(\eta_1, \ldots, \eta_n)=g(\chi) ev_\chi(U_{St}(\varphi/A[t_1, \ldots, t_n])).$$\par

Select $v$ a place of $L$ above $\infty,$ and let $\widetilde{\pi}_v= (0, \ldots, O, \widetilde{\pi}, 0, \ldots, 0) \in L_{\infty}=\prod_{w\in S_\infty (L)} L_w$ which has zero at all its components except at the component $v$. Let's set:
$$\widetilde{\pi}_\chi = e_\chi \widetilde{\pi}_v\in e_\chi \Omega.$$
Then by \cite{ANG&TAE1}, Proposition 8.4, we have:
$${\rm Ker}(\exp_{C, \sigma}\mid_{e_\chi \Omega})= \widetilde{\pi}_\chi A[\eta_1, \ldots, \eta_n].$$
Since $\chi$ is odd and of type $>1,$ we deduce that:
$$e_\chi U(C/O_L)\otimes_{\mathbb F_q}\mathbb F_q(\eta_1, \ldots, \eta_n)= \widetilde{\pi}_\chi A[\eta_1, \ldots, \eta_n].$$
Since $ev_\chi (\mathbb B(t_1, \ldots, t_n))= [e_\chi H(C/O_L)\otimes _{\mathbb F_q} \mathbb F_q(\eta_1, \ldots, \eta_n)]_{A[\eta_1, \ldots, \eta_n]}$ (\cite{APTR}, Corollary 9.3),  by Lemma \ref{LemmaS2-1} and  \cite{ANG&TAE1}, Theorem A, Theorem B and Proposition 8.5,  we get:
$$e_\chi U(C/O_L)\otimes_{\mathbb F_q}\mathbb F_q(\eta_1, \ldots, \eta_n)=g(\chi) ev_\chi(U(\varphi/A[t_1, \ldots, t_n])).$$\par
Set:
$$\exp_{C, \sigma}^{(1)}=\sum_{j\geq 1} \frac{j}{D_j}\sigma^j.$$
As in the case of section \ref{nodd}, $\exp_{C, \sigma}^{(1)}$ induces a map:
$$ \frac{e_\chi \Omega}{g(\chi)\mathcal N}\rightarrow \frac{e_\chi \Omega}{g(\chi)\mathcal N},$$
where 
$$\mathcal N=\{ x\in \mathbb F_q(\eta_1, \ldots, \eta_n)((\frac{1}{\theta})), v_\infty(x)>\frac{n-q}{q-1}\}$$
(note that $e_\chi \Omega = \widetilde{\pi}A[\eta_1, \ldots, \eta_n]\oplus g(\chi)\mathcal N$). Now, using $ev_\chi,$  the proof of Proposition \ref{PropositionS2-4} works in this case and the above map is injective.\par
 Observe that:
$$\exp_{C, \sigma}(e_\chi \Omega)= g(\chi) \mathcal N,$$
$$e_\chi \Omega= \widetilde{\pi}_\chi A[\eta_1, \ldots, \eta_n]\oplus g(\chi)\mathcal N.$$
Set $V=\{ x\in e_\chi \Omega, \exp_{C, \sigma}^{(1)}(x)\in g(\chi)A[\eta_1, \ldots, \eta_n]+ g(\chi)\mathcal N\}.$ Then, we have an injective map induced by $\exp_{C, \sigma}^{(1)}:$
$$\frac{V}{g(\chi)\mathcal N}\hookrightarrow \frac{g(\chi)A[\eta_1, \ldots, \eta_n]\oplus g(\chi)\mathcal N}{g(\chi)\mathcal N}.$$
Note that the map $\psi$ defined in \eqref{equapsi} induces a map of $A[\eta_1, \ldots, \eta_n]$-modules:
$$\psi_\chi: e_\chi \frac{U(C/O_L)}{U_{St}(C/O_L)}\otimes_{\mathbb F_q} \mathbb  F_q(\eta_1, \ldots, \eta_n)\rightarrow\frac{ e_\chi \Omega}{g(\chi) A[\eta_1, \ldots, \eta_n]\oplus g(\chi)\mathcal N}.$$
 If we consider the proof of Theorem \ref{TheoremS2-2}, there exists  a monic polynomial $G\in A[\eta_1, \ldots, \eta_n]$  such that:
 $$e_\chi {U_{St}(C/O_L)}\otimes_{\mathbb F_q} \mathbb  F_q(\eta_1, \ldots, \eta_n)\subset G\widetilde{\pi}_\chi A[\eta_1, \ldots, \eta_n],$$
 $${\rm Ker} \, \psi_\chi= \frac{G\widetilde{\pi}_\chi A[\eta_1, \ldots, \eta_n]}{e_\chi {U_{St}(C/O_L)}\otimes_{\mathbb F_q} \mathbb  F_q(\eta_1, \ldots, \eta_n)},$$
$$G\widetilde{\pi}_\chi A[\eta_1, \ldots, \eta_n]+g(\chi)\mathcal N=V.$$
Let $e_\chi \mathcal M= g(\chi)ev_\chi (\mathcal M(\varphi/A[t_1, \ldots, t_n])).$ Then we have a natural injective map of $A[\eta_1, \ldots, \eta_n]$-modules (induced by $\exp_{C, \sigma}^{(1)}$):

 \begin{tikzcd}\displaystyle
 \frac{G\widetilde{\pi}_\chi A[\eta_1, \ldots, \eta_n]\oplus g(\chi)\mathcal N}{e_\chi U_{St}(C/O_L)\otimes_{\mathbb F_q} \mathbb F_q(\eta_1, \ldots, \eta_n)\oplus g(\chi) \mathcal N}\arrow[hookrightarrow]{r} & \displaystyle \frac{g(\chi)A[\eta_1, \ldots , \eta_n]\oplus g(\chi)\mathcal N}{e_\chi \mathcal M\oplus g(\chi)\mathcal N} \arrow[hookrightarrow]{d}
 \\
 & \displaystyle \frac{g(\chi)A[\eta_1, \ldots , \eta_n]}{e_\chi \mathcal M}
 \end{tikzcd}
 
 \noindent thus:
$${\rm Ker} \, \psi_\chi \hookrightarrow \frac{g(\chi)A[\eta_1, \ldots , \eta_n]}{e_\chi \mathcal M}.$$\par
Let $\mathcal M$ be the $A[\Delta]$-module via $C$ generated by $u_{[\chi]}^{(1)}(1)$ (see section \ref{P-adicDirichlet-GossL-series}) and let $\sqrt{\mathcal M}=\{ x\in O_L, \exists a\in A\setminus\{0\}, C_a(x)\in \mathcal{M}\}.$ Let's consider the map induced by $\psi:$
$$\psi_{[\chi]}: e_{[\chi]}\frac{U(C/O_L)}{U_{St}(C/O_L)}\rightarrow e_{[\chi]}H(C/O_L).$$
Then, by the above discussion, we have an injective morphism of $A[\Delta]$-modules:
$${\rm Ker}\,  \psi_{[\chi]}\hookrightarrow \frac{e_{[\chi]} O_L}{\mathcal M}.$$
But ${\rm Ker}\,  \psi_{[\chi]}$ is a finite $A$-module, therefore we have an injective morphism of $A[\Delta]$-modules:
$${\rm Ker}\,  \psi_{[\chi]}\hookrightarrow \frac{\sqrt{\mathcal M}}{\mathcal M}.$$ Note that $\mathcal M$ is a free $A$-module of rank $\mid [\chi]\mid.$ Thus,  by \cite{POO}, $\frac{\sqrt{\mathcal M}}{\mathcal M}$ is a finite $A[\Delta]$-module. 
By Corollary \ref{corollaryStark2}, we have:
$$[e_{[\chi]}\frac{U(C/O_L)}{U_{St}(C/O_L)}]_{A[\Delta]}=[e_{[\chi]}H(C/O_L)]_{A[\Delta]}.$$
Thus:
$$[{\rm Ker}\,  \psi_{[\chi]}]_{A[\Delta]}= [{\rm Coker}\,  \psi_{[\chi]}]_{A[\Delta]}.$$
Since $\frac{U(C/O_L)}{U_{St}(C/O_L)}$ is $A[\Delta]$-cyclic, we have $[{\rm Ker}\,  \psi_{[\chi]}]_{A[\Delta]} e_{[\chi]}H(C/O_L)$ is $A[\Delta]$-cyclic. But $[{\rm Ker}\,  \psi_{[\chi]}]_{A[\Delta]}$ divides $[\frac{\sqrt{\mathcal M}}{\mathcal M}]_{A[\Delta]}$ in $A[\Delta],$ thus we deduce that the module $[\frac{\sqrt{\mathcal M}}{\mathcal M}]_{A[\Delta]}e_{[\chi]}H(C/O_L)$ is $A[\Delta]$-cyclic. This in particular implies that :
$$ ({\rm Fitt}_{A_P} e_\chi (\frac{\sqrt{\mathcal M}}{\mathcal M}\otimes_AA_P)) X(\chi)\, {\rm is\,  } A_P{\rm -cyclic},$$
where  for $M$ a finite $A_P$-module, ${\rm Fitt}_{A_P} M$ denotes its  Fitting ideal. \par ${}$\par
Let $F$ be the $P$-adic completion of $L.$ Let $O_F$ be the valuation ring of $F$ and   let $\frak P=\{ x\in F, v_P(x)>0\}.$ Recall  from section \ref{PadicLseries} that $C$ extends into a morphism of $\mathbb F_q$-algebras $C: A_P\rightarrow A_P\{\{ \tau \}\}.$ We denote by $C(\frak P)$ the $\mathbb F_q$-vector space $\frak P$ viewed as an $A_P$-module via $C.$ Then $\log_{C,P}$ induces an isomorphism of $A_P[\Delta]$-modules:
$$\log_{C,P}: C(\frak P^2)\simeq \frak P^2.$$
We also observe that, since $\chi$ is odd of type $>1,$  $g(\chi^{q^j})A_P=e_{\chi^{q^j}}O_F \subset \frak P^2$ for all $j\geq 0$ (this follows for example from \cite{THA}, Theorem III). In particular:
$$\mathcal M, \sqrt{\mathcal M} \subset C(\frak P^2).$$
 Let $\overline{\mathcal M}$ be the $P$-adic closure of $\mathcal M$ in $C(\frak P^2).$ Then clearly $\overline{\mathcal M}$ is the $A_P[\Delta]$-module generated by $u_{[\chi]}^{(1)}(1).$ Now, by Proposition \ref{PropositionS3-1},  we have:
$$\log_{C,P}(e_\chi \overline{\mathcal M})= g(\chi) L_P^{(1)}(1, \chi) A_P.$$
 By Theorem \ref{TheoremS3-1}, we have an isomorphism of $A_P$-modules:
$$e_\chi (\mathcal M\otimes_A A_P) \simeq e_\chi \overline{\mathcal M}.$$
Furthermore:
$${\rm Fitt}_{A_P}\frac{e_\chi C(\frak P^2)}{e_\chi \overline{\mathcal M}}= L_P^{(1)}(1, \chi) A_P.$$
Let $\overline{\sqrt{\mathcal M}}$ be the $P$-adic closure of $\sqrt{\mathcal M}$ in $C(\frak P^2).$ Then, we have an isomorphism of $A_P$-modules:
$$e_\chi(\frac{\sqrt{\mathcal M}}{\mathcal M}\otimes_A A_P) \simeq e_\chi \frac{\overline{\sqrt{\mathcal M}}}{\overline{\mathcal M}}.$$
Thus:
$$L_P^{(1)}(1, \chi)A_P\subset {\rm Fitt}_{A_P} e_\chi (\frac{\sqrt{\mathcal M}}{\mathcal M}\otimes_AA_P) .$$
The Theorem follows.
\end{proof}
\begin{remark}\label{examples} Let $P$ be a monic irreducible element in $A$ of degree $d.$ Note that, in general, for $\chi$ odd of conductor $P,$ we have:
$$L_P^{(1)}(1, \chi) X(\chi)\not =\{0\}.$$
Indeed, let $\chi_P$ be the $P$-adic Teichm\"uller character (observe that this is a Dirichlet character of type $1$), i.e. :
$$\forall a\in A, \chi_P(a)\equiv a\pmod{PA_P}.$$
Let $n\in \{ 2, \ldots, q^d-2\},$ $ n\equiv 1\pmod{q-1}.$ Then, by \cite{TAE3}, we have:
$$X(\chi_P^n)\not =\{ 0\} \, {\rm if \, and \, only\, if\,} BC(q^d-n)\equiv 0\pmod{P},$$
where $BC(q^d-n)$ denotes the $(q^d-n)$th Bernoulli-Carlitz number (see \cite{GOS}, paragraph 9.2). Let $\beta(n-1)$ be the $(n-1)$th Bernoulli-Goss number (see \cite{GOS3}, \cite{GOS&SIN}, \cite{IRE&SMA}), then:
$$L_P^{(1)}(1, \chi_P^n)\equiv \beta(n-1)\pmod{P}.$$
For example, for $q=3,$ $P=\theta^3-\theta-1,$ we have (by \cite{GOS}, page 354):
$$BC(10)=\frac{2\theta^6+2\theta^4+2\theta^2+1}{\theta^3+2\theta}\equiv 0\pmod{P},$$
and a direct computation shows:
$$\beta(16)= \theta^{30} + 2\theta^{28} + 2\theta^4 + \theta^2 + 1 \equiv 1 \pmod{P}.$$
This implies:
$$L_P^{(1)}(1, \chi_P^{17})X(\chi_P^{17})\not =\{0\}.$$
\end{remark}


\begin{thebibliography}{29}

 \bibitem{AND} G. Anderson, Log-algebraicity of twisted $A$-harmonic series and special values of $L$-series in characteristic $p$, {\it  Journal of  Number Theory} {\bf 60} (1996), 165-209.
\bibitem{AND&THA} G. Anderson, D. Thakur, Tensor powers of the Carlitz module and zeta values, {\it Annals of Mathematics} {\bf 132} (1990), 159-191.
 \bibitem{ABBL} B. Angl\`es, A. Bandini, F. Bars, I. Longhi, Iwasawa Main Conjecture for the Carlitz cyclotomic extension and applications, arXiv : 1412.5957 .
 \bibitem{ANG&PEL1} B. Angl\`es, F. Pellarin, Universal Gauss-Thakur sums and $L$-series,  {\it Inventiones  mathematicae} {\bf 200 } (2015), 653-669.
 \bibitem{APTR} B. Angl\`es, F. Pellarin,  F. Tavares Ribeiro,  Arithmetic of positive characteristic $L$-series values in Tate algebras, to appear in {\it Compositio Mathematica},   arXiv:1402.0120.
 \bibitem{APTR2} B. Angl\`es, F. Pellarin, F. Tavares Ribeiro, Anderson-Stark units for $\mathbb F_q[\theta],$ arXiv: 1501.06804.
 \bibitem{ANG&TAE1} B. Angl\`es, L. Taelman, Arithmetic of characteristic $p$ special $L$-values,  {\it Proceedings of the  London Mathematical  Society} {\bf 110} (2015), 1000-1032.
 \bibitem{ANG&TAE2} B. Angl\`es, L. Taelman, The Spiegelungssatz for the Carlitz module, {\it Journal of Number Theory} {\bf 133} (2013), 2139-2142.  
 
 \bibitem{DEM} F. Demeslay, A class formula for $L$-series in positive characteristic,  arXiv : 412.3704v1 .
 \bibitem{FAN} J. Fang, Special $L$-values of abelian $t$-modules, {\it Journal of Number Theory} {\bf 147} (2015), 300-325. 
 \bibitem{FAN2} J. Fang, Equivariant $L$-values for abelian $t$-modules, arXiv: 1503.07243v1 .
 
\bibitem{PUT} J. Fresnel, M. van der Put, {\it Rigid Analytic Geometry and Its Applications}, Birkh\"auser, 2004.
  \bibitem{GOS} D. Goss, {\it Basic Structures of Function Field Arithmetic}, Springer, Berlin, 1996.
 \bibitem{GOS3} D. Goss, Units and class groups in the arithmetic of function fields, {\it Bulletin of the American Mathematical Society} {\bf 13} (1985), 131-132.
 \bibitem{GOS4} D. Goss, $v$-adic Zeta Functions, $L$-series and Measures for Function Fields, {\it Inventiones mathematicae} {\bf 55} (1979), 107-116.
 \bibitem{GOS&SIN} D. Goss, W. Sinnott, Class groups of function fields, {\it Duke Mathematical Journal} {\bf 52} (
 1985), 507-516.
 \bibitem{GRE} R. Greenberg, Iwasawa theory-past and present, {\it Advanced  Studies in Pure Mathematics} {\bf 30} (2001), 335-385.  
 \bibitem{IRE&SMA} K. Ireland, D. Small, A note on Bernoulli-Goss polynomials, {\it Canadian Mathematical  Bulletin} {\bf 27} (1984),179-184.
 \bibitem{LAN} S. Lang, {\it Cyclotomic Field I and II}, Springer-Verlag, 1990.
 \bibitem{LEO} H.W. Leopoldt, Zur Strukture der l-Klassengruppe galoischer Zahlkorper, {\it Journal f\"ur die  reine und angewandte  Mathematik} {\bf 199} (1958), 165-174.
 \bibitem{MAZ} B. Mazur, A. Wiles, Class fields of abelian extensions of $\mathbb Q,$ {\it Inventiones mathematicae} {\bf 76} (1984), 179-330. 
  \bibitem{POO} B. Poonen, Local height functions and the Mordell-Weil theorem for Drinfeld modules, 
  {\it Compositio Mathematica} {\bf 97} (1995), 349-368.

\bibitem{PEL} F. Pellarin, Values of certain $L$-series in positive characteristic, {\it Annals of Mathematics} {\bf 176} (2012), 1-39.
  \bibitem{NOE} E. Noether, Normalbasis bei K\"orpern ohne h\"ohere Verzweigung, {\it Journal f\"ur die reine und angewandte Mathematik} {\bf 167} (1931), 147-152.
  \bibitem{ROS} M. Rosen, {\it Number Theory in Function Fields}, Springer, 2002.
  \bibitem{TAE1} L. Taelman, A Dirichlet unit theorem for Drinfeld modules, {\it Mathematische Annalen} {\bf 348} (2010), 899-907.
  \bibitem{TAE2} L. Taelman, Special $L$-values of Drinfeld modules, {\it Annals of Mathematics} {\bf 75} (2012), 369-391.
  \bibitem{TAE3} L. Taelman, A Herbrand-Ribet theorem for function fields, {\it Inventiones  mathematicae} {\bf 188} (2012), 253-275.
  \bibitem{THA} D. Thakur, Gauss sums for $\mathbb F_q[t],$  
  {\it Inventiones  mathematicae}   {\bf 94} (1988), 105-112 .
  \bibitem{WAS} L. C. Washington, {\it Introduction to Cyclotomic Fields}, Second Edition, Springer, 1997.
  
 \end{thebibliography}
  \end{document}